\documentclass[11pt,reqno]{amsart}

\usepackage[a4paper,top=25mm,bottom=29mm,left=28mm,right=28mm]{geometry}
\usepackage[T1]{fontenc}
\usepackage{lmodern}
\usepackage{amsmath,amssymb,amsfonts,amsthm,mathtools,bm,mathrsfs}
\usepackage{float,graphicx,subcaption,placeins,adjustbox}
\usepackage{booktabs,tabularx,array,longtable,multirow}
\usepackage{enumerate,enumitem,xcolor,microtype,siunitx,url}

\usepackage[hidelinks]{hyperref}

\usepackage[nameinlink,capitalise,noabbrev]{cleveref}
\graphicspath{{figures/}}

\theoremstyle{plain}
\newtheorem{theorem}{Theorem}[section]
\newtheorem*{theorem*}{Theorem}

\newtheorem{lemma}[theorem]{Lemma}
\newtheorem*{claim*}{Claim}
\newtheorem{proposition}[theorem]{Proposition}

\theoremstyle{definition}

\newtheorem{example}[theorem]{Example}
\theoremstyle{remark}

\newtheorem{remark}[theorem]{Remark}
\numberwithin{equation}{section}

\newcommand{\Pent}{\mathbb{P}}
\newcommand{\SwitchAlg}{\mathcal A}
\newcommand{\Vpure}{V}

\newcommand{\Imag}{\operatorname{Im}}

\newcommand{\Span}{\operatorname{span}}

\newcommand{\R}{\mathbb{R}}

\newcommand{\C}{\mathbb{C}}

\newcommand{\Tr}{\mathrm{Tr}}

\newcommand{\diag}{\operatorname{diag}}

\newcommand{\kg}{\kappa_\gamma}
\newcommand{\sg}{\sigma_\gamma}

\newcommand{\sn}{\operatorname{sn}}
\newcommand{\cn}{\operatorname{cn}}
\newcommand{\dn}{\operatorname{dn}}

\newcommand{\dd}{\,\mathrm d}

\begin{document}
\raggedbottom

\title[From Quaternions to Pentarnions II:
quadratic differential equations]{
From Quaternions to Pentarnions II:
quadratic differential equations}

\author{Alireza Poshtkohi}
\address{School of Physics, Engineering and Computer Science, University of Hertfordshire, Hatfield, Hertfordshire, United Kingdom}
\email{a.poshtkohi@herts.ac.uk}

\author{Susanne Pumpl\"un}
\address{School of Mathematical Sciences, University of Nottingham,
Nottingham NG7 2RD, United Kingdom}
\email{Susanne.Pumpluen@nottingham.ac.uk}

\keywords{
Nonassociative algebras, Pentarnion algebras, algebra-valued differential equations, quadratic differential equations, Jacobi elliptic solutions, invariant algebraic curves, Riccati equations, bifurcation theory, periodically varying parameters.}

\subjclass[2020]{Primary: 34A34; Secondary: 34C23, 34C25, 37G15, 17A45}

\date{\today}

\begin{abstract}We investigate differential equations defined through elements of a nonassociative five-dimensional real
algebra $\mathcal A_\gamma$, and show that every system of quadratic differential equations $\dot p=F(p)$ on $\mathbb R^5$ can be written as a differential equation over such a non-flexible quadratic algebra.
We then study  quadratic differential equations and their solutions over the whole family $\{\mathcal A_\gamma\}_{0\leq\gamma\leq1}$ of algebras while varying $\gamma$ across the parameter interval $[0,1]$, exploiting the fact that by definition, $\mathcal A_0$ contains Hamilton's quaternion algebra and that $\mathcal A_1=\mathbb{P}$ is the real Pentarnion algebra.

We construct explicit Jacobi elliptic solutions that
deform a quaternion-valued solution at $\gamma=0$ into a Pentarnion-valued one at $\gamma=1$, and an algebra-valued Riccati equation whose equilibrium $p=0$ loses stability at a critical value of $\gamma$. Numerical experiments indicate period doubling, coexisting stable periodic orbits and irregular dynamics.
\end{abstract}

\maketitle

\section*{Introduction}

Inspired by fundamental work  connecting nonassociative algebras with quadratic differential equations \cite{Markus1960, KinyonSagle1995,Walcher1991}, quaternion-valued differential
equations have been studied extensively, e.g. see
\cite{CamposMawhin2006,Wilczynski2009}. Most notably, every polynomial differential equation in $\R^4$ can be rewritten as a differential equation over Hamilton's quaternion algebra $\mathbb H$, and (most) quaternionic differential equations with coefficients in $\C$ behave very similarly to differential equations over $\C$ \cite{WalcherZhang2021}. 

Hamilton's quaternion algebra plays an important role in many areas of applied mathematics and physics.
From an algebraic point of view, Hamilton's quaternions are a non-commutative generalisation of the complex numbers (and together with $\R,$ both these algebras are the only finite dimensional real associative division algebras). Both are locally complex and quadratic real unital algebras.

It seems only  natural to further inquire about particular features of ordinary differential equations over other such locally complex and quadratic real $n$-dimensional algebras, viewed as $n$-dimensional real systems. 

Recently, such an algebra, the five-dimensional real Pentarnion algebra
$\Pent$, was used to study some $5$-dimensional real systems of ordinary differential equations defined via elements in  $\Pent$ \cite{PoshtkohiPumplunSubmitted}. In this paper, we will generalise our approach and investigate some quadratic differential equations over 
 a family of five-dimensional locally complex quadratic algebras $\mathcal{A}_\gamma$, $\gamma\in [0,1]$, defined for the first time in  \cite{PoshtkohiPumpluen2026a}.
 Every algebra  $\mathcal{A}_\gamma$
 has a canonical basis $1,i,j,k,\ell$; $\mathcal{A}_0$
contains Hamilton's quaternion algebra $\mathbb{H}$ as the subalgebra generated by $1,i,j,k$, while
$\mathcal{A}_1$ is the \emph{Pentarnion algebra} $\Pent$ from \cite{PoshtkohiPumplunSubmitted}.

Identifying $\mathbb R^5$ with $\mathcal A_\gamma$ via the canonical basis $1,i,j,k,\ell$, every quadratic differential equations on $\mathbb R^5$ can be written in terms of its multiplication $\star_\gamma$. Since $\mathcal A_\gamma$ is neither commutative nor associative, we have to take great care when talking about a ``polynomial'' over $\mathcal A_\gamma$, and the properties of what we could consider or define as polynomials differ vastly from the familiar commutative setting.

We recall the definition of algebras $\mathcal A_\gamma$  and present some first results in the preliminaries in \cref{sec:prelim}. Most importantly, we show that every
system of quadratic differential equations $\dot p=F(p)$ on $\mathbb R^5$ can be written as a differential equation over $\mathcal A_\gamma$
(Proposition \ref{prop:every-quadratic}). In general, such a differential equation over $\mathcal A_\gamma$  will have  many terms.

Since only the multiplication $\star_\gamma$ and the coefficients depend on $\gamma$, an equation over $\mathcal A_\gamma$ defines a one-parameter family of quadratic systems on $\mathbb R^5$. In the families studied here, the vector field preserves $\mathbb H\subset\mathcal A_0$ at $\gamma=0$ and the initial values lie in $\mathbb H$, so the
solutions are quaternion-valued; we study how they change as $\gamma$ moves to the Pentarnion algebra at $\gamma=1$.

 In \Cref{sec:exact-integrable-path}, we
present a family of polynomial differential equations together with explicit solutions in terms of Jacobi elliptic functions,
and describe how they depend on $\gamma$ (Proposition \ref{thm:flagship-orbit}, Theorem \ref{thm:integrable-geometric-transition}).
\cref{sec:riccati-dynamics} studies an algebra-valued Riccati equation with an analytic instability threshold  (\cref{thm:riccati-hopf})  and numerically observed periodic branches, period doubling, multistability and irregular regimes, for fixed $\gamma$ and for two periodic protocols $\gamma(t)$.

Results that rely on computer algebra are marked as such; the code is described in
Appendix~\ref{app:reproducibility}. Appendices~\ref{app:coordinate-matrices} and~\ref{app:numerics} give coordinate formulas and numerical
methods.

\section{Basic definitions, conventions and first results} \label{sec:prelim}

\subsection{The algebras $\mathcal{A}_\gamma$} Let us summarize the main results from \cite{PoshtkohiPumpluen2026a}; a more detailed exposition can be found in \cite[Sections 2 -- 4]{PoshtkohiPumpluen2026a}.
Let  $\gamma\in \R$ with $0\leq\gamma\leq1$, and define
\[
 \kappa=\kg=\frac{1}{\sqrt{1+\gamma^2}},
 \quad
 \sigma=\sg=\frac{\gamma}{\sqrt{1+\gamma^2}}.
\]
We note that $\kg^2+\sg^2=1$.

We define a five-dimensional real unital algebra $\mathcal{A}_\gamma=\mathcal{A}_\gamma(1,1,1,1)$
with \emph{canonical basis $1,i,j,k,\ell$} via the following Table~\ref{tab:mult} from \cite{PoshtkohiPumpluen2026a}. Multiplication in $\mathcal{A}_\gamma$ is denoted by  $\star_\gamma$,  or simply by juxtaposition when the context is clear.

\begin{table}[ht]
\centering
\renewcommand{\arraystretch}{1.3}
\begin{tabular}{c|cccc}
$\star_\gamma$ & $i$ & $j$ & $k$ & $\ell$\\ \hline
$i$ & $-1$ & $\kappa_\gamma k+\sigma_\gamma\ell$ & $-\kappa_\gamma j-\sigma_\gamma\ell$ & $-\tfrac{1}{\sqrt2}(j+k)$\\
$j$ & $-\kappa_\gamma k-\sigma_\gamma\ell$ & $-1$ & $\kappa_\gamma i+\sigma_\gamma\ell$ & $-\tfrac{1}{\sqrt2}(i+k)$\\
$k$ & $\kappa_\gamma j+\sigma_\gamma\ell$ & $-\kappa_\gamma i-\sigma_\gamma\ell$ & $-1$ & $-\tfrac{1}{\sqrt2}(i+j)$\\
$\ell$ & $\tfrac{1}{\sqrt2}(j+k)$ & $\tfrac{1}{\sqrt2}(i+k)$ & $\tfrac{1}{\sqrt2}(i+j)$ & $-1$
\end{tabular}
\caption{Multiplication in $\mathcal A_\gamma$; the entry in row $x$ and column $y$ is $x\star_\gamma y$.}
\label{tab:mult}
\end{table}

At $\gamma=0$ the subalgebra of $\mathcal{A}_0$ generated by
$1,i,j,k$ is the quaternion algebra $\mathbb H=(-1,-1)_\R$ with canonical basis $1,i,j,k$, and $\mathcal{A}_1=\mathbb P$ is the Pentarnion algebra. 

Every $p\in \mathcal{A}_\gamma $   can be written as
$$p=p_0+p_1i+p_2j+p_3k+p_4\ell,$$
for $p_0,p_1,p_2,p_3,p_4\in \R;$
 $p_0$ is called the \emph{scalar part}, and $ \Imag(p)=p_1i+p_2j+p_3k+p_4\ell
$ the \emph{pure part} of $p$. The map
\[
  p=p_0+p_1i+p_2j+p_3k+p_4\ell \mapsto  \overline{p}=p_0-p_1i-p_2j-p_3k-p_4\ell,
\]
is a scalar involution on $\mathcal{A}_\gamma$.
There exists a quadratic form  $N_{\mathcal{A}_\gamma}:\mathcal{A}_\gamma \to \R$,
\[
  N_{\mathcal{A}_\gamma }(p)=p\overline{p}=p_0^2+p_1^2+p_2^2+p_3^2+p_4^2,
\]
 the  \emph{norm} of $\mathcal{A}_\gamma$, and a linear form  $\Tr_{\mathcal{A}_\gamma}:\mathcal{A}_\gamma \to \R$,
\[
  \Tr_{\mathcal{A}_\gamma}(p)=p+\overline{p}=2p_0.
\]
 the  \emph{trace}  of $\mathcal{A}_\gamma$. We note that $N_{\mathcal{A}_\gamma}$ is isometric to the diagonal quadratic form $\langle 1,1,1,1,1\rangle$ 
   and that $N_{\mathcal{A}_\gamma}(p)=p\overline{p}=\overline{p}p$ \cite[p.~86]{McCrimmon1985}. 

The algebra $\mathcal{A}_\gamma$ is a unital quadratic algebra; every $p\in \mathcal{A}_\gamma$ satisfies the quadratic equation
$$ p^2-\Tr(p)p+N(p)1=0.$$
 This also means that $\mathcal{A}_\gamma$ is power associative.
The algebra $\mathcal{A}_\gamma$ is not flexible, however, since we have the associator 
\[
 [i+j,\ell,i+j]
 =2\left(\sqrt{\frac{1}{2}}+
          \sqrt{\frac{1}{2}}\right)1\neq0.
\]
Since every alternative algebra is flexible,
$\mathcal{A}_\gamma$ is not alternative either.

The algebras $\SwitchAlg_{\gamma}$ are pairwise non-isomorphic for different values of $\gamma$, more precisely, for all $\gamma,\delta\geq 0$, we have $\SwitchAlg_{\gamma}\cong \SwitchAlg_{\delta}$ if and only if $\gamma=\delta$ \cite[Theorem 4.4]{PoshtkohiPumpluen2026a}.

Since $\mathcal{A}_\gamma $ is a quadratic algebra, we have
 $\mathcal{A}_\gamma =\R 1\oplus V_\gamma $ with the pure part
\[
V_\gamma 
 =\{u\in\mathcal{A}_\gamma \mid \Tr(u)=0\}
 =\Span_\R\{i,j,k,\ell\}.
\]
Define
\[
\begin{aligned}
\times_\gamma:\;&V_\gamma \times V_\gamma 
   \longrightarrow V_\gamma ,\\
&u\times_\gamma v=\operatorname{pr}(uv),
\end{aligned}
\]
where $\operatorname{pr}:\mathcal{A}_\gamma\to V_\gamma$ is the canonical projection.  Then
$(V_\gamma ,\times_\gamma)$ is a four-dimensional anti-commutative algebra over $\R$, and
\[
uv=-\frac{1}{2}N_{\mathcal{A}_\gamma }(u,v)+u\times_\gamma v
\]
for all $u,v\in V_\gamma $.
Here,  $\frac{1}{2}N_{\mathcal{A}_\gamma }(u,v)$ denotes the symmetric bilinear form  associated to $N_{\mathcal{A}_\gamma }$, and
  $\times_\gamma $ is anti-symmetric. The multiplication of $\mathcal{A}_\gamma =\R\oplus V_\gamma $ is given by
 $$(\alpha,u)(\beta,v)=(\alpha\beta-\frac{1}{2}N_{\mathcal{A}_\gamma }(u,v),\alpha v+\beta u+u\times_\gamma v) $$
 for all $\alpha,\beta\in \R$ and $u,v\in V_\gamma $.
Note that
\[
  uv+vu=-N_{\mathcal{A}_\gamma }(u,v) \text{ and }   uv-vu=2(u \times_\gamma v).
\]
 We denote $V_\gamma={\rm span}_\R\{i,j,k,\ell\}$, $\langle u,v\rangle=\frac{1}{2}N(u,v)$ and  write $u\times v=u\times_\gamma v={\rm pr}(uv)$ when the context is clear. We also write $N=N_{\mathcal{A}_\gamma }$.

Henceforward, we will view the family $\{\mathcal{A}_\gamma\,|\, \gamma \in [0,1]\}$ as a family of algebras that are connected via \emph{switching} the parameter $\gamma$, which can also be viewed as a variable that depends on the time $t$.

\subsection{Systems of quadratic differential equations on $\mathbb R^5$}
We first show that the multiplication algebra of $\mathcal A_\gamma$ is as large as possible.
 This means that every
system of quadratic differential equations $\dot p=F(p)$ on $\mathbb R^5$ can be written as a differential
equation over $\mathcal A_\gamma$, whenenver $F:\mathcal A_\gamma\to\mathcal A_\gamma$ is a map whose components in the basis $1,i,j,k,\ell$ are real polynomials of
degree at most two in $p_0,\dots,p_4$ .

\begin{proposition}\label{prop:mult-algebra}
For every $\gamma\in\mathbb R$, $0\leq\gamma\leq1$, the multiplication algebra of $\mathcal A_\gamma$, i.e. the subalgebra of
${\rm End}_{\mathbb R}(\mathcal A_\gamma)$ generated by all $L^\gamma_a$ and $R^\gamma_a$, $a\in\mathcal A_\gamma$, is all of
${\rm End}_{\mathbb R}(\mathcal A_\gamma)$.
\end{proposition}

\begin{proof}
Let $\mathcal M$ denote the multiplication algebra. We first show that
$\mathcal A_\gamma$ is an irreducible $\mathcal M$-module. For this, let $W\neq0$ be a subspace of $\mathcal A_\gamma$, such that
$L^\gamma_aW\subseteq W$ and $R^\gamma_aW\subseteq W$ for all $a\in \mathcal A_\gamma$, and let $0\neq w\in W$. Then
$N(w)1=w\star_\gamma\bar w=R^\gamma_{\bar w}w\in W$ with $N(w)>0$, so $1\in W$, and hence $a=L^\gamma_a1\in W$ for
all $a$. Thus $W=\mathcal A_\gamma$.

 Next, let $C\in {\rm End}_{\mathbb R}(\mathcal A_\gamma)$ such that $C\circ L^\gamma_a=L^\gamma_a\circ C$ and $C\circ  R^\gamma_a = R^\gamma_a \circ  C$. Put $c=C(1)$. Then
$a\star_\gamma c=L^\gamma_aC(1)=C(a)=R^\gamma_aC(1)=c\star_\gamma a$ for all $a$, so the pure part $u=u_1i+u_2j+u_3k+u_4\ell$
of $c$ satisfies $u\times_\gamma v=0$ for all $v\in V$. By Table~\ref{tab:mult},
\[
u\times_\gamma\ell=-\tfrac1{\sqrt2}\bigl((u_2+u_3)i+(u_1+u_3)j+(u_1+u_2)k\bigr),
\]
so $u_1=u_2=u_3=0$, and then $u\times_\gamma i=u_4\,\tfrac1{\sqrt2}(j+k)$ gives $u_4=0$. Hence $c\in\mathbb R1$ and
$C(a)=a\star_\gamma c=c\,a$, i.e.\ $C=c\,\mathrm{id}$.

By the Jacobson Density Theorem~\cite{Jacobson1945}, every subalgebra of endomorphisms of a finite-dimensional real vector space
that acts irreducibly and has commutant $\mathbb R\,\mathrm{id}$ is the full endomorphism algebra. Hence $\mathcal M=\operatorname{End}_{\mathbb R}(\mathcal A_\gamma)$.
\end{proof}

\begin{proposition}\label{prop:every-quadratic}
Let $0\leq\gamma\leq1$ and identify $\mathbb R^5$ with $\mathcal A_\gamma$ via the basis $1,i,j,k,\ell$.
Every map $F:\mathcal A_\gamma\to\mathcal A_\gamma$ whose components in the basis $1,i,j,k,\ell$ are real polynomials of
degree at most two in $p_0,\dots,p_4$ can be written as 
\begin{equation}\label{eq:every-quadratic}
F(p)=c+T(p)+\sum_{r=0}^{4}T_r(p)\star_\gamma P_r(p),
\end{equation}
where $c\in\mathcal A_\gamma$ and each of $T,T_0,\dots,T_4,P_0,\dots,P_4$ is a finite sum of compositions
$M_1\circ\dots\circ M_s$ with $M_1,\dots,M_s\in\{L^\gamma_a,R^\gamma_a\mid a\in\mathcal A_\gamma\}$. In particular, every
system of quadratic differential equations $\dot p=F(p)$ on $\mathbb R^5$ can be written as a differential
equation over $\mathcal A_\gamma$.
\end{proposition}

The representation \eqref{eq:every-quadratic} is not unique; for instance $p\star_\gamma p=2p_0p-N(p)1$.

\begin{proof}
Decomposing each component of $F$ into its homogeneous parts in $\mathbb R[p_0,\dots,p_4]$, which is a unique decomposition, gives
$F=F_0+F_1+F_2$ with $F_d$ homogeneous of degree $d$.

 Thus $F_0=c$ is a constant
in $\mathcal A_\gamma$ and $F_1=T$ is linear. Each component of $F_2$ is a quadratic form, so
\[
B(x,y)=\tfrac12\bigl(F_2(x+y)-F_2(x)-F_2(y)\bigr)
\]
defines a symmetric bilinear map $B:\mathcal A_\gamma\times\mathcal A_\gamma\to\mathcal A_\gamma$ with $F_2(p)=B(p,p)$.

Put $e_0=1,e_1=i,e_2=j,e_3=k,e_4=\ell$. For $r=0,\dots,4$, define linear maps $T_r,P_r:\mathcal A_\gamma\to\mathcal A_\gamma$ by
\[
T_r(x)=B(x,e_r),\qquad P_r(p)=p_r1,
\]
where $p=\sum_{r=0}^4p_re_r$. Since $\star_\gamma$ is bilinear and $1$ is the identity of $\mathcal A_\gamma$, we get
\[
T_r(p)\star_\gamma P_r(p)=T_r(p)\star_\gamma(p_r1)=p_r\,\bigl(T_r(p)\star_\gamma1\bigr)=p_r\,B(p,e_r).
\]
Summing over $r$ and using that $B$ is linear in its second argument, we obtain that
\[
\sum_{r=0}^4T_r(p)\star_\gamma P_r(p)=B\Bigl(p,\sum_{r=0}^4p_re_r\Bigr)=B(p,p)=F_2(p).
\]
This proves \eqref{eq:every-quadratic} for linear maps $T,T_r,P_r$.

The subalgebra of ${\rm End}_{\mathbb R}(\mathcal A_\gamma)$ generated by the $L^\gamma_a$ and $R^\gamma_a$ is the linear span of
all compositions $M_1\circ\dots\circ M_s$ with $M_1,\dots,M_s\in\{L^\gamma_a,R^\gamma_a\mid a\in\mathcal A_\gamma\}$, since this span
contains the generators and is closed under composition; the identity is $L^\gamma_1$. Real scalar factors can be
absorbed, since $\lambda L^\gamma_a=L^\gamma_{\lambda a}$, so this span consists of finite sums of such compositions. By
Proposition~\ref{prop:mult-algebra}, it equals $\operatorname{End}_{\mathbb R}(\mathcal A_\gamma)$, so it contains $T$, $T_r$ and $P_r$.
\end{proof}

The proof of Proposition~\ref{prop:every-quadratic} only uses that $\mathcal A_\gamma$ is unital and that its
multiplication algebra is ${\rm End}_{\mathbb R}(\mathcal A_\gamma)$; it therefore holds for every finite-dimensional 
unital real algebra with this property. This includes $\mathbb H$, but not $\mathbb C$ (since $\mathbb C\subsetneq{\rm End}_{\mathbb R}(\mathbb C)$).

\begin{remark}\label{rem:basis}
An equation such as $\dot p=A\star_\gamma p+B\star_\gamma(p\star_\gamma p)$ (which is Equation \ref{eq:riccati-main}) does not depend on a basis, but the
five real equations obtained from it do. If $G\in\mathrm{GL}_5(\mathbb R)$ is a change of basis and $q=Gp$, then
$\dot p=F(p)$ becomes $\dot q=GF(G^{-1}q)$, which is again quadratic, but in general with different
coefficients.  Proposition~\ref{prop:every-quadratic} holds for every choice of basis, since
Proposition~\ref{prop:mult-algebra} does not refer to a basis. We use the canonical basis
$1,i,j,k,\ell$, in which the products are given by Table~\ref{tab:mult}.
\end{remark}

By  Proposition~\ref{prop:every-quadratic}, every quadratic system on $\mathbb R^5$ has a representation over
$\mathcal A_\gamma$, and such representations in general have many terms. We will study systems with a short representation: the
Riccati equations $\dot p=A\star_\gamma p+B\star_\gamma(p\star_\gamma p)$, $A,B\in\mathcal A_\gamma$, form a
$10$-dimensional linear subspace of the $105$-dimensional space of quadratic vector fields on $\mathbb R^5$.

\subsection{State convention}
Let
\[
 E=\R1\oplus\R i\oplus\R j\oplus\R k\oplus\R\ell
\]
denote the underlying vector space of the algebras $\SwitchAlg_\gamma$.
Define the $\gamma$-dependent linear map
\[
 S_\gamma:E\longrightarrow E,
 \qquad
 S_\gamma(p_0+p_1i+p_2j+p_3k+p_4\ell)
 =p_0+p_1i+p_2j+p_3k+\gamma p_4\ell,
\]
and write
\begin{equation}
 p_\gamma:=S_\gamma(p)
 =p_0+p_1i+p_2j+p_3k+\gamma p_4\ell.
 \label{eq:weighted-state-main}
\end{equation}
This map is used for the fifth component of the elliptic solutions
in \cref{sec:exact-integrable-path} and for the coefficients and
initial conditions of the Riccati equations in
\cref{sec:riccati-dynamics}.
For $\gamma>0$, $S_\gamma$ is invertible. At $\gamma=0$, we have $ \operatorname{im}S_0=\mathbb H.$

For the Riccati equations in \cref{sec:riccati-dynamics}, the selection is
$A_\gamma=S_\gamma A_1$, $B_\gamma=S_\gamma B_1$, and
$p_\gamma(0)=S_\gamma p_1(0)$.
At $\gamma=0$, these elements have zero $\ell$-component and
belong to $\mathbb H$, so their sums and $\star_0$-products remain
in $\mathbb H$.
Note that the ambient algebra $\SwitchAlg_0$ is
five-dimensional; the image of the state map has dimension four.

\begin{remark}\label{prop:one-element-invisibility}
Let $p=p_0+p_1i+p_2j+p_3k+p_4\ell$ with coordinates independent of $\gamma$. Since $\mathcal{A}_\gamma $ is quadratic, we have
$p\star_\gamma p=2p_0p-N(p)1$ and  $N$ does not depend on $\gamma$. Hence all powers of $p$,
and thus every real polynomial or convergent real power series in $p$, are independent of $\gamma$.
Thus $\gamma$ is only ``visible'' in products of at least two linearly independent elements, such as
$B_\gamma\star_\gamma(p\star_\gamma p)$ in \eqref{eq:riccati-main}. Thus \emph{switching} (i.e., changing $\gamma$, or more precisely, the algebra $\mathcal A_\gamma$) is detected only by interactions among independent elements, nonreal coefficients, or different bracketings.
\end{remark}

\section{Elliptic solutions depending on $\gamma$}
\label{sec:exact-integrable-path}
We now present a family of polynomial differential equations on $\R^5$  together with an explicit solution in terms of
Jacobi elliptic functions.
 Identifying $\R^5$ with
$\mathcal A_\gamma$ via the basis $1,i,j,k,\ell$, the equations can be written in terms of the
multiplication $\star_\gamma$. For $\gamma=0$ the solution is trigonometric and lies in the quaternion
subalgebra $\Span\{1,i,j,k\}$ of $\mathcal A_0$. For $\gamma=1$ it is the solution in $\Pent$ found in~\cite{PoshtkohiPumplunSubmitted}.

As $\gamma$ moves away from $0$, the invariant curve changes from a curve isomorphic to the conic $p_0^2+p_1^2=1$  to a curve whose smooth projective model is an elliptic curve, and the affine span of the solution from dimension four to five,
and its period  increases with $\gamma$.

\subsection{A polynomial system with an invariant elliptic curve}

Let $m\in(0,1)$ and $\sigma\neq0$.  Consider
\begin{equation}
\begin{aligned}
\dot p_0={}&p_1-p_3+p_0p_1+2(1+m)p_0p_3-2(1+m)p_1p_2,\\
\dot p_1={}&p_2-(1+m)p_0-p_0^2+2mp_0p_2,\\
\dot p_2={}&2p_0p_1,\\
\dot p_3={}&1-2(1+m)p_0^2+\frac{3m}{\sigma}p_4p_0,\\
\dot p_4={}&3\sigma p_0p_3.
\end{aligned}
\label{eq:general-elliptic-family-main}
\end{equation}
Define
\begin{equation}
\begin{aligned}
G_1&=p_2-p_0^2,&
G_2&=p_3-p_0p_1,\\
G_3&=p_4-\sigma p_0p_2,&
G_4&=p_1^2-mp_2^2+(1+m)p_2-1.
\end{aligned}
\label{eq:general-ideal-main}
\end{equation}

\begin{theorem} [Invariant prime genus-one ideal]
\label{thm:general-invariant-ideal}
Let $m\in(0,1)$, $\sigma\ne0$, $f(t)=(1-t^2)(1-mt^2)$, and let
$$I_{m,\sigma}=\langle G_1,G_2,G_3,G_4\rangle\subset\C[p_0,\dots,p_4].$$
Then the following holds:
\\ (i)
 $I_{m,\sigma}$ is invariant under \cref{eq:general-elliptic-family-main}, i.e. $\mathcal L_F G_r\in I_{m,\sigma}$ for
$r=1,\dots,4$, where $\mathcal L_F G=\nabla G\cdot F$;
\\ (ii) $\C[p_0,\dots,p_4]/I_{m,\sigma}\cong\C[p_0,p_1]/(p_1^2-f(p_0))$, and $I_{m,\sigma}$ is a prime ideal;
\\ (iii) $V(I_{m,\sigma})$ is isomorphic to the affine curve $p_1^2=f(p_0)$, whose smooth projective
model has genus one.
\end{theorem}

We note that the hypothesis  in Theorem~\ref{thm:general-invariant-ideal} forces $\gamma>0$ (if  $m=\mu\gamma$ and $\sigma=\gamma$ then $\gamma=0$ would imply that  $m=0\notin(0,1)$ and $\sigma=0$). The case $\gamma=0$ will be treated separately.

\begin{proof}
$(ii)$ Let $\varphi\colon\C[p_0,\dots,p_4]\to\C[p_0,p_1]$ be the substitution
$p_2\mapsto p_0^2$, $p_3\mapsto p_0p_1$, $p_4\mapsto\sigma p_0^3$. Its kernel is
$\langle G_1,G_2,G_3\rangle$, and we have
\[
\varphi(G_4)=p_1^2-mp_0^4+(1+m)p_0^2-1=p_1^2-f(p_0).
\]
Hence we get $\C[p_0,\dots,p_4]/I_{m,\sigma}\cong\C[p_0,p_1]/(p_1^2-f(p_0))$.
The polynomial $p_1^2-f(p_0)\in\C[p_0][p_1]$ is monic of degree two in $p_1$, so it is reducible
only if $f$ is a square in $\C[p_0]$. But $f$ has the four simple roots $\pm1,\pm1/\sqrt m$.
Hence $p_1^2-f$ is irreducible in the unique factorisation domain $\C[p_0,p_1]$, the quotient is a domain, and
$I_{m,\sigma}$ is prime.
Moreover, this gives
\begin{equation}
 p_1^2=(1-p_0^2)(1-mp_0^2).
 \label{eq:elliptic-quartic-main}
\end{equation}
\\
$(i)$ It suffices to show $\varphi(\mathcal L_FG_r)\in(p_1^2-f)$. Applying $\varphi$ to
\eqref{eq:general-elliptic-family-main} gives
\[
\dot p_0\mapsto p_1,\quad \dot p_1\mapsto -(1+m)p_0+2mp_0^3=\tfrac12f'(p_0),\quad
\dot p_2\mapsto 2p_0p_1,
\]
$$\dot p_3\mapsto 1-2(1+m)p_0^2+3mp_0^4,\quad
\dot p_4\mapsto 3\sigma p_0^2p_1.$$
Therefore
\begin{align*}
\varphi(\mathcal L_FG_1)&=2p_0p_1-2p_0p_1=0,\\
\varphi(\mathcal L_FG_2)&=1-2(1+m)p_0^2+3mp_0^4-p_1^2-p_0\bigl(-(1+m)p_0+2mp_0^3\bigr)=f(p_0)-p_1^2,\\
\varphi(\mathcal L_FG_3)&=3\sigma p_0^2p_1-\sigma p_0^2p_1-2\sigma p_0^2p_1=0,\\
\varphi(\mathcal L_FG_4)&=p_1f'(p_0)-4mp_0^3p_1+2(1+m)p_0p_1=0 .
\end{align*}
$(iii)$ By $(ii)$, $V(I_{m,\sigma})$ is the image of the curve $p_1^2=f(p_0)$ under the closed
embedding $(p_0,p_1)\mapsto(p_0,p_1,p_0^2,p_0p_1,\sigma p_0^3)$. Since $f$ is square-free of
degree four, this affine curve is smooth, and its smooth projective model is a double cover
of $\mathbb P^1$ branched at four points, hence of genus one.
\end{proof}

Let
\[
C_{m,\sigma}=\{p\in\mathbb R^5\mid G_1(p)=G_2(p)=G_3(p)=G_4(p)=0\},
\]
be the  set of real zeros
of the ideal $I_{m,\sigma}$, cf. Theorem~\ref{thm:general-invariant-ideal}. By Theorem~\ref{thm:general-invariant-ideal}$(i)$, there are
polynomials $h_{rs}$ with $L_FG_r=\sum_{s=1}^4h_{rs}G_s$, so for a solution $p(t)$ the functions
$g_r(t)=G_r(p(t))$ satisfy the linear system $\dot g_r=\sum_{s}h_{rs}(p(t))\,g_s$. If $p(0)\in C_{m,\sigma}$, then
$g(0)=0$, hence $g\equiv0$ by uniqueness, i.e.\ $p(t)\in C_{m,\sigma}$ for all $t$. Thus we know that $C_{m,\sigma}$ is invariant under \eqref{eq:general-elliptic-family-main}.  We call $C_{m,\sigma}$ the \emph{invariant curve} of \eqref{eq:general-elliptic-family-main}.

For $0\leq m<1$, let 
$$K(m)=\int_0^{\pi/2}(1-m\sin^2(\theta))^{-1/2}d\theta,$$ and let $\varphi(t)=\operatorname{am}(t\mid m)$
be the inverse of $\varphi\mapsto\int_0^\varphi(1-m\sin^2\theta)^{-1/2}d\theta$. Let $\sn(t\mid m)=\sin\varphi(t)$, $\cn(t\mid m)=\cos\varphi(t)$ and $\dn(t\mid m)=\sqrt{1-m\,\sn^2(t\mid m)}$ denote
the Jacobi elliptic functions with parameter $m$. Since $\frac{d\varphi}{dt}=\dn(t\mid m)$, the chain
rule gives
\[
\begin{gathered}
 \tfrac{d}{dt}\sn(t\mid m)=\cn(t\mid m)\dn(t\mid m),\qquad
 \tfrac{d}{dt}\cn(t\mid m)=-\sn(t\mid m)\dn(t\mid m),\\
 \tfrac{d}{dt}\dn(t\mid m)=-m\,\sn(t\mid m)\cn(t\mid m),
 \end{gathered}
\]
and by definition $\cn^2(t\mid m)=1-\sn^2(t\mid m)$ and $\dn^2(t\mid m)=1-m\,\sn^2(t\mid m)$.
The function $\sn(\cdot\mid m)$ has minimal period $4K(m)$ and for $m=0$ they equal $\sin$, $\cos$ and $1$.

\begin{theorem}[Jacobi solution]
\label{thm:general-jacobi-solution}
The solution of \cref{eq:general-elliptic-family-main} through $(0,1,0,0,0)$ is
\begin{equation}
\begin{aligned}
 p_0(t)&=\sn(t\mid m),\\
 p_1(t)&=\cn(t\mid m)\dn(t\mid m),\\
 p_2(t)&=\sn^2(t\mid m),\\
 p_3(t)&=\sn(t\mid m)\cn(t\mid m)\dn(t\mid m),\\
 p_4(t)&=\sigma\sn^3(t\mid m),
\end{aligned}
\label{eq:general-jacobi-solution-main}
\end{equation}
with minimal period $4K(m)$.
\end{theorem}

\begin{proof}
The functions satisfy
$\dot p_0(t)=\cn(t\mid m)\dn(t\mid m)=p_1(t)$. Moreover,
\[
-p_3(t)+p_0(t)p_1(t)=0,\qquad p_0(t)p_3(t)-p_1(t)p_2(t)=0,
\]
so the first equation holds. Using $\cn^2(t\mid m)=1-\sn^2(t\mid m)$ and
$\dn^2(t\mid m)=1-m\,\sn^2(t\mid m)$, we obtain
\[
\begin{aligned}
\dot p_1(t)&=-\sn(t\mid m)\dn^2(t\mid m)-m\,\sn(t\mid m)\cn^2(t\mid m)\\
&=-(1+m)\sn(t\mid m)+2m\,\sn^3(t\mid m)\\
&=p_2(t)-(1+m)p_0(t)-p_0(t)^2+2m\,p_0(t)p_2(t),
\end{aligned}
\]
and $\dot p_2(t)=2\sn(t\mid m)\cn(t\mid m)\dn(t\mid m)=2p_0(t)p_1(t)$. Further,
\[
\begin{aligned}
\dot p_3(t)&=\cn^2(t\mid m)\dn^2(t\mid m)-\sn^2(t\mid m)\bigl(\dn^2(t\mid m)+m\,\cn^2(t\mid m)\bigr)\\
&=1-2(1+m)\sn^2(t\mid m)+3m\,\sn^4(t\mid m)\\
&=1-2(1+m)p_0(t)^2+\tfrac{3m}{\sigma}\,p_4(t)p_0(t),
\end{aligned}
\]
and
\[
\dot p_4(t)=3\sigma\,\sn^2(t\mid m)\cn(t\mid m)\dn(t\mid m)=3\sigma\,p_0(t)p_3(t).
\]
Hence \eqref{eq:general-jacobi-solution-main} solves \eqref{eq:general-elliptic-family-main}. Since $\sn(0\mid m)=0$ and
$\cn(0\mid m)=\dn(0\mid m)=1$, it satisfies $p(0)=(0,1,0,0,0)$, and it is the solution through this point
because the right-hand side of \eqref{eq:general-elliptic-family-main} is polynomial. All components are $4K(m)$-periodic, and every period of $p$ is a period of
$p_0=\sn(\cdot\mid m)$, whose minimal period is $4K(m)$.
\end{proof}

\begin{proposition}\label{prop:orbit-on-curve}
Let $m\in(0,1)$, $\sigma\neq0$, $f(t)=(1-t^2)(1-mt^2)$ and
$\iota(p_0,p_1)=(p_0,\,p_1,\,p_0^2,\,p_0p_1,\,\sigma p_0^3)$.
\\ (i) We have 
$C_{m,\sigma}=\iota\bigl(\{(p_0,p_1)\in\mathbb R^2\mid p_1^2=f(p_0)\}\bigr).$
\\ (ii)
 $C_{m,\sigma}$ has three connected components: the image of the oval with $|p_0|\le1$, and the images of the two
unbounded branches with $p_0\ge1/\sqrt m$ and $p_0\le-1/\sqrt m$. The orbit $\{p(t)\mid t\in\mathbb R\}$ of the solution
\eqref{eq:general-jacobi-solution-main} is the image of the oval.
\end{proposition}

\begin{proof}
$(i)$ The equations $G_1=G_2=G_3=0$ mean that $p_2=p_0^2$, $p_3=p_0p_1$ and $p_4=\sigma p_0p_2=\sigma p_0^3$, i.e.
$p=\iota(p_0,p_1)$. Substituting this into $G_4$ gives $p_1^2-mp_0^4+(1+m)p_0^2-1=p_1^2-f(p_0)$.
\\ $(ii)$ We know that  $f(t)\geq0$ exactly when $|t|\leq1$ or $|t|\ge1/\sqrt m$, and that $f$ vanishes only at $\pm1$ and $\pm1/\sqrt m$, thus the
real curve $p_1^2=f(p_0)$ consists of the oval $p_1=\pm\sqrt{f(p_0)}$, $|p_0|\le1$, and of two  branches
$p_1=\pm\sqrt{f(p_0)}$, $p_0\ge1/\sqrt m$ and $p_0\le-1/\sqrt m$. Since $\iota$ is continuous and injective, $C_{m,\sigma}$ has
the three stated components.

By \eqref{eq:general-jacobi-solution-main}, we know that $p(t)=\iota (\sn(t\mid m),\cn(t\mid m)\dn(t\mid m))$, and
$\cn^2(t\mid m)\dn^2(t\mid m)=f(\sn(t\mid m))$. The amplitude $\varphi(t)$ satisfies $\dot\varphi(t)=\dn(t\mid m)>0$,
$\varphi(0)=0$, $\varphi(K(m))=\pi/2$ and $\varphi(t+2K(m))=\varphi(t)+\pi$. Hence on $[-K(m),K(m)]$ the function
$\sn(t\mid m)=\sin \varphi(t)$ increases from $-1$ to $1$ with $\cn(t\mid m)=\cos\varphi(t)\ge0$, and on $[K(m),3K(m)]$ it
decreases from $1$ to $-1$ with $\cn(t\mid m)\le0$. As $\dn(t\mid m)>0$, the point
$\bigl(\sn(t\mid m),\cn(t\mid m)\dn(t\mid m)\bigr)$ runs through the upper half $p_1=\sqrt{f(p_0)}$ and then the lower half
$p_1=-\sqrt{f(p_0)}$ of the oval. Thus the orbit is the image of the oval.
\end{proof}

\subsection{From the quaternion to the Pentarnion solution}

Fix $0<\mu<1$ and choose
\begin{equation}
 m=\mu\gamma,
 \qquad
 \sigma=\gamma.
 \label{eq:switching-integrable-choice-main}
\end{equation}
Because $m/\sigma=\mu$, the coordinate field remains regular at $\gamma=0$:
\begin{equation}
\begin{aligned}
\dot p_0={}&p_1-p_3+p_0p_1+2(1+\mu\gamma)p_0p_3
 -2(1+\mu\gamma)p_1p_2,\\
\dot p_1={}&p_2-(1+\mu\gamma)p_0-p_0^2+2\mu\gamma p_0p_2,\\
\dot p_2={}&2p_0p_1,\\
\dot p_3={}&1-2(1+\mu\gamma)p_0^2+3\mu p_4p_0,\\
\dot p_4={}&3\gamma p_0p_3.
\end{aligned}
\label{eq:flagship-switching-field}
\end{equation}

\begin{proposition}
\label{thm:flagship-orbit}
Let $0<\mu<1$ and $0\leq\gamma\leq 1$. Let
\[
 s_\gamma(t)=\sn(t\mid\mu\gamma),
 \qquad
 v_\gamma(t)=\cn(t\mid\mu\gamma)\dn(t\mid\mu\gamma).
\]
Then the solution of  \cref{eq:flagship-switching-field} such that $p(0)=(0,1,0,0,0)$ is
\begin{equation}
 p_\gamma^*(t)
 =s_\gamma+v_\gamma i+s_\gamma^2j+s_\gamma v_\gamma k
 +\gamma s_\gamma^3\ell
\label{eq:flagship-orbit-main}
\end{equation}
and its minimal period is
\begin{equation}
 T(\gamma)=4K(\mu\gamma).
 \label{eq:flagship-period-main}
\end{equation}
The fifth component has precisely the form selected by the state map in \cref{eq:weighted-state-main}.

For $\gamma=0$, we have $p_0^*(t)=\sin t+\cos t\, i+\sin^2t\, j+\sin t\cos t\, k$, which lies in the
quaternion subalgebra $\Span\{1,i,j,k\}$ of $\mathcal A_0$.
\end{proposition}

\begin{proof}
Put $m=\mu\gamma\in[0,1)$. The system \eqref{eq:flagship-switching-field} is
\eqref{eq:general-elliptic-family-main} with $\sigma=\gamma$, except that the term
$\frac{3m}{\sigma}p_4p_0$ is replaced by $3\mu\, p_4p_0$. For $p_4=\gamma s_\gamma^3$ and $p_0=s_\gamma$ we
have $3\mu\,p_4p_0=3m\,s_\gamma^4$, which is exactly the value used for this term in the proof of
Theorem~\ref{thm:general-jacobi-solution}. Hence the computation in that proof, with $s=s_\gamma$,
$cd=v_\gamma$ and $\sigma=\gamma$, shows that \eqref{eq:flagship-orbit-main} solves
\eqref{eq:flagship-switching-field}. This includes $\gamma=0$, where $s_0=\sin$, $v_0=\cos$ and $p_4=0$.
At $t=0$ we get $p(0)=(0,1,0,0,0)$, and uniqueness follows since the right-hand side is polynomial. Analogously as  in Theorem~\ref{thm:general-jacobi-solution}, every period of $p_\gamma^*$ is a period of $s_\gamma=\sn(\cdot\mid m)$, whose minimal period is $4K(m)$.
\end{proof}

It remains to treat the case that $\gamma=0$. For $0\le\gamma\le1$ and $0<\mu<1$, let $G_1^\gamma,\dots,G_4^\gamma$ denote the polynomials from \eqref{eq:general-ideal-main} with
$m=\mu\gamma$ and $\sigma=\gamma$, i.e.
\[
G_1^\gamma=p_2-p_0^2,\quad G_2^\gamma=p_3-p_0p_1,\quad G_3^\gamma=p_4-\gamma p_0p_2,\quad
G_4^\gamma=p_1^2-\mu\gamma p_2^2+(1+\mu\gamma)p_2-1,
\]
and let $J_\gamma=\langle G_1^\gamma,\dots,G_4^\gamma\rangle\subset\mathbb C[p_0,\dots,p_4]$.

For $0\leq\gamma\leq 1$, the invariant curve is given by
$$C_\gamma=\{p\in\mathbb R^5\mid G_1^\gamma(p)=\dots=G_4^\gamma(p)=0\}.$$ 
For $\gamma>0$, we obtain $C_\gamma=C_{\mu\gamma,\gamma}$.

\begin{lemma}\label{lem:ideal-switching}
For $0\le\gamma\le1$ and $0<\mu<1$, the ideal $J_\gamma$ and hence $C_\gamma$ is invariant under  under \cref{eq:general-elliptic-family-main}.
\end{lemma}

\begin{proof}
For $\gamma>0$, the assertion follows from  Theorem~\ref{thm:general-invariant-ideal}$(i)$  with $m=\mu\gamma$ and $\sigma=\gamma$, since then
$\tfrac{3m}{\sigma}=3\mu$, and $J_\gamma=I_{m,\sigma}$.

 For $\gamma=0$, take
$\varphi_0:\mathbb C[p_0,\dots,p_4]\to\mathbb C[p_0,p_1]$, $p_2\mapsto p_0^2$, $p_3\mapsto p_0p_1$, $p_4\mapsto0$.
Then ${\rm Ker}(\varphi_0)=\langle G_1^0,G_2^0,G_3^0\rangle$ and $\varphi_0(G_4^0)=p_0^2+p_1^2-1$, so it suffices to show
$\varphi_0(L_FG_r^0)\in(p_0^2+p_1^2-1)$. Applying $\varphi_0$ to \eqref{eq:flagship-switching-field} with $\gamma=0$ gives
\[
\dot p_0\mapsto p_1,\qquad \dot p_1\mapsto-p_0,\qquad \dot p_2\mapsto2p_0p_1,\qquad \dot p_3\mapsto1-2p_0^2,\qquad
\dot p_4\mapsto0.
\]
Therefore
\[
\begin{aligned}
\varphi_0(L_FG_1^0)&=2p_0p_1-2p_0p_1=0,\\
\varphi_0(L_FG_2^0)&=1-2p_0^2-p_1^2+p_0^2=-(p_0^2+p_1^2-1),\\
\varphi_0(L_FG_3^0)&=0,\\
\varphi_0(L_FG_4^0)&=-2p_0p_1+2p_0p_1=0.
\end{aligned}
\]
\end{proof}

 We can now include the case $\gamma=0$: For $\gamma>0$, Proposition~\ref{prop:orbit-on-curve} with $m=\mu\gamma$, $\sigma=\gamma$ shows that the
orbit of \eqref{eq:general-jacobi-solution-main} is the component of $C_\gamma$ with $|p_0|\le1$. For $\gamma=0$,
$C_0=\{(p_0,p_1,p_0^2,p_0p_1,0)\mid p_0^2+p_1^2=1\}$ is connected, and the orbit
$$p^*_0(t)=\sin t+\cos t\,i+\sin^2t\,j+\sin t\cos t\,k$$ is all of $C_0$.

\begin{theorem}[Geometric transition]
\label{thm:integrable-geometric-transition}
 Fix some $\mu$, $0<\mu<1$.
\begin{enumerate}[label=\textnormal{(\roman*)}]

\item 
at $\gamma=0$, the invariant curve is isomorphic to the conic
$p_1^2+p_0^2=1$, has genus zero, and the orbit has affine span four;

\item for every $\gamma>0$, the smooth projective model of $C_\gamma$ has genus one and the orbit has full affine span five;
\item the period $T(\gamma)$ is strictly increasing;
\item the maximal fifth-coordinate amplitude is exactly $\gamma$.
\end{enumerate}
\end{theorem}

\begin{proof}
$(i) $ At $\gamma=0$, $m=0$, $p_4=0$, and the Jacobi functions reduce to sine and cosine.
The curve is the image of $x^2+y^2=1$ under
$(x,y)\mapsto(x,y,x^2,xy,0)$, with inverse given by projection onto
the first two coordinates. The linear independence of
$1,\sin t,\cos t,\cos 2t,\sin 2t$ gives affine span four.
\\ $(ii)$ For $\gamma>0$, \cref{eq:elliptic-quartic-main} has four distinct branch points.  If an affine relation holds on the orbit, it can be written as $A(p_0)+B(p_0)p_1=0$, with $\deg A\leq3$ and $\deg B\leq1$.  Since $1$ and $p_1$ are independent over $\R(p_0)$ on the nonsquare quartic, all coefficients vanish. 
\\ $(iii)$ and $(iv)$: Monotonicity of the complete elliptic integral $K(m)$ on $0<m<1$ gives the period statement, and $\max|\sn|=1$ gives the amplitude statement.
\end{proof}

At $\gamma=1$, choosing $\mu=1/2$ gives
\[
 p_1^*(t)=s+vi+s^2j+svk+s^3\ell,
 \qquad
 s=\sn(t\mid1/2),
 \quad
 v=\cn(t\mid1/2)\dn(t\mid1/2),
\]
which is exactly the full-affine-span Jacobi solution of \cite{PoshtkohiPumplunSubmitted}. 

In \cite{PoshtkohiPumplunSubmitted}, we only treated the case that  $m=1/2$, now $m=\mu\gamma$ varies continously with $\gamma$.

\begin{proposition}
\label{prop:ordered-product-representation}
For every $0<\gamma\leq1$ and $0<\mu<1$, there exist real-linear maps
\[
 \mathsf A_{\gamma,\mu},\mathsf B_{\gamma,\mu}:
 \SwitchAlg_\gamma\to\SwitchAlg_\gamma
\]
and a linear map $\mathsf N_{\gamma,\mu}$ such that the vector field in \cref{eq:flagship-switching-field} can be written
\begin{equation}
 F_{\gamma,\mu}(p)
 =k+\mathsf N_{\gamma,\mu}(p)
 +p\star_\gamma\mathsf A_{\gamma,\mu}(p)
 +\mathsf B_{\gamma,\mu}(p)\star_\gamma p.
 \label{eq:ordered-representation-main}
\end{equation}
The maps $A_{\gamma,\mu}$, $B_{\gamma,\mu}$ are not unique; the possible choices form an affine space of
dimension six.
\end{proposition}

\begin{proof}
Write $e_0=1$, $e_1=i$, $e_2=j$, $e_3=k$ and $e_4=\ell$.
Comparing the coefficients of $p_rp_s\,e_t$, with
$0\leq r\leq s\leq4$ and $0\leq t\leq4$, on both sides of
\eqref{eq:ordered-representation-main} gives a linear system
of $75$ equations for the $50$ entries of the matrices of
$A_{\gamma,\mu}$ and $B_{\gamma,\mu}$.

An explicit particular solution and its exact symbolic verification
are supplied with the accompanying code and data described in
Appendix~\ref{app:reproducibility}.
The verifier constructs the coefficient system from
Table~\ref{tab:mult}, solves it using exact symbolic arithmetic,
and verifies all $75$ coefficient equations by substitution.
With
\[
 \gamma=\frac{2\tau}{1-\tau^2},
 \qquad 0<\tau\leq\sqrt2-1,
\]
the parameter-dependent denominator factors of the supplied
solution are
\[
 1-\tau,\qquad 1+\tau,\qquad 1+\tau^2,\qquad
 \tau^4+26\tau^2+1.
\]
None vanishes on this interval, so the solution is defined
throughout the stated parameter range.

To establish the dimension of the solution space, consider
the associated homogeneous identity
\[
 p\star_\gamma A(p)+B(p)\star_\gamma p=0
 \qquad\text{for all }p\in\mathcal A_\gamma.
\]
Put $V=\operatorname{span}_{\mathbb R}\{i,j,k,\ell\}$ and
$v=\operatorname{Im}A(1)$.
Substituting $p=1$ and then comparing the coefficient of $s$
after substituting $p=s1+u$, $u\in V$, gives
\[
 B(1)=-A(1),\qquad (A+B)(u)=2v\times_\gamma u.
\]
Taking scalar parts of the homogeneous identity at $p=u$ yields
\[
 \langle u,v\times_\gamma u\rangle=0
 \qquad\text{for all }u\in V.
\]
Write $v=x+d\ell$ with $x\in\operatorname{span}\{i,j,k\}$,
and put $\omega=i+j+k$.
Applying this equality successively to
$u=\omega$, $u=\omega+\ell$ and $u=x+\ell$, the multiplication
table gives, respectively,
\[
 3\sqrt2\,d=0,\qquad
 -\sqrt2\,\langle x,\omega\rangle=0,\qquad
 \frac{\|x\|^2}{\sqrt2}=0.
\]
Thus $v=0$ and $B=-A$.

Let $T(u)=\operatorname{Im}A(u)$ for $u\in V$. The homogeneous
identity now gives $u\times_\gamma T(u)=0$ for every $u\in V$.
The multiplication table also gives
\[
 \ker\bigl(u\mapsto\ell\times_\gamma u\bigr)=\mathbb R\ell.
\]
Consequently, $T(\ell)=c\ell$ for some $c\in\mathbb R$.
Polarising $u\times_\gamma T(u)=0$ with $\ell$ gives
\[
 \ell\times_\gamma\bigl(T(u)-cu\bigr)=0,
\]
so $T(u)-cu\in\mathbb R\ell$. For $u\notin\mathbb R\ell$,
the equality $u\times_\gamma T(u)=0$ then forces $T(u)=cu$;
the same conclusion holds on $\mathbb R\ell$.
Hence the homogeneous solutions are exactly
\[
 A(p)=f(p)1+cp,\qquad
 B(p)=-f(p)1-cp,
\]
where $f:\mathcal A_\gamma\to\mathbb R$ is an arbitrary
real-linear functional and $c\in\mathbb R$.
Conversely, every such pair satisfies the homogeneous identity.
These choices give six independent real parameters, so the
coefficient matrix has rank $50-6=44$.

Since a particular solution exists, the solutions for
$A_{\gamma,\mu}$ and $B_{\gamma,\mu}$ form an affine space of
dimension six. The constant and linear terms determine
$k$ and $N_{\gamma,\mu}$, respectively.
\end{proof}

The same system also has an invariant affine plane on which all
solutions can be written explicitly.

\begin{proposition}[Invariant affine plane]
\label{prop:hyperbolic-invariant-plane}
Let $0<\mu<1$, $0\leq\gamma\leq1$, and $m=\mu\gamma$. Set
\[
 \xi=\frac{1}{2(1+m)},\qquad
 \zeta=\frac{2m^2+4m+3}{4(1+m)(1+2m)}.
\]
The affine plane
\[
 \mathcal P_{\mu,\gamma}
 =\{p\in\mathbb R^5:p_0=\xi,\ p_1=0,\ p_2=\zeta\}
\]
is invariant under \eqref{eq:flagship-switching-field}.
On this plane, put
\[
 U=p_3,\qquad W=p_4+\frac{1+2m}{3\mu}.
\]
Then
\[
 \dot U=\frac{3\mu}{2(1+m)}W,\qquad
 \dot W=\frac{3\gamma}{2(1+m)}U.
\]
For $\gamma>0$, let
$\lambda=3\sqrt{\mu\gamma}/(2(1+\mu\gamma))$.
For arbitrary $U_0=U(0)$ and $W_0=W(0)$, the solution is
\[
\begin{aligned}
 U(t)&=U_0\cosh(\lambda t)
       +\sqrt{\frac{\mu}{\gamma}}\,W_0\sinh(\lambda t),\\
 W(t)&=W_0\cosh(\lambda t)
       +\sqrt{\frac{\gamma}{\mu}}\,U_0\sinh(\lambda t).
\end{aligned}
\]
At $\gamma=0$, it is instead
$W(t)=W_0$ and $U(t)=U_0+\tfrac32\mu W_0t$.
All these solutions exist for every $t\in\mathbb R$.
\end{proposition}

\begin{proof}
Write $a=1+m$. On $\mathcal P_{\mu,\gamma}$, the first three
equations of \eqref{eq:flagship-switching-field} give
\[
 \dot p_0=(-1+2a\xi)p_3=0,\qquad
 \dot p_1=(1+2m\xi)\zeta-a\xi-\xi^2=0,\qquad
 \dot p_2=0.
\]
Thus the plane is invariant. The remaining equations become
\[
 \dot p_3=\frac{1+2m+3\mu p_4}{2a},\qquad
 \dot p_4=\frac{3\gamma}{2a}p_3,
\]
which give the stated system for $U,W$. For $\gamma>0$, its
coefficient matrix squares to $\lambda^2I_2$, giving the
hyperbolic-function formulas. The case $\gamma=0$ follows by
direct integration.
\end{proof}

The following example demonstrates this phenomenon.

\begin{example} 
 At $\mu=\tfrac12$ and $\gamma=1$, the initial point
$(\tfrac13,0,\tfrac{11}{24},0,0)$ gives
\[
 p(t)=\frac13+\frac{11}{24}j
 +\frac{2\sqrt2}{3}\sinh\!\left(\frac{t}{\sqrt2}\right)k
 +\frac43\left[\cosh\!\left(\frac{t}{\sqrt2}\right)-1\right]\ell.
\]
This solution grows exponentially as $t\to+\infty$, whereas
\eqref{eq:flagship-orbit-main} is periodic at the same parameter
values.
\end{example}

\begin{figure}[t]
\centering
\begin{subfigure}[t]{0.48\textwidth}
\centering
\includegraphics[width=\textwidth]{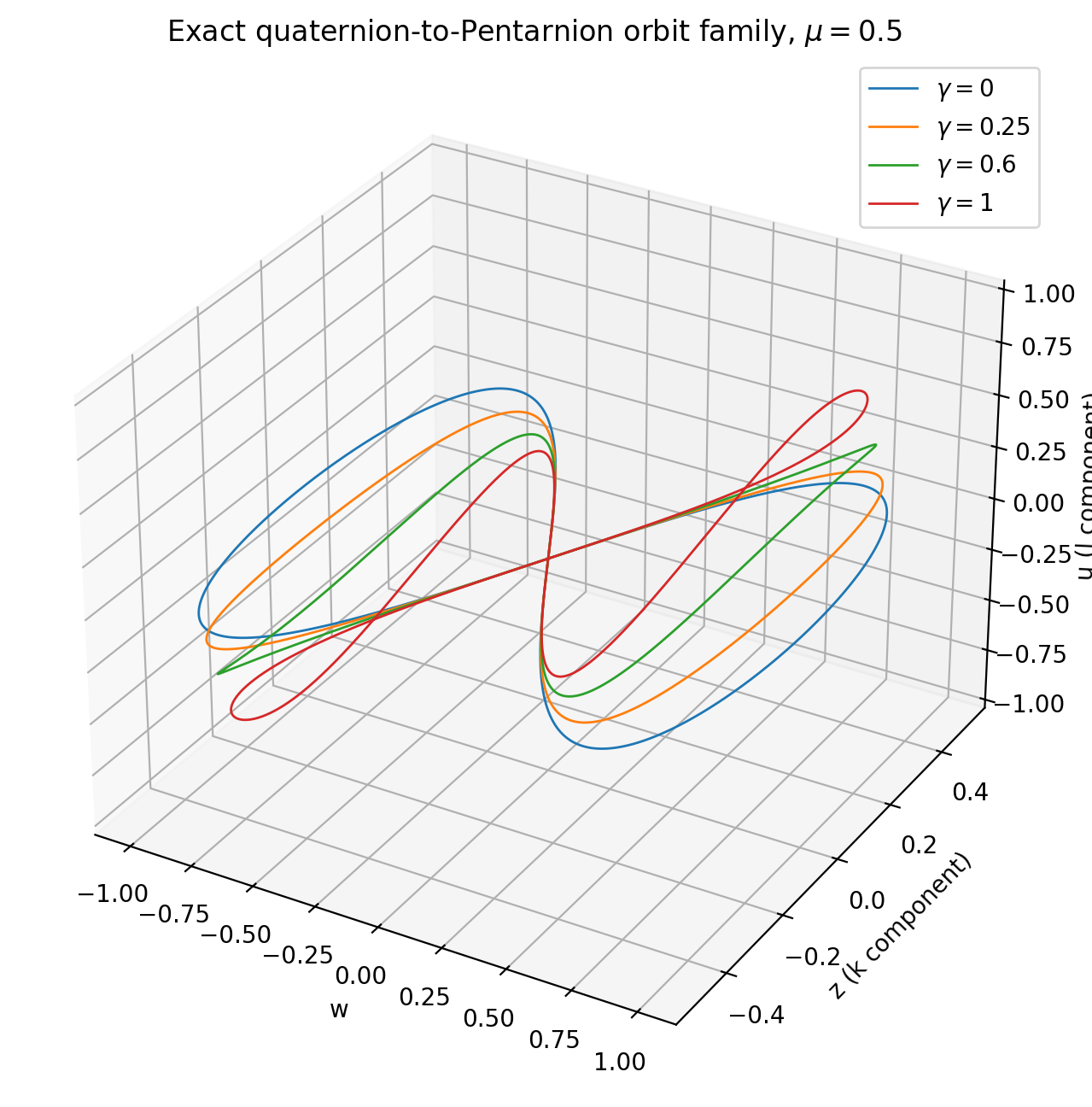}
\caption{Exact orbit family in a three-coordinate projection.}
\end{subfigure}\hfill
\begin{subfigure}[t]{0.48\textwidth}
\centering
\includegraphics[width=\textwidth]{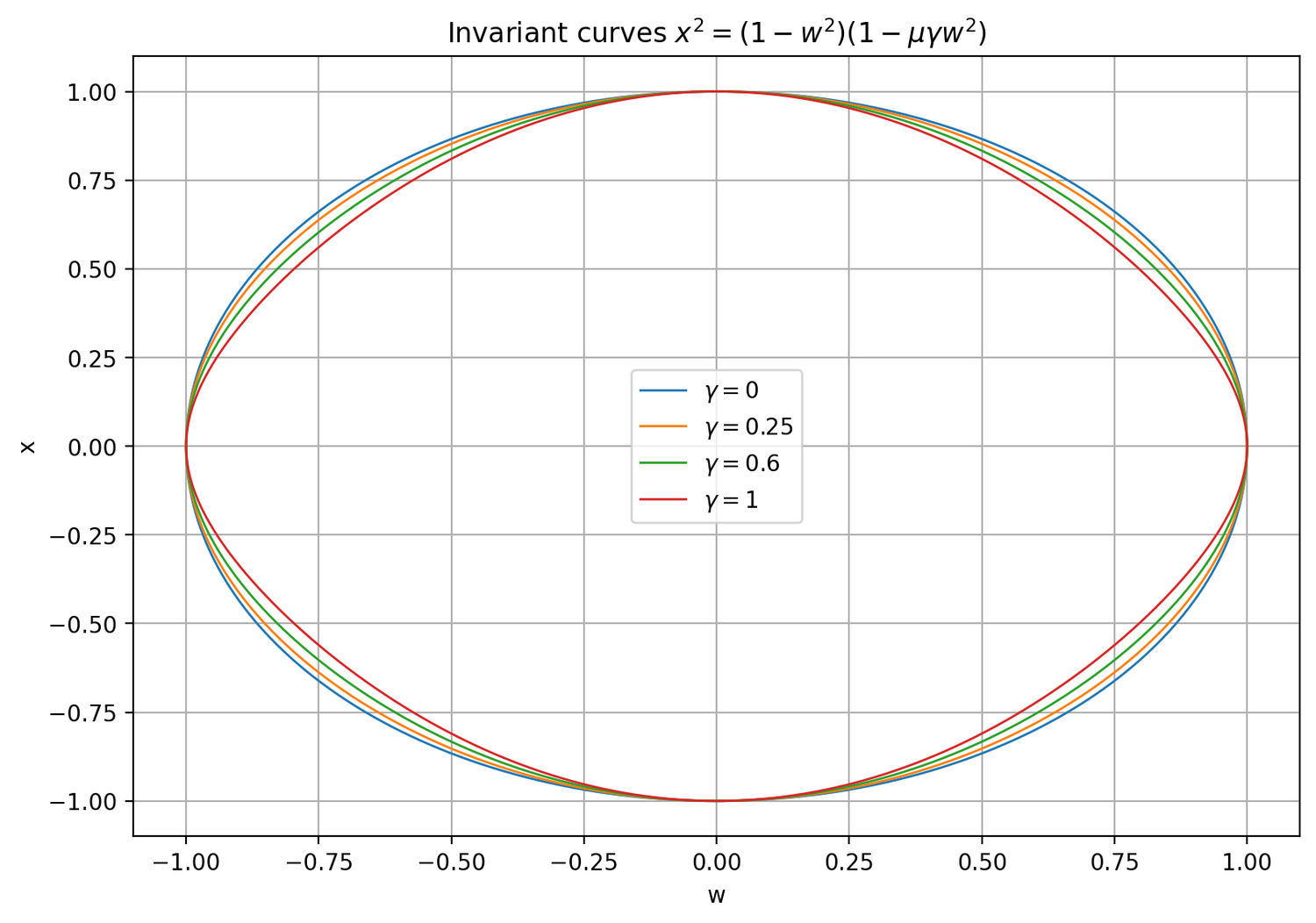}
\caption{Conic-to-elliptic transition.}
\end{subfigure}

\medskip
\begin{subfigure}[t]{0.48\textwidth}
\centering
\includegraphics[width=\textwidth]{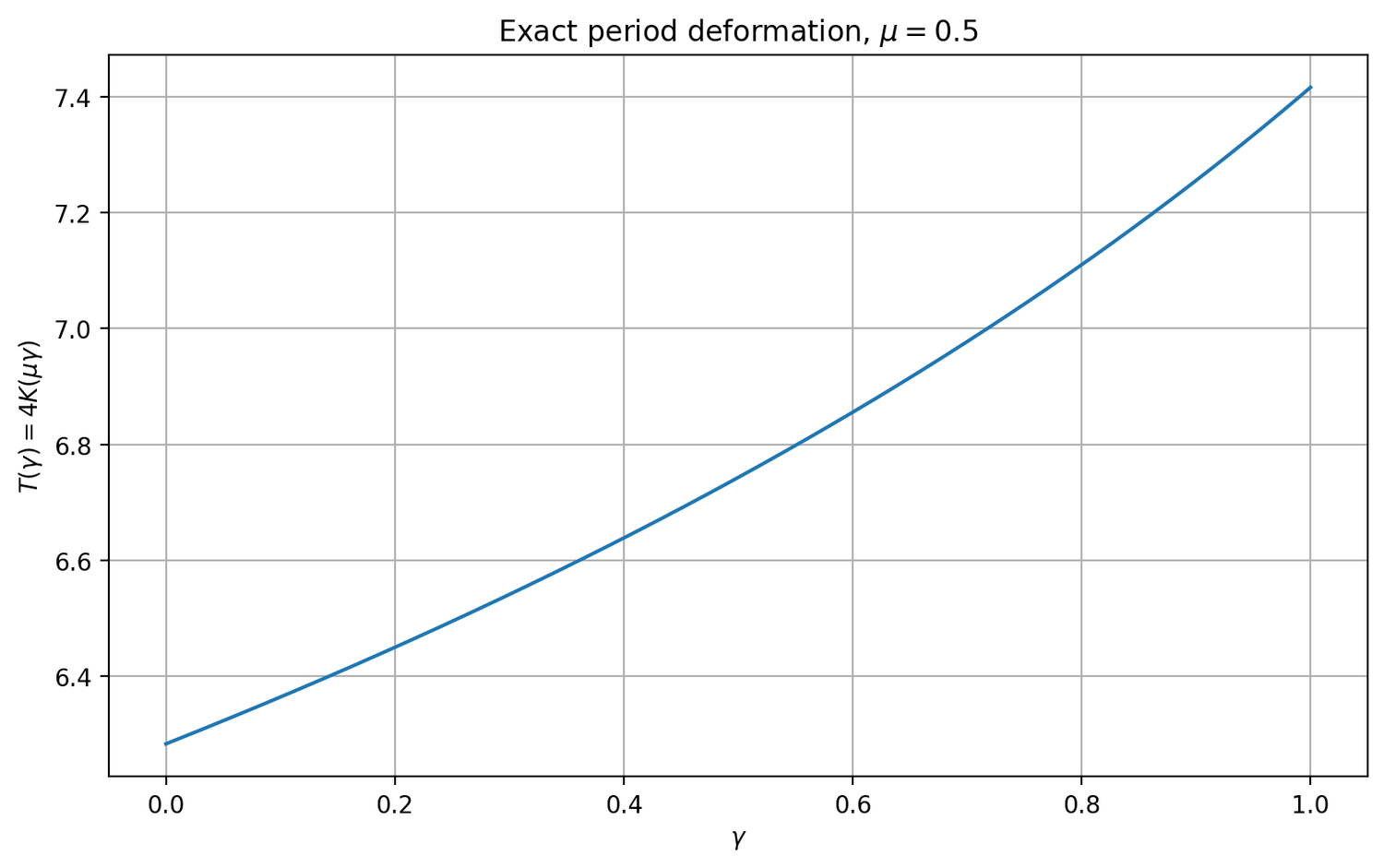}
\caption{Exact period law for $\mu=1/2$.}
\end{subfigure}\hfill
\begin{subfigure}[t]{0.48\textwidth}
\centering
\includegraphics[width=\textwidth]{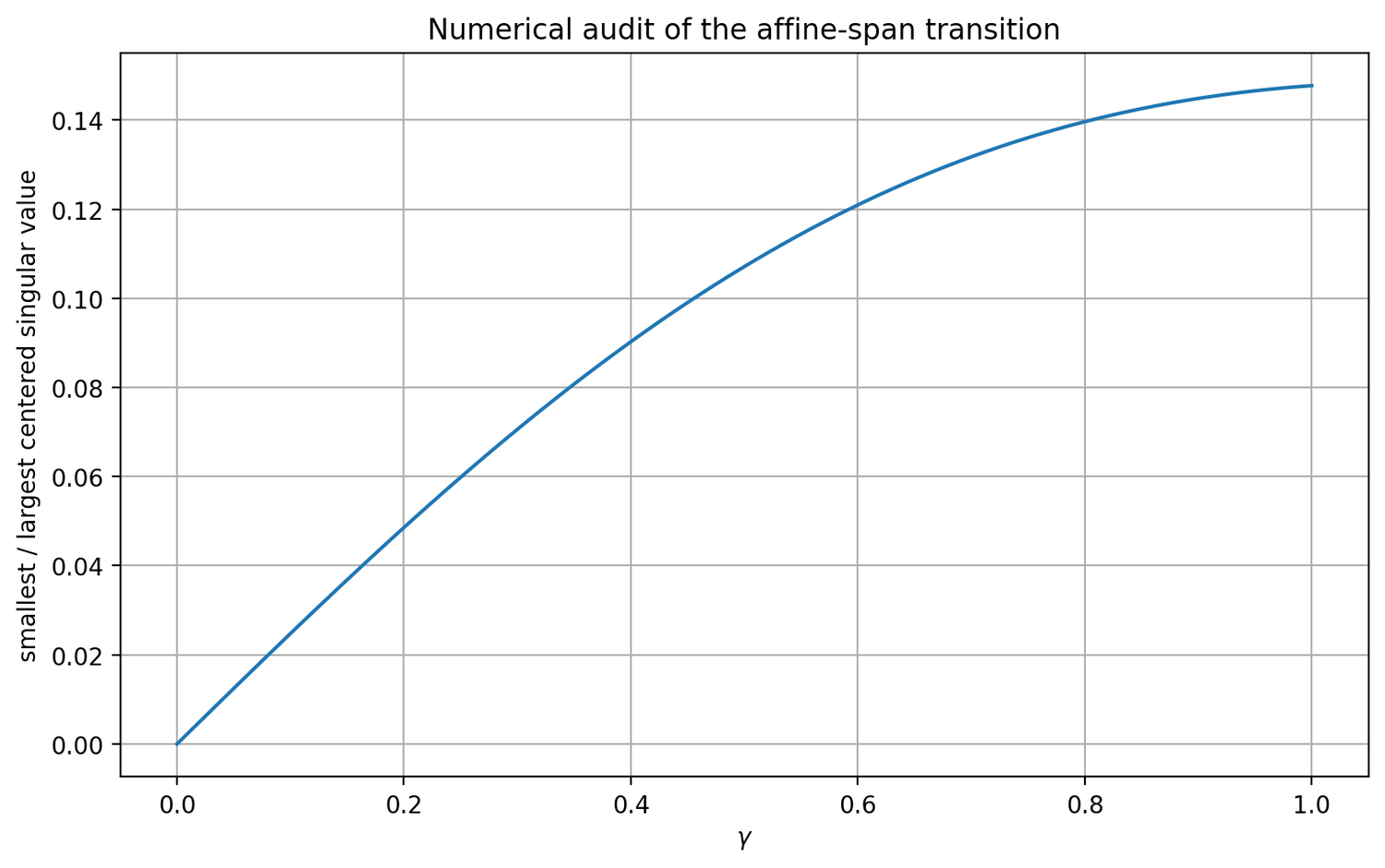}
\caption{Numerical audit of the affine-span jump.}
\end{subfigure}
\caption{
Exact solutions \eqref{eq:flagship-orbit-main} for $\mu=1/2$,
with plotting coordinates
$(w,x,y,z,u)=(p_0,p_1,p_2,p_3,p_4)$.
(a) Orbits over one period in $(w,z,u)$ coordinates for
$\gamma=0,0.25,0.6,1$.
(b) The same orbits in $(w,x)$ coordinates, satisfying
$x^2=(1-w^2)(1-\gamma w^2/2)$ with $|w|\leq1$.
(c) The period $T(\gamma)=4K(\gamma/2)$ versus $\gamma$.
(d) Ratio of the smallest to largest singular value of the
matrix of 1600 equally spaced orbit samples over one period,
after subtracting each coordinate's mean, versus $\gamma$.
The affine dimensions, four at $\gamma=0$ and five for
$\gamma>0$, are proved in
Theorem~\ref{thm:integrable-geometric-transition}.}
\label{fig:exact-integrable-panel}
\end{figure}

\section[A Riccati equation in Agamma]{A Riccati equation in $\mathcal A_\gamma$}
\label{sec:riccati-dynamics}

In \cite{PoshtkohiPumplunSubmitted}\ we investigated a quadratic differential equation in the algebra $\Pent$ for different bracketings. Now we keep the bracketing fixed and vary $\gamma$ (i.e., the algebra $\mathcal A_\gamma$); noting that both the multiplication and the
coefficients of the equations depend on $\gamma$. We first determine analytically when the origin loses stability, and
then study the subsequent dynamics numerically.

\subsection{Model, variational equation and  instability at the origin}

Let
\begin{align}
 A_\gamma&=-\frac12-2i+\frac12j-k-\frac\gamma2\ell,\label{eq:A-gamma-riccati}\\
 B_\gamma&=1+2j-\frac\gamma2\ell,\label{eq:B-gamma-riccati}\\
 p_\gamma(0)&=\frac12(1-i-j+k-\gamma\ell).
 \label{eq:p0-gamma-riccati}
\end{align}
These choices satisfy $A_\gamma=S_\gamma A_1$, $B_\gamma=S_\gamma B_1$, and $p_\gamma(0)=S_\gamma p_1(0)$, with $S_\gamma$ defined in \cref{eq:weighted-state-main}. Consider

\begin{equation}
 \dot p=A_\gamma\star_\gamma p
 +B_\gamma\star_\gamma(p\star_\gamma p).
 \label{eq:riccati-main}
\end{equation}
For $\gamma=1$ this is the equation in $\Pent$ from \cite{PoshtkohiPumplunSubmitted}.

For $\gamma=0$, all elements lie in
the quaternion subalgebra $\mathbb H=\Span\{1,i,j,k\}$ of $\mathcal A_0$, and since $\mathbb H$ is closed under $\star_0$,
the solution stays in $\mathbb H$. Note that varying $\gamma$ changes both the multiplication and the
coefficients.

The  component equations are given in Appendix~\ref{app:coordinate-matrices}.

Write $F_\gamma(p)=A_\gamma\star_\gamma p+B_\gamma\star_\gamma(p\star_\gamma p)$. Differentiating with respect to $p$, with $A_\gamma$, $B_\gamma$, and $\gamma$ held fixed, and using the bilinearity of $\star_\gamma$, gives
\begin{equation}
 D_pF_\gamma(p)[h]
 =A_\gamma\star_\gamma h
 +B_\gamma\star_\gamma(p\star_\gamma h+h\star_\gamma p).
 \label{eq:riccati-state-derivative}
\end{equation}
Thus a perturbation $h(t)$ along a solution $p(t)$ satisfies $\dot h=D_pF_\gamma(p(t))h$.
\begin{equation}
 \dot h=A_\gamma\star_\gamma h
 +B_\gamma\star_\gamma
 \bigl(p\star_\gamma h+h\star_\gamma p\bigr).
 \label{eq:riccati-tangent}
\end{equation}
Equivalently,
\begin{equation}
 D_pF_\gamma(p)=L_{A_\gamma}^{\gamma}
   +L_{B_\gamma}^{\gamma}(L_p^\gamma+R_p^\gamma),
 \label{eq:riccati-jacobian-matrix}
\end{equation}
where products of linear maps denote composition.

For prescribed $\gamma(t)$, the same state derivative is evaluated at the instantaneous value of the parameter.

For a nonconstant periodic solution $p_*(t)$ with period $T$,  let $\Phi(t)$ solve
$\dot\Phi=D_pF_\gamma(p_*(t))\Phi$, $\Phi(0)=I_5$.

The eigenvalues of the monodromy matrix $\Phi(T)$ are the Floquet multipliers.
 One multiplier is $1$, corresponding to the direction along the orbit. If the other four, called \emph{transverse multipliers}, have modulus less than $1$, the orbit is orbitally asymptotically stable.
These are the multipliers used in the continuation diagnostics below.

At the origin, the quadratic term vanishes and
\[
 DF_\gamma(0)=L_{A_\gamma}^{\gamma}.
\]
Substituting $A_\gamma$ from \eqref{eq:A-gamma-riccati}
into the multiplication matrix
\eqref{eq:left-multiplication-matrix}, an exact symbolic
calculation (see Appendix~\ref{app:reproducibility} for the
verification code) shows that the characteristic polynomial
$\chi_\gamma(\lambda)=\det(\lambda I_5-L_{A_\gamma}^{\gamma})$
factors as
\begin{equation}
 \chi_\gamma(\lambda)
 =\left(\lambda^2+\lambda+\frac{22+\gamma^2}{4}\right)
\left(\lambda^3+\frac32\lambda^2+b_\gamma\lambda+c_\gamma\right),
 \label{eq:riccati-characteristic}
\end{equation}
where
\begin{equation}
 b_\gamma=
 \frac{3(16+\gamma^2-\gamma^4)}{8(1+\gamma^2)},
 \qquad
 c_\gamma=
 \frac{44-\gamma^2-3\gamma^4
 +\sqrt2\gamma(\gamma^4+\gamma^2+194)}{16(1+\gamma^2)}.
 \label{eq:riccati-bc}
\end{equation}
The quadratic factor has roots  $\lambda=\frac12(-1\pm i\sqrt{21+\gamma^2})$ so that  $\operatorname{Re}\lambda=-\frac12<0$. The coefficients $b_\gamma$ and $c_\gamma$ are positive on $[0,1]$, so by the Routh-Hurwitz criterion all the roots of the cubic factor have negative real part if and only if $\frac32b_\gamma-c_\gamma>0$.
The stability boundary hence is $\tfrac32b_\gamma-c_\gamma=0$, or equivalently,
\begin{equation}
 H(\gamma)=100+10\gamma^2-6\gamma^4
 -\sqrt2(\gamma^5+\gamma^3+194\gamma)=0.
 \label{eq:hopf-condition}
\end{equation}

\begin{theorem}[Unique origin instability]
\label{thm:riccati-hopf}
Equation~\eqref{eq:hopf-condition} has a unique root $\gamma_H\in[0,1]$, numerically given by
\begin{equation}
 \gamma_H=0.36874617836684\ldots.
 \label{eq:gamma-H}
\end{equation}
Let $\lambda_1(\gamma),\dots,\lambda_5(\gamma)$ be the roots of $\chi_\gamma$. Then
 $\max_r\operatorname{Re}\lambda_r(\gamma)<0$ for $0\leq\gamma<\gamma_H$, and $\max_r\operatorname{Re}\lambda_r(\gamma)>0$ for
 $\gamma_H<\gamma\le1$.
 At $\gamma=\gamma_H$, the roots are $\pm i\omega_H$ with $\omega_H=\sqrt{b_{\gamma_H}}$,
 $-\tfrac32$ and $\tfrac12(-1\pm i\sqrt{21+\gamma_H^2})$;
and $\lambda(\gamma)$ with $\lambda(\gamma_H)=i\omega_H$ satisfies
$ \left.\frac{\dd\operatorname{Re}\lambda}{\dd\gamma}\right|_{\gamma_H}>0$.
\end{theorem}

In particular, the origin is asymptotically stable for $0\leq\gamma<\gamma_H$ and unstable for $\gamma_H<\gamma\leq 1$.

\begin{proof}
Set $D(\gamma)=\tfrac32b_\gamma-c_\gamma$.
Substituting \eqref{eq:riccati-bc} gives
\[
16(1+\gamma^2)D(\gamma)
=9(16+\gamma^2-\gamma^4)-(44-\gamma^2-3\gamma^4)
-\sqrt2\,\gamma(\gamma^4+\gamma^2+194)
=H(\gamma).
\]
Thus $D(\gamma)=H(\gamma)/(16(1+\gamma^2))$.
Moreover, $H(0)=100>0$, $H(1)=104-196\sqrt2<0$, and
\[
H'(\gamma)=20\gamma-24\gamma^3
-\sqrt2(5\gamma^4+3\gamma^2+194)
\leq20-194\sqrt2<0
\]
on $[0,1]$.
Hence $H$ has exactly one zero $\gamma_H\in(0,1)$.
Since $16(1+\gamma^2)>0$, the same holds for $D$, with
$D>0$ on $[0,\gamma_H)$ and $D<0$ on $(\gamma_H,1]$.

The quadratic factor of $\chi_\gamma$ has the roots
$\tfrac12(-1\pm i\sqrt{21+\gamma^2})$, with real part $-\tfrac12$.
For the cubic factor
$\lambda^3+\tfrac32\lambda^2+b_\gamma\lambda+c_\gamma$,
we have $b_\gamma,c_\gamma>0$ on $[0,1]$, so by the
Routh--Hurwitz criterion all its roots have negative real part
if and only if $\tfrac32b_\gamma>c_\gamma$, i.e.\
$D(\gamma)>0$; if $D(\gamma)<0$, the
cubic factor has a root with
positive real part.

At $D(\gamma_H)=0$, we have
$c_{\gamma_H}=\tfrac32b_{\gamma_H}$, and the cubic factors as
\[
(\lambda+\tfrac32)(\lambda^2+b_{\gamma_H}).
\]
Its roots are $-\tfrac32$ and $\pm i\omega_H$, with
$\omega_H=\sqrt{b_{\gamma_H}}>0$.
They differ from each other and from the roots of the quadratic
factor, which have real part $-\tfrac12$, so all five roots of
$\chi_{\gamma_H}$ are simple.

Since $i\omega_H$ is a simple root, it continues to a
differentiable root $\lambda(\gamma)$ of the cubic
$p(\lambda,\gamma)=\lambda^3+\tfrac32\lambda^2
+b_\gamma\lambda+c_\gamma$ near $\gamma_H$, and
\[
\lambda'
=-\frac{\partial_\gamma p}{\partial_\lambda p}
=-\frac{b_\gamma'\lambda+c_\gamma'}
{3\lambda^2+3\lambda+b_\gamma}.
\]
Write $b_H=b_{\gamma_H}$,
$b'_H=\left.\frac{\dd b_\gamma}{\dd\gamma}\right|_{\gamma_H}$
and
$c'_H=\left.\frac{\dd c_\gamma}{\dd\gamma}\right|_{\gamma_H}$.
At $\lambda=i\omega_H$, using $\omega_H^2=b_H$, we obtain
\[
\lambda'(\gamma_H)
=-\frac{c'_H+i\omega_Hb'_H}{-2b_H+3i\omega_H}.
\]
Multiplying numerator and denominator by
$-2b_H-3i\omega_H$ and taking real parts gives
\[
\operatorname{Re}\lambda'(\gamma_H)
=\frac{2b_Hc'_H-3\omega_H^2b'_H}{4b_H^2+9\omega_H^2}
=\frac{2c'_H-3b'_H}{4b_H+9}
=-\frac{D'(\gamma_H)}{2(b_H+9/4)},
\]
where the last equality uses
$D'(\gamma_H)=\tfrac32b'_H-c'_H$.

Since $H(\gamma_H)=0$, we have
\[
D'(\gamma_H)
=\frac{H'(\gamma_H)}{16(1+\gamma_H^2)}<0,
\]
hence
\[
\left.\frac{\dd\operatorname{Re}\lambda}{\dd\gamma}
\right|_{\gamma_H}
=-\frac{D'(\gamma_H)}{2(b_{\gamma_H}+9/4)}>0.
\]
This proves transversality and the direction of the crossing.
\end{proof}

A separate numerical calculation gives the crossing frequency
$\omega_H\approx2.306641614$ and the first Lyapunov coefficient
$l_1\approx-0.18325828$, with the convention stated in
Appendix~\ref{app:riccati-validation}.
The negative numerical value of $l_1$ is consistent with a
supercritical Hopf bifurcation and with the numerically
computed stable primary periodic branch shown in
Figure~\ref{fig:primary-periods}.
The stability argument above is analytical; the sign of $l_1$
has not been certified by interval arithmetic.

\subsection{Fixed $\gamma$: the primary period-doubling sequence}

For each fixed value of $\gamma$, we integrate Equation~\eqref{eq:riccati-main}
with the prescribed initial condition $p_\gamma(0)$. The largest Lyapunov exponent at $201$ values of $\gamma$, the local maxima of $p_3$
in Figure~\ref{fig:peak} on a finer grid of $\gamma$, and the numerical
continuation of periodic orbits are separate computations.
Their methods and horizons are specified in Appendix~\ref{app:numerics}.
Together they give numerical evidence for the sequence
\[
 \text{stable equilibrium}\ \longrightarrow\ \text{periodic motion}
 \ \longrightarrow\ \text{period doubling and irregular dynamics},
\]
followed by
intervals of periodic behaviour and further intervals of $\gamma$ with a
positive largest Lyapunov exponent.
These computations describe the
solutions starting at $p_\gamma(0)$, not all coexisting attractors.

\begin{figure}[t]
\centering
\begin{subfigure}[t]{.48\textwidth}
\centering
\includegraphics[width=\textwidth]{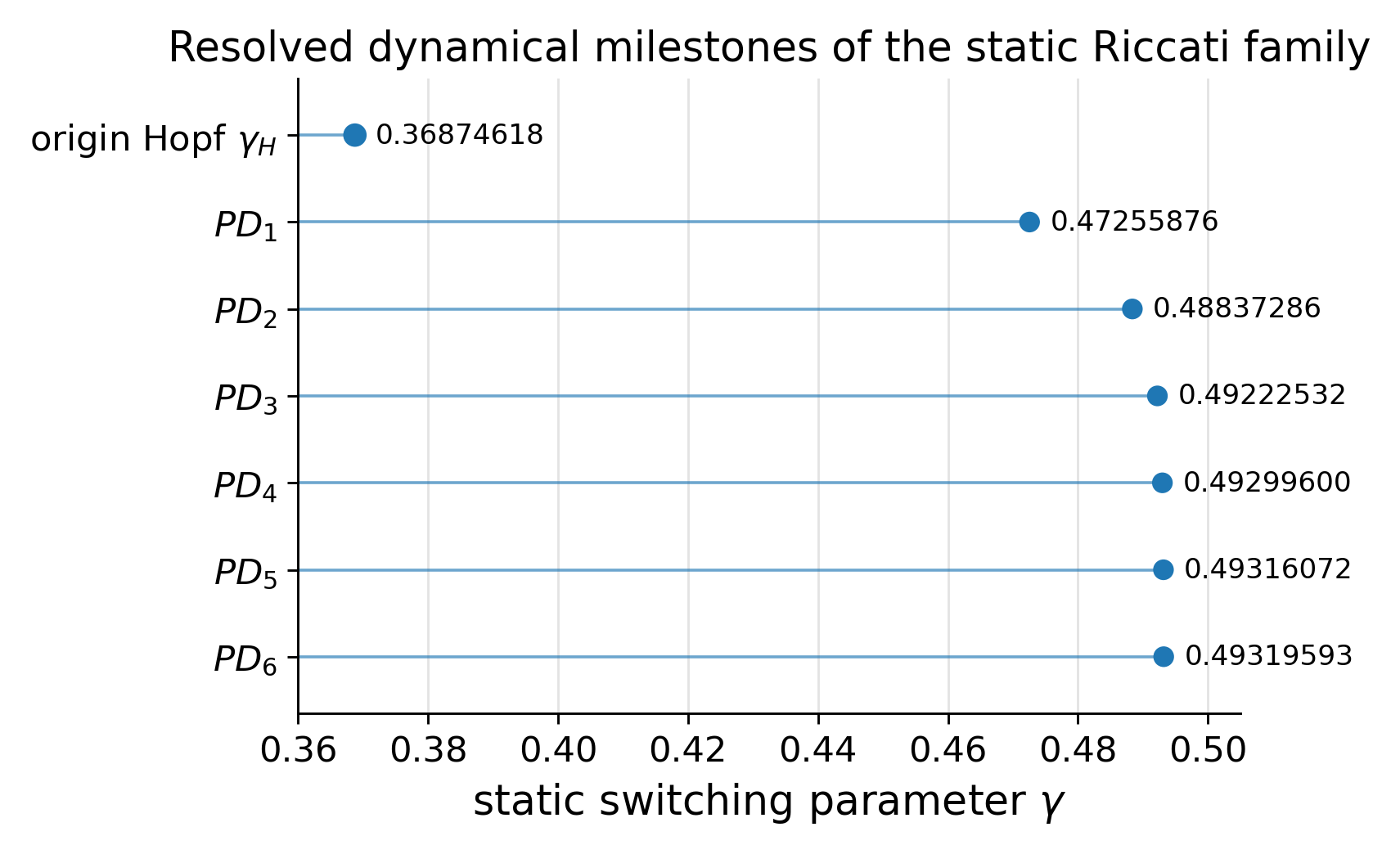}
\caption{Origin-instability threshold and six estimated period-doubling transitions.}
\label{fig:riccati-event-timeline}
\end{subfigure}
\hfill
\begin{subfigure}[t]{.48\textwidth}
\centering
\includegraphics[width=\textwidth]{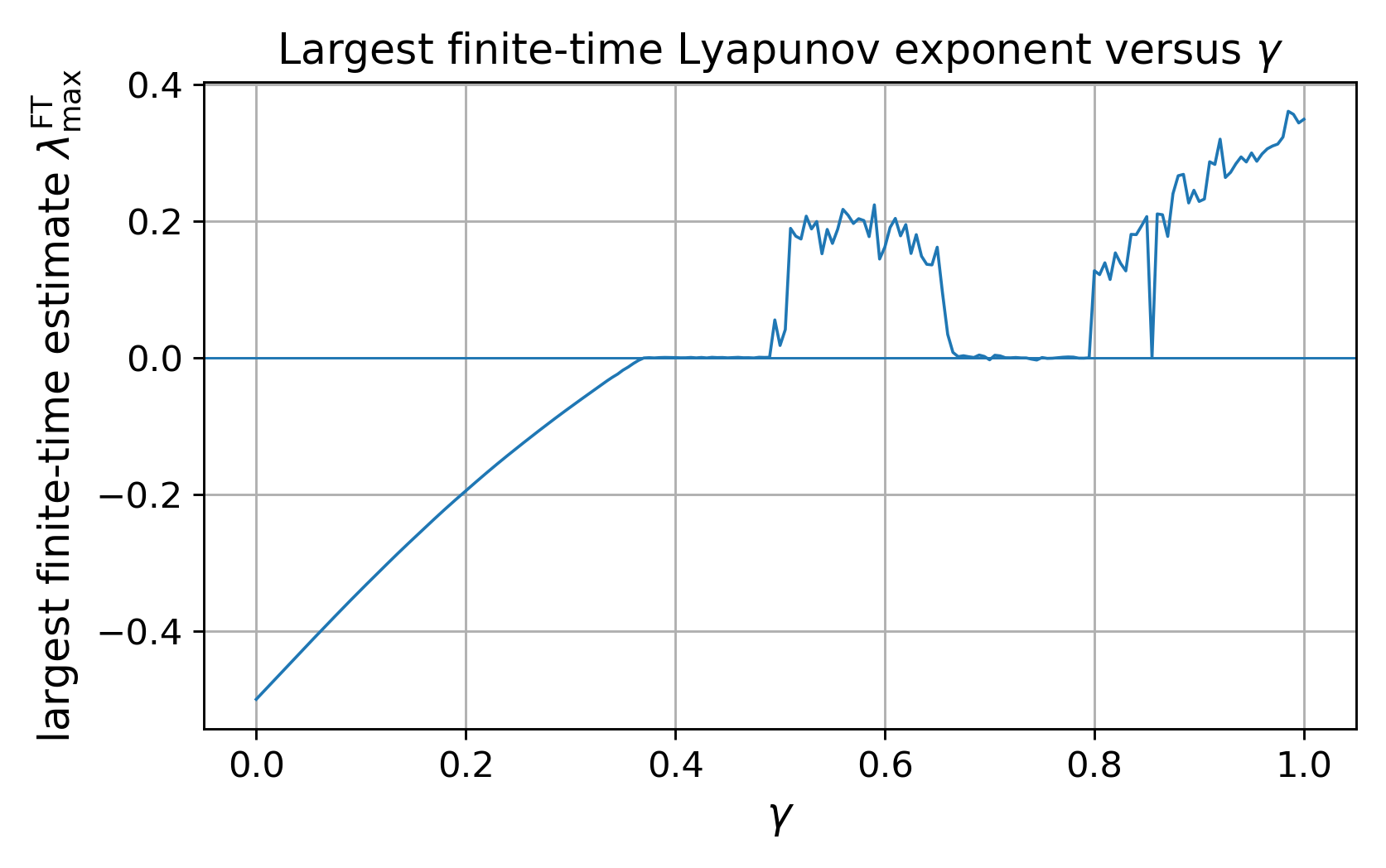}
\caption{Largest finite-time Lyapunov exponent estimate.}
\label{fig:riccati-lle}
\end{subfigure}
\caption{Fixed $\gamma$. (A) The threshold $\gamma_H$ of Theorem~\ref{thm:riccati-hopf} and six estimated
period-doubling values of $\gamma$. (B) Estimated largest Lyapunov exponent. 
The threshold $\gamma_H$ is characterised analytically; the
period-doubling values are obtained by interpolation and, like the exponents, are numerical estimates.
}
\end{figure}

\begin{figure}[t]
\centering
\begin{subfigure}[t]{.94\textwidth}
\centering
\includegraphics[width=\textwidth]{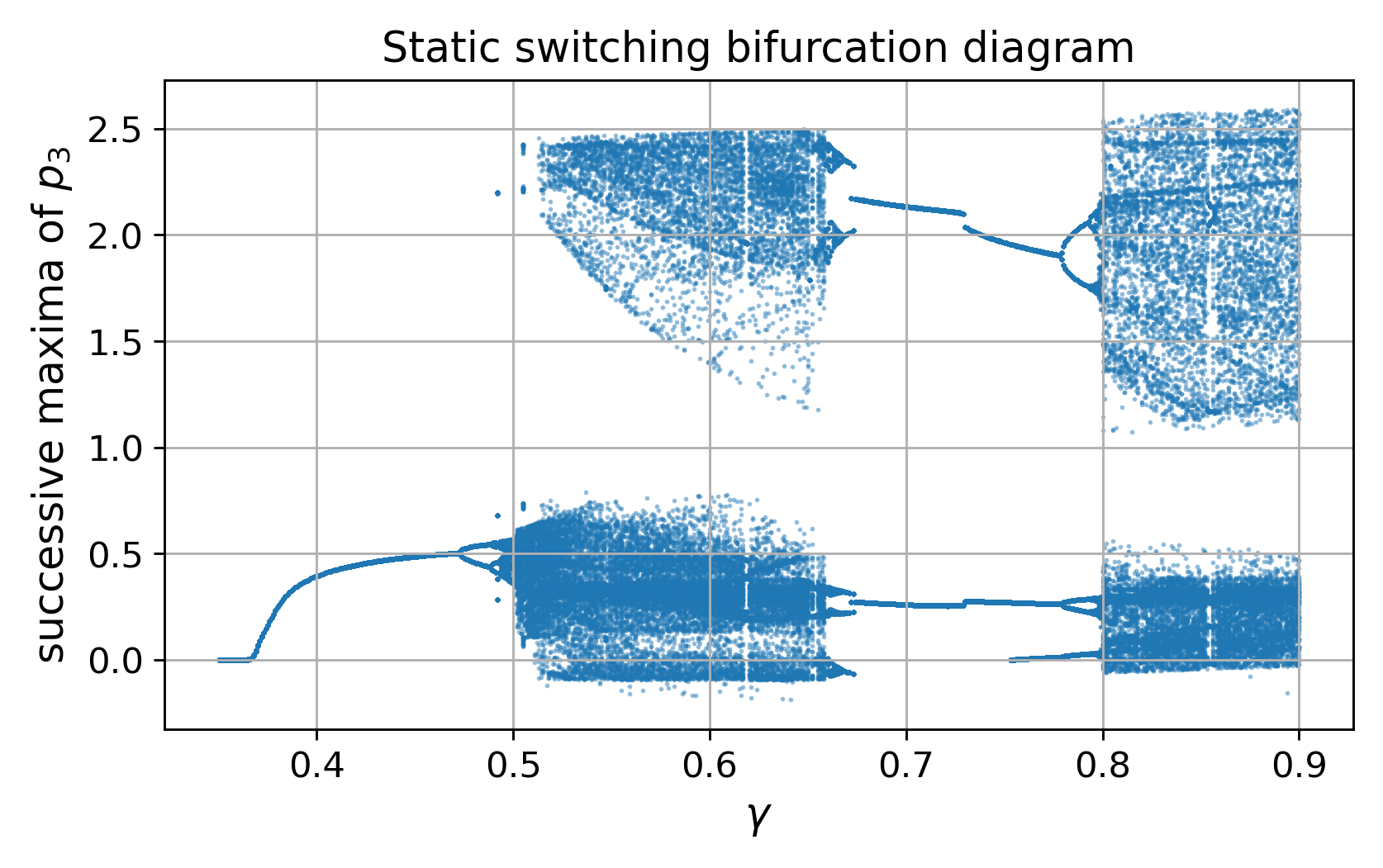}
\caption{Post-transient local maxima of $p_3$ on the global parameter grid.}
\label{fig:riccati-bifurcation}
\end{subfigure}
\caption{The local maxima of $p_3$  identify broad irregular and regular intervals. Its fixed grid does not resolve the narrowest period-doubling intervals; these are examined using continuation.}
\label{fig:peak}
\end{figure}

We denote the branches of periodic orbits obtained by
continuation by $P_1,P_2,P_4,\dots,P_{32}$, where $P_{2n}$ arises from $P_n$ by period doubling.
Six numerically detected crossings of a real transverse multiplier through $-1$ are consistent with
\begin{equation}
 P_1\to P_2\to P_4\to P_8\to P_{16}\to P_{32}\to P_{64}.
 \label{eq:primary-cascade}
\end{equation}
The last arrow denotes the crossing detected on the branch $P_{32}$; the $P_{64}$  branch has not been independently continued. The estimated transition parameters and  the periods of the orbits $P_1,\dots,P_{32}$ at the crossings
are shown in \cref{tab:primary-flip-estimates}.

\begin{table}[t]
\centering
\caption{Estimated  period-doubling parameters from adjacent Floquet calculations. The last column gives the separation of the sampled parameter values on either side of the selected multiplier crossing, not a certified error bound.}
\label{tab:primary-flip-estimates}
{\small
\begin{tabular}{lrrr}
\toprule
Parent & $\gamma$ estimate & parent
period & sample separation\\
\midrule
$P_1$ & 0.47255876 & 3.207809 & $2\times10^{-5}$\\
$P_2$ & 0.48837286 & 6.467492 & $1\times10^{-4}$\\
$P_4$ & 0.49222532 & 12.973525 & $5\times10^{-5}$\\
$P_8$ & 0.49299600 & 25.960095 & $2\times10^{-5}$\\
$P_{16}$ & 0.49316072 & 51.926332 & $5\times10^{-6}$\\
$P_{32}$ & 0.49319593 & 103.855145 & $1\times10^{-7}$\\
\bottomrule
\end{tabular}}
\end{table}

The transition parameters are estimated by secant interpolation of the computed real transverse multiplier nearest $-1$. This interpolation is not an independent shooting refinement, and the multiplier selection need not follow a single eigenvalue continuously. The late ratios of successive parameter intervals are close to Feigenbaum's constant $\delta\approx4.669$, with an extrapolated accumulation parameter $\gamma_\infty\approx0.4932055$. This is evidence of Feigenbaum-like numerical scaling, not a universality theorem or a certified accumulation point.

\begin{figure}[t]
\centering
\begin{subfigure}[t]{.48\textwidth}
\centering
\includegraphics[width=\textwidth]{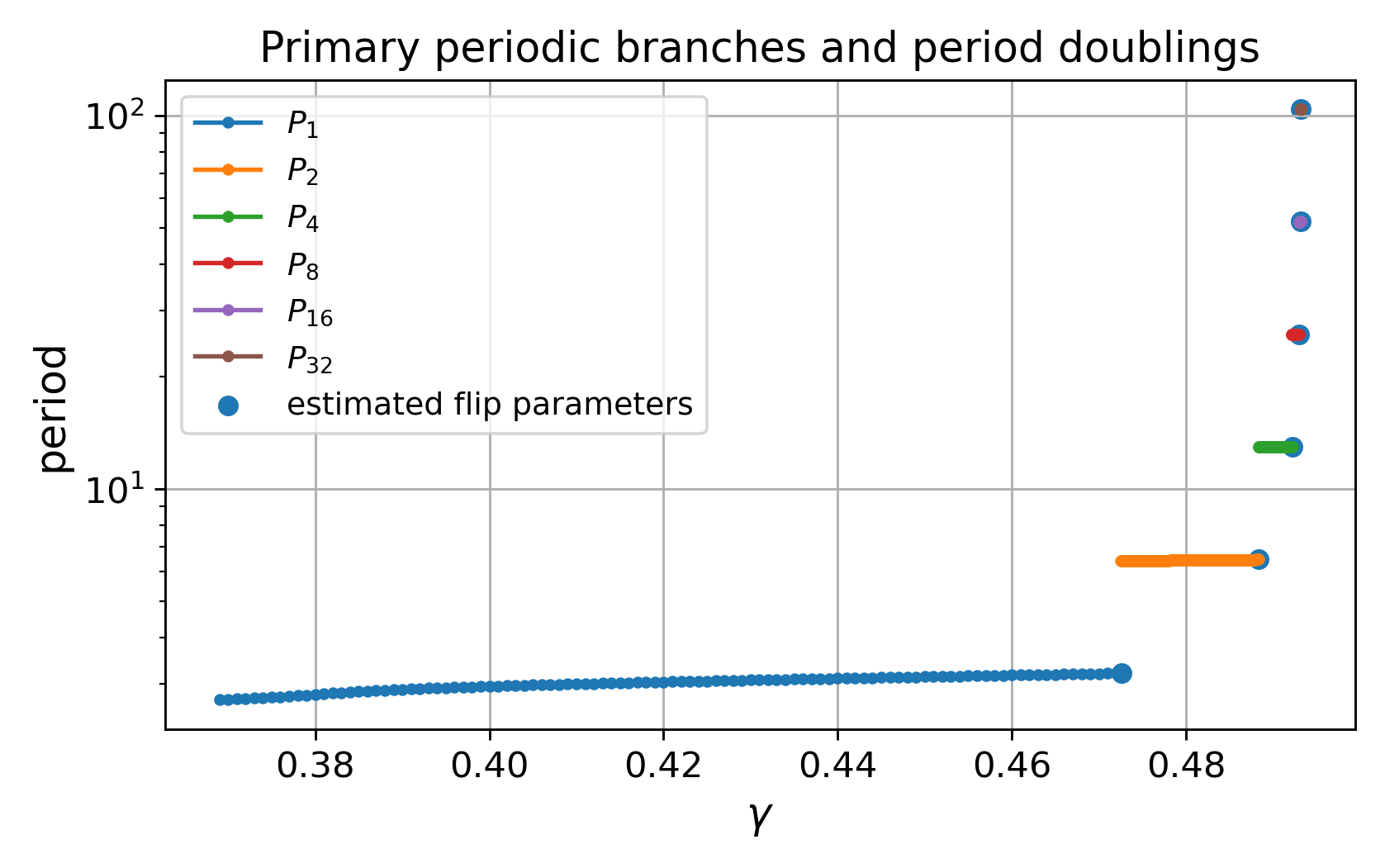}
\caption{Periods along the stable primary branches and estimated period-doubling parameters.}
\label{fig:primary-periods}
\end{subfigure}
\hfill
\begin{subfigure}[t]{.48\textwidth}
\centering
\includegraphics[width=\textwidth]{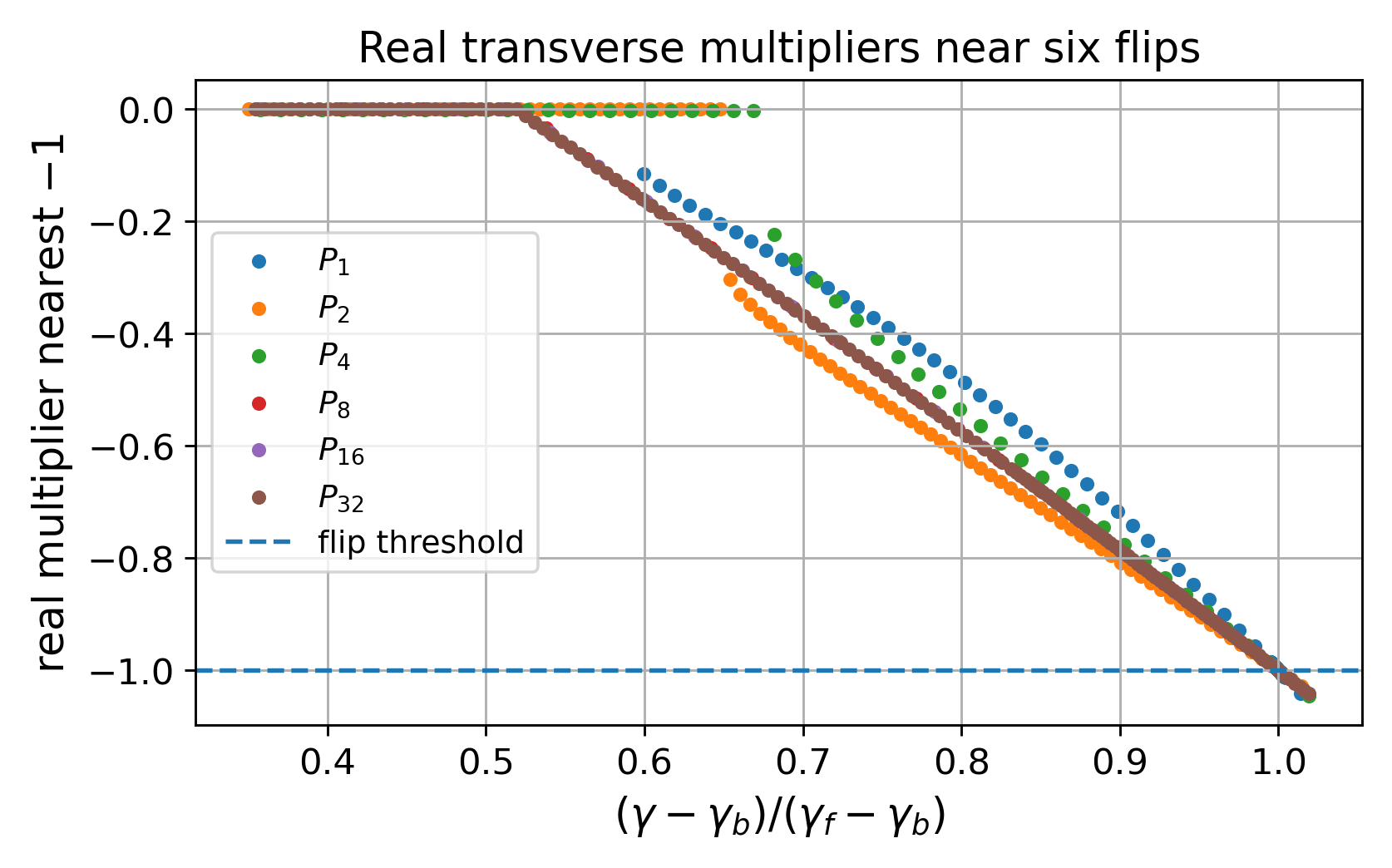}
\caption{At each sample, the real transverse multiplier nearest $-1$.}
\label{fig:primary-flips}
\end{subfigure}
\caption{Primary continuation diagnostics. In the multiplier panel, the horizontal coordinate is $(\gamma-\gamma_b)/(\gamma_f-\gamma_b)$, where $\gamma_b$ is the estimated value of $\gamma$ at which the branch arises
and $\gamma_f$  the  next estimated  period-doubling value. The plotted selection need not track a single eigenvalue continuously.}
\end{figure}

\begin{figure}[t]
\centering
\begin{subfigure}[t]{.90\textwidth}
\centering
\includegraphics[width=\textwidth]{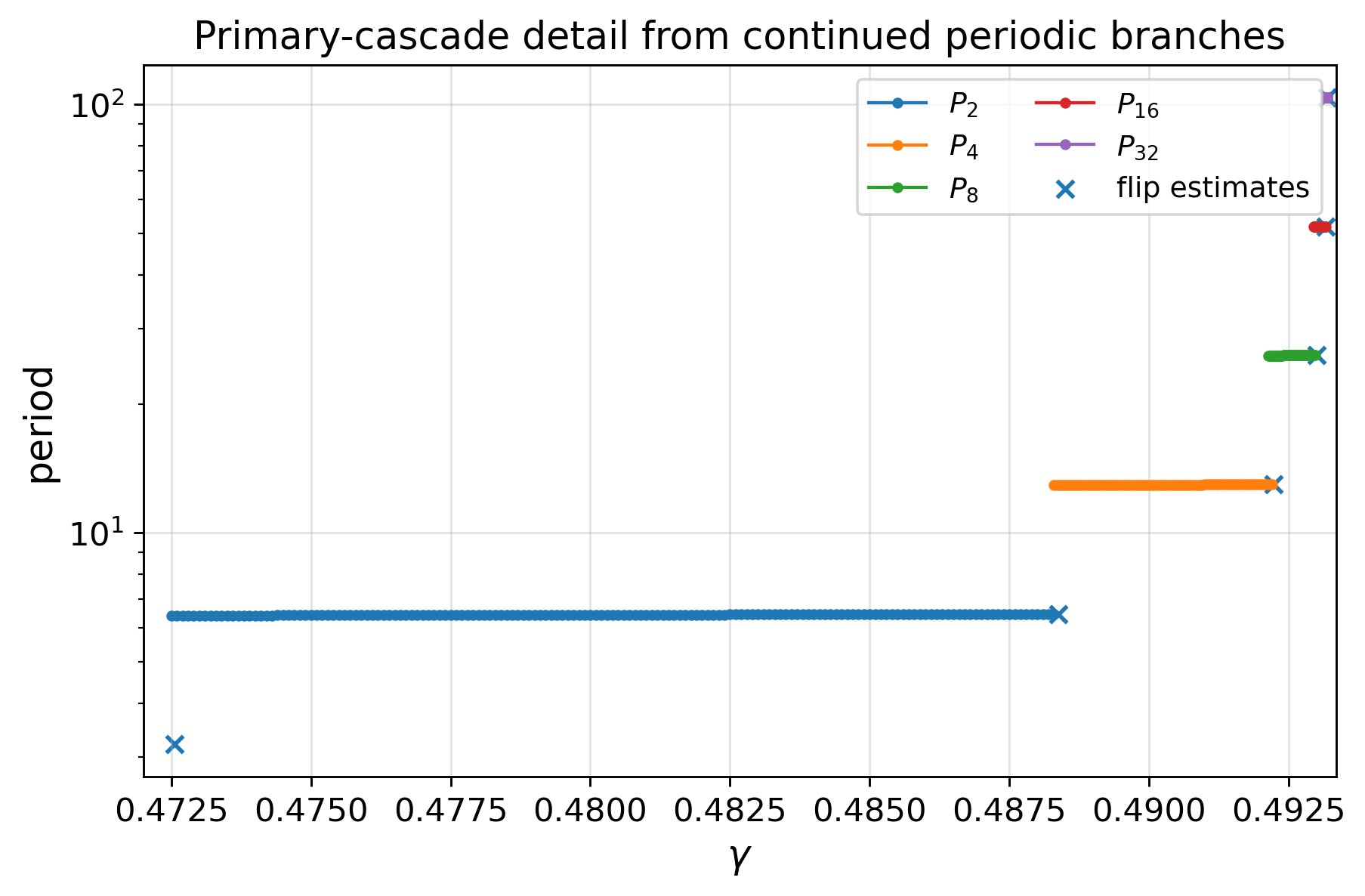}
\caption{Continuation detail on the narrow primary-cascade parameter interval.}
\label{fig:primary-cascade-zoom}
\end{subfigure}
\caption{An enlargement based on periodic-orbit continuation, rather than interpolation of the global peak diagram. The computed  branches extend through $P_{32}$.}
\end{figure}

\begin{figure}[t]
\centering
\begin{subfigure}[t]{.48\textwidth}
\centering
\includegraphics[width=\textwidth]{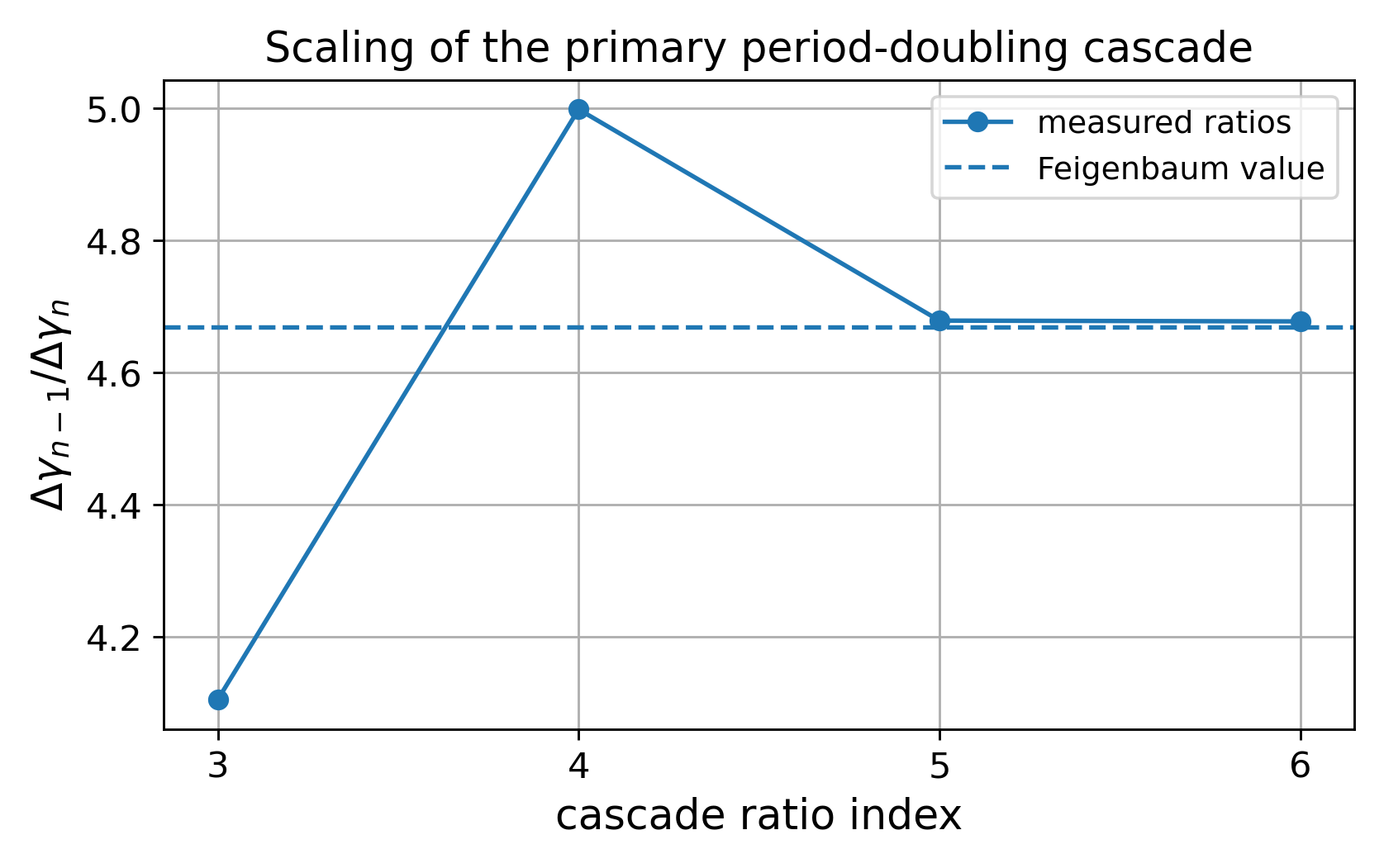}
\caption{Ratios of successive period-doubling parameter intervals.}
\label{fig:cascade-scaling}
\end{subfigure}
\hfill
\begin{subfigure}[t]{.48\textwidth}
\centering
\includegraphics[width=\textwidth]{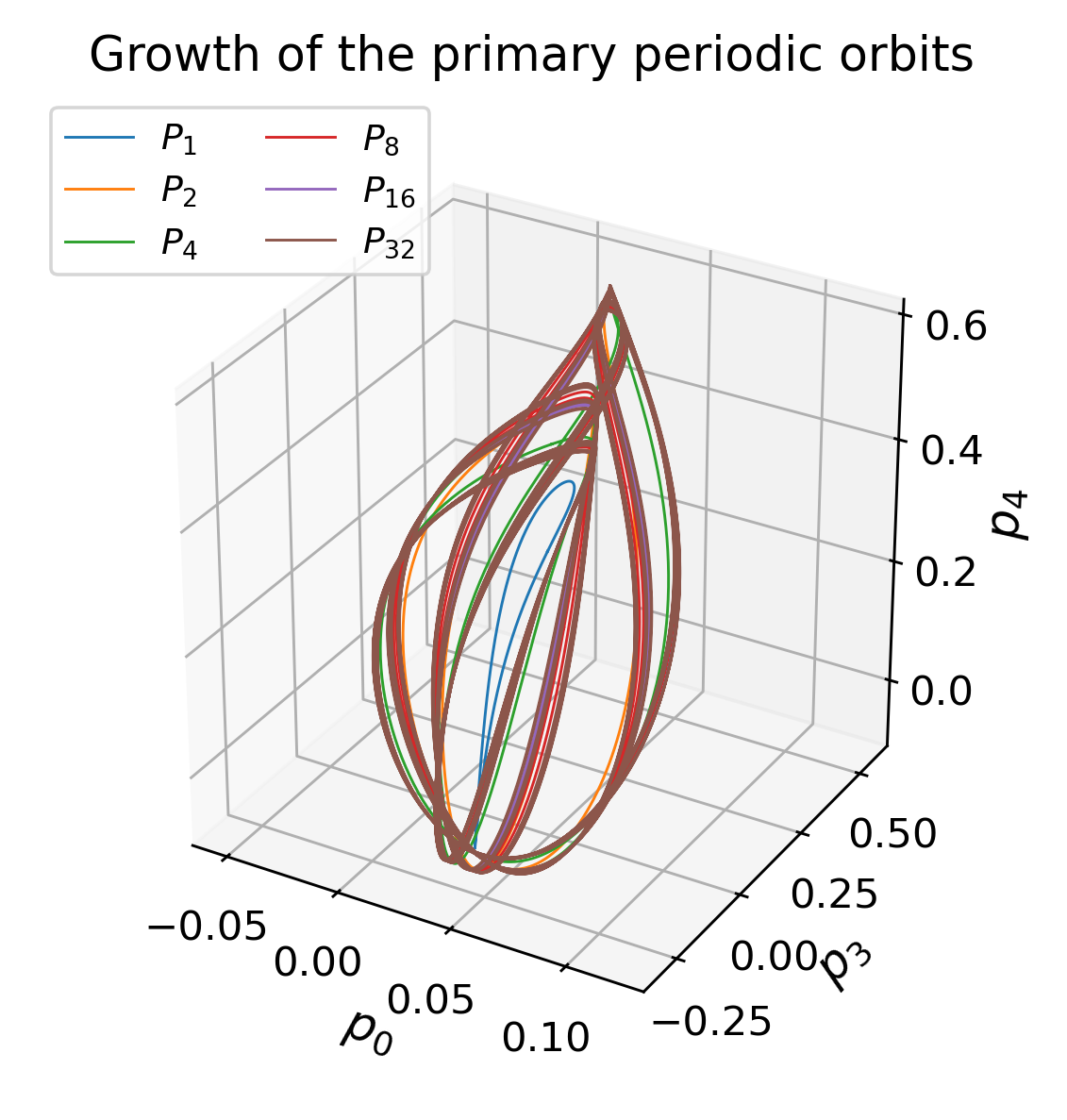}
\caption{Representative periodic orbits on the primary branches.}
\label{fig:primary-orbits}
\end{subfigure}
\caption{Numerical scaling and projected orbit geometry of the primary branches. The reference scaling value does not establish universality for this system.}
\end{figure}

\subsection{Periodic windows, coexisting cycles and representative trajectories}

A separate  branch of periodic orbits
 spans a \emph{periodic window},
i.e. an interval of $\gamma$ on which the observed dynamics is periodic, and is numerically stable on
two subintervals of $\gamma$.
We refer to it as the broad-window branch. Its endpoints are presumably fold bifurcations, inferred from continuation in $\gamma$;
continuation through these folds
has not been performed. Another branch of periodic
orbits and the branches arising from it by period doubling
form a further cascade in this window, while another narrow periodic window occurs near $\gamma=0.855$. The computed endpoints do not establish that the broad-window branch is an \emph{isola}, i.e. a closed curve of periodic
orbits.

\begin{figure}[t]
\centering
\begin{subfigure}[t]{.48\textwidth}
\centering
\includegraphics[width=\textwidth]{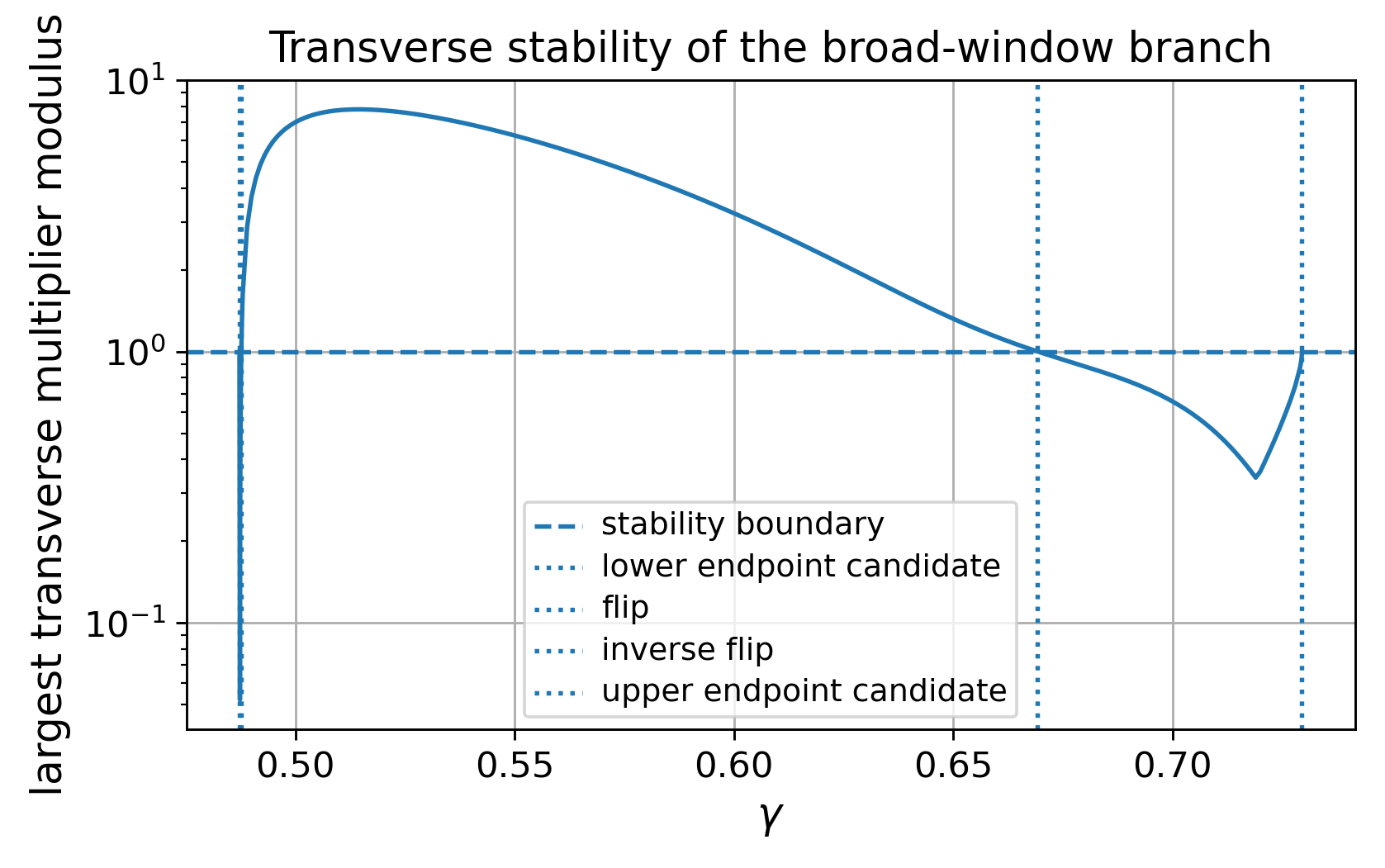}
\caption{Transverse stability of the broad-window branch.}
\label{fig:broad-p2}
\end{subfigure}
\hfill
\begin{subfigure}[t]{.48\textwidth}
\centering
\includegraphics[width=\textwidth]{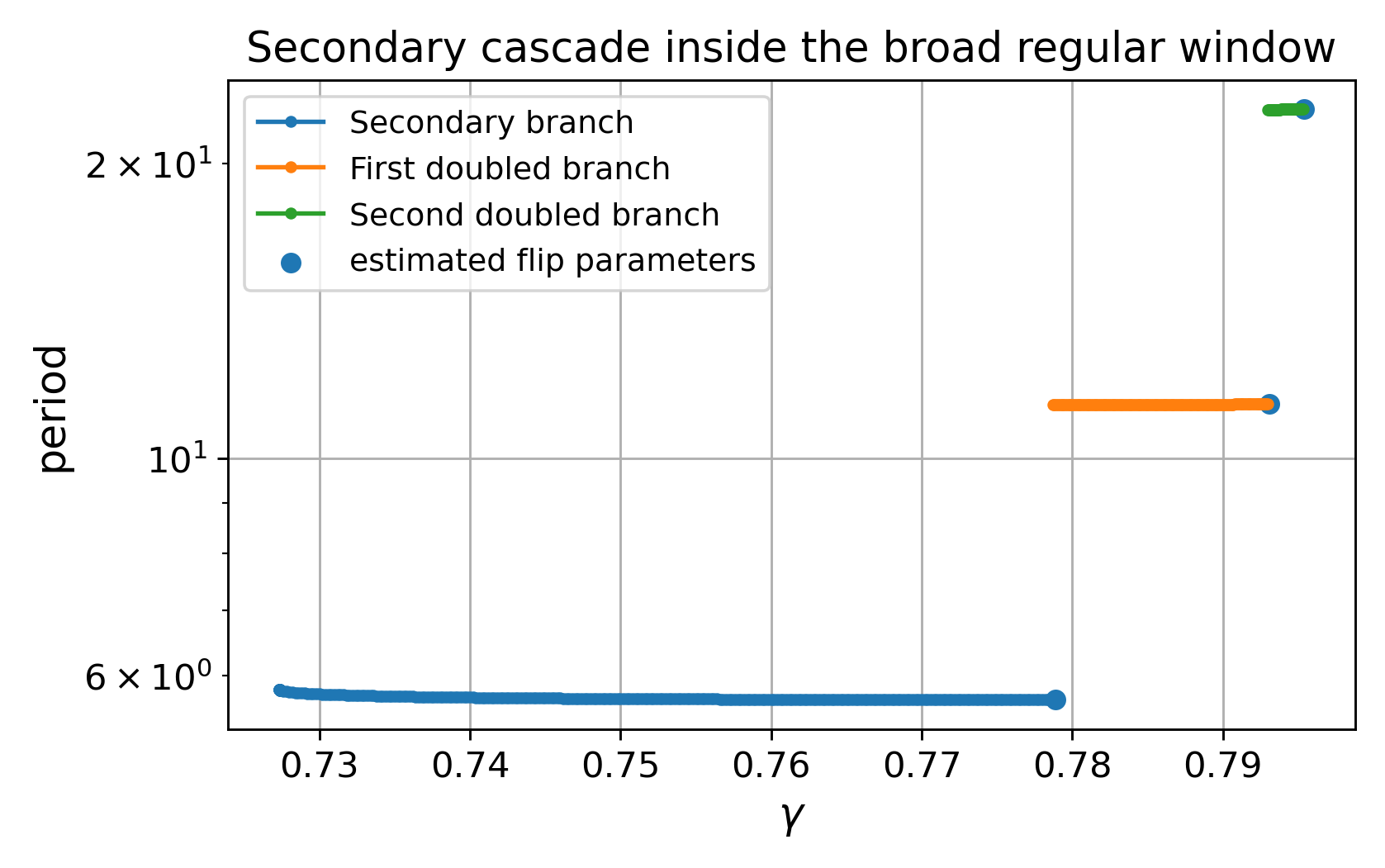}
\caption{Secondary doubling sequence in the regular window.}
\label{fig:window-cascade}
\end{subfigure}
\caption{Additional periodic families. Stability is evaluated from the computed transverse Floquet multipliers; the endpoint markers in the broad-window panel indicate presumed folds.}
\end{figure}

\begin{figure}[t]
\centering
\begin{subfigure}[t]{.48\textwidth}
\centering
\includegraphics[width=\textwidth]{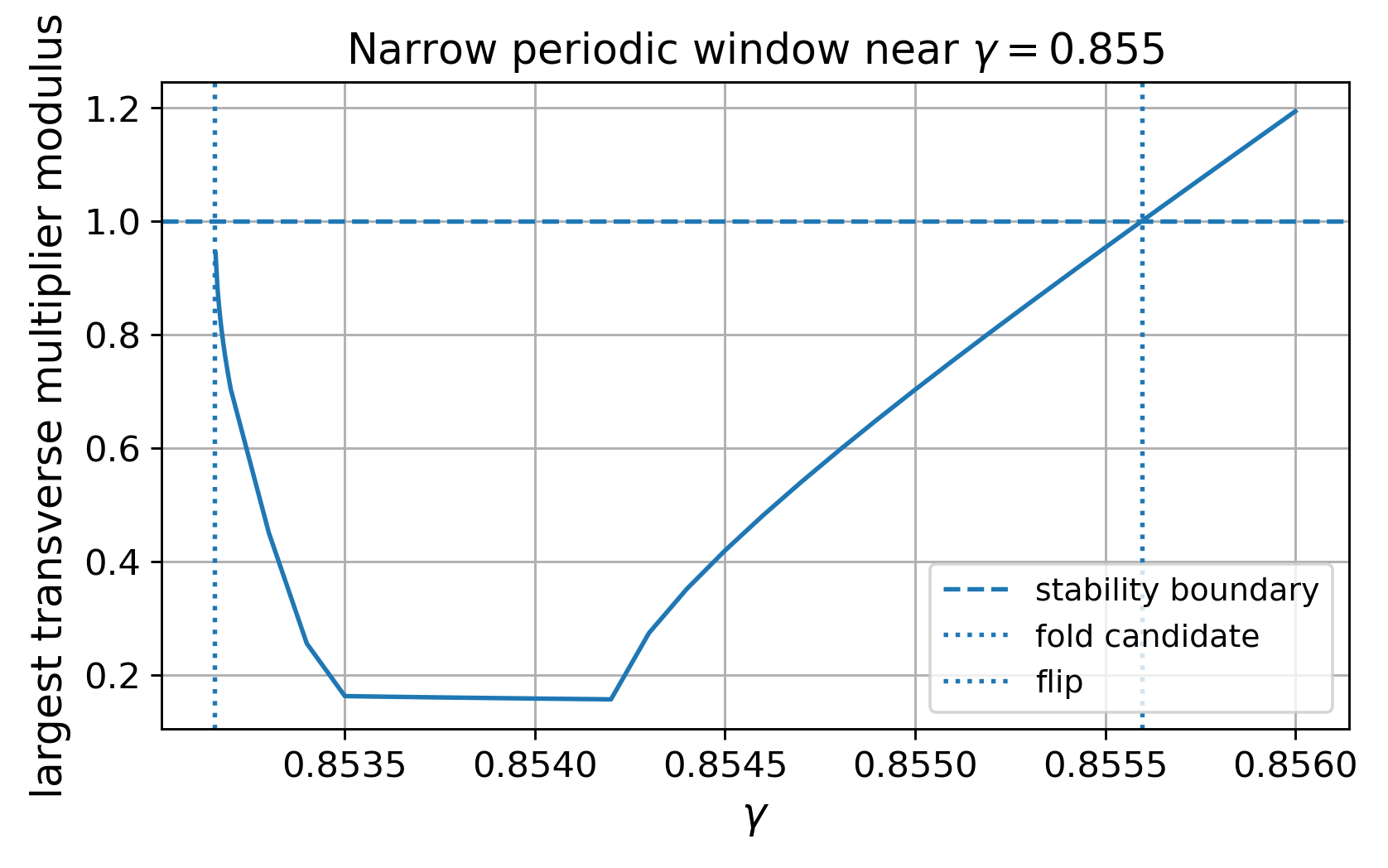}
\caption{Transverse stability near $\gamma=0.855$.}
\label{fig:narrow-window}
\end{subfigure}
\hfill
\begin{subfigure}[t]{.48\textwidth}
\centering
\includegraphics[width=\textwidth]{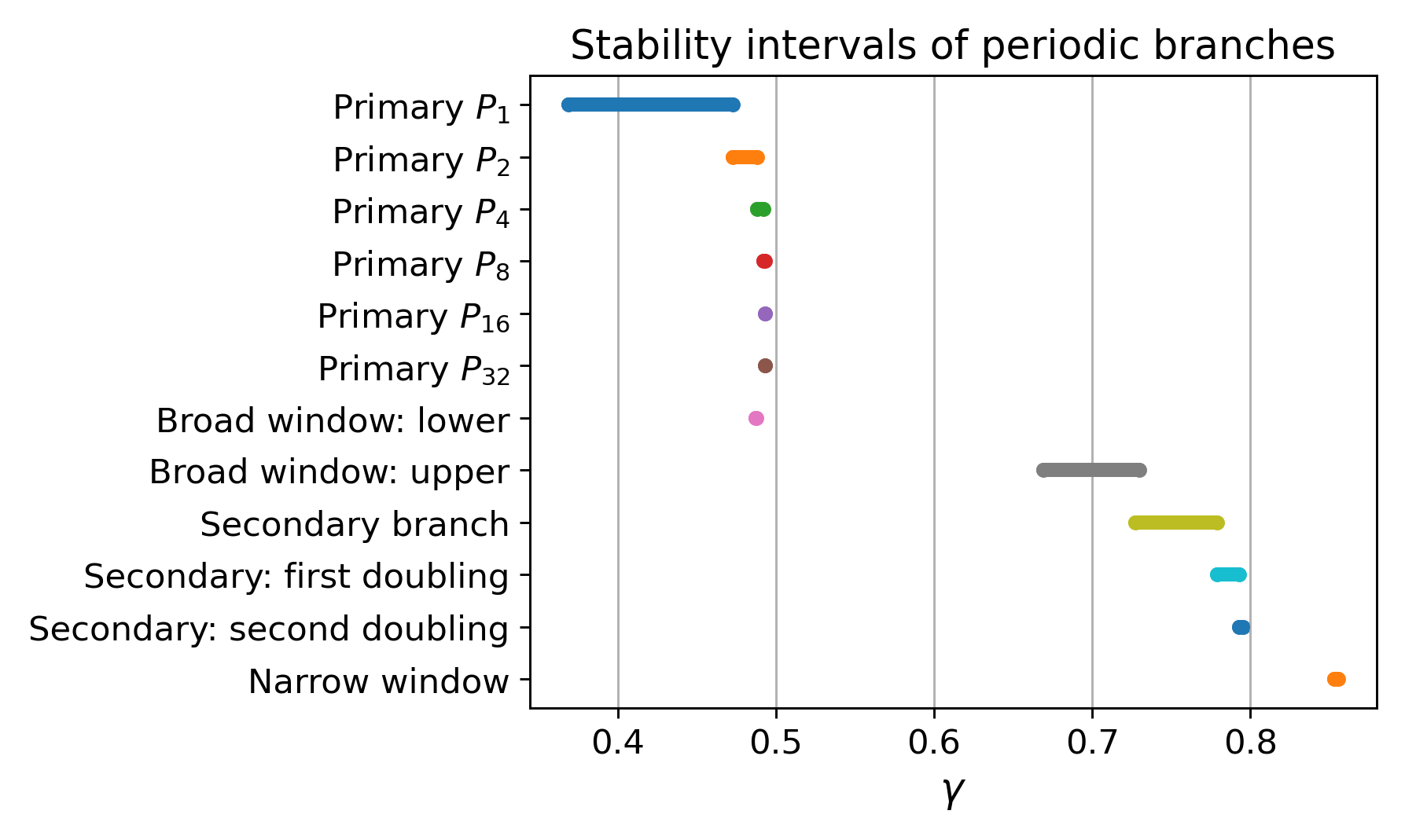}
\caption{Stability intervals of the periodic branches and their overlap.}
\label{fig:branch-atlas}
\end{subfigure}
\caption{Numerically computed periodic windows. Where the stability intervals in (B) overlap, two stable
periodic orbits at the same value of $\gamma$ are compared in \cref{tab:riccati-bistability}; the overlap alone does not
determine the basins of attraction.
}
\end{figure}

Table~\ref{tab:riccati-bistability} gives numerical evidence of bistability at two fixed  values of $\gamma$.
At each value, two distinct periodic orbits are obtained. Independent adaptive integration from their computed initial points   $p(0)$ gives $\|p(T)-p(0)\|$, $T$ the period,
between $2.2\times10^{-12}$ and $3.1\times10^{-10}$,
 and all transverse multipliers have
modulus less than $1$.
The different periods and trajectories support coexistence of distinct stable cycles. These computations provide numerical, rather than rigorously
verified,
evidence for their existence and stability; the basins of attraction are not determined.

\begin{table}[t]
\centering
\caption{Coexisting periodic orbits at fixed $\gamma$. The quantity $\rho_\perp$ is the largest modulus of the four transverse Floquet multipliers; the neutral flow multiplier is excluded.}
\label{tab:riccati-bistability}
{\small
\begin{tabular}{llrr}
\toprule
$\gamma$ & branch & period $T$ & $\rho_\perp$\\
\midrule
0.4875 & Primary $P_2$ & 6.46449988 & 0.89297049\\
0.4875 & Broad-window branch & 7.04181447 & 0.52867506\\
0.7280 & Broad-window branch & 6.04808795 & 0.74638201\\
0.7280 & Secondary branch & 5.77391290 & 0.79556517\\
\bottomrule
\end{tabular}
}
\end{table}

\subsection{Two examples: $\gamma=0.6$ and  $\gamma=0.7$} Let first $\gamma=0.6$.
After a transient has been discarded (Appendix~\ref{app:numerics}), the
solution is sampled at $200{,}001$ equally spaced times over $2{,}000$ time units and has $803$ times $t$ with $p_1(t)=0$ and $\dot p_1(t)>0$.
Independent integration of the variational equation gives values of the largest Lyapunov exponent, computed over
finite time intervals,
between $0.1766$ and $0.1864$ when the integration step and the length of the time interval
are varied as specified in Appendix~\ref{app:riccati-validation}. The exponent remains positive in all these tests, although its magnitude is not established to high precision. The trajectory, the Poincar\'e section  $\{p_1=0,\ \dot p_1>0\}$  and  the time series of the
coordinates $p_0,\dots,p_4$
provide complementary evidence of irregular dynamics.

\begin{figure}[t]
\centering
\begin{subfigure}[t]{.48\textwidth}
\centering
\includegraphics[width=\textwidth]{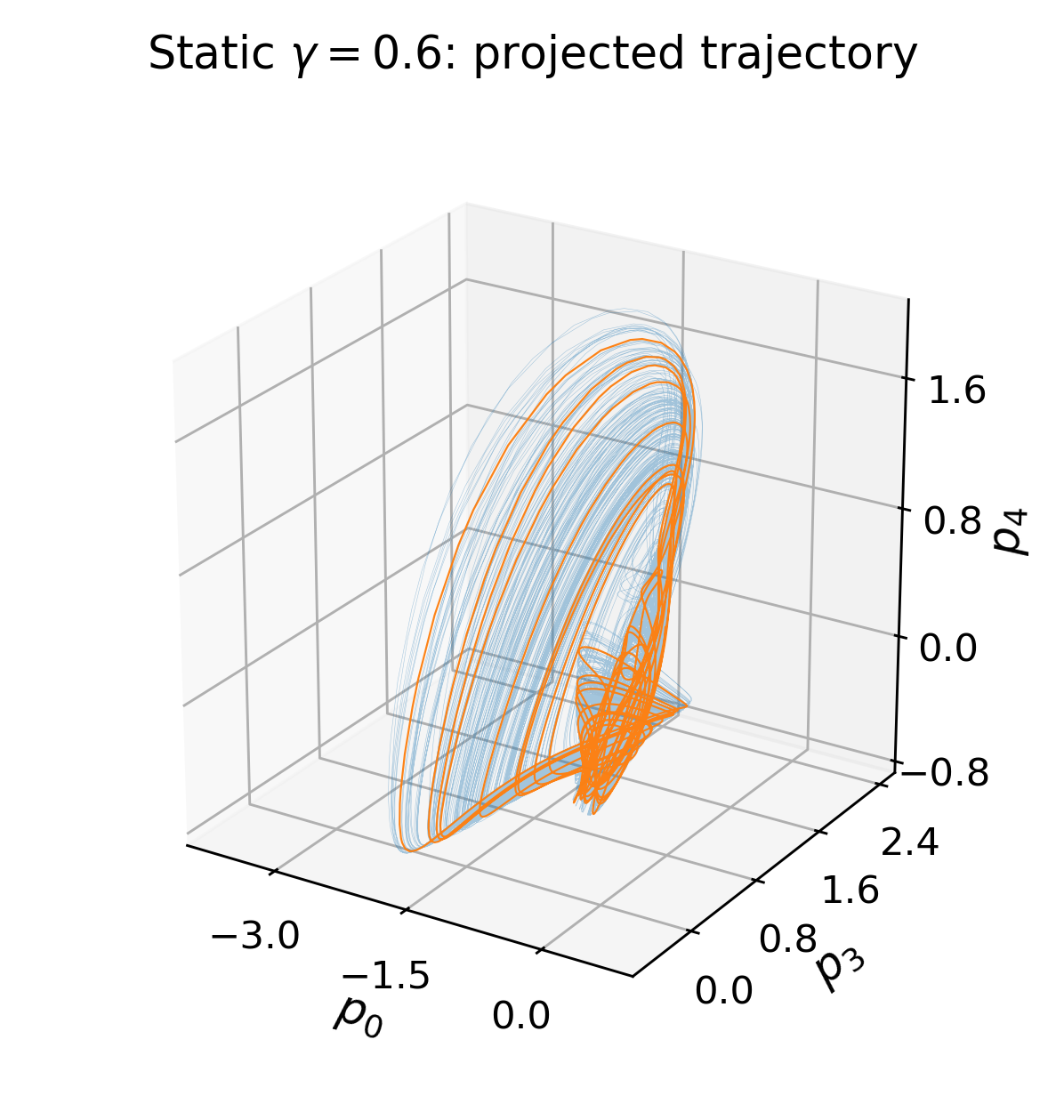}
\caption{Projection of the post-transient trajectory.}
\label{fig:riccati-chaotic-trajectory}
\end{subfigure}
\hfill
\begin{subfigure}[t]{.48\textwidth}
\centering
\includegraphics[width=\textwidth]{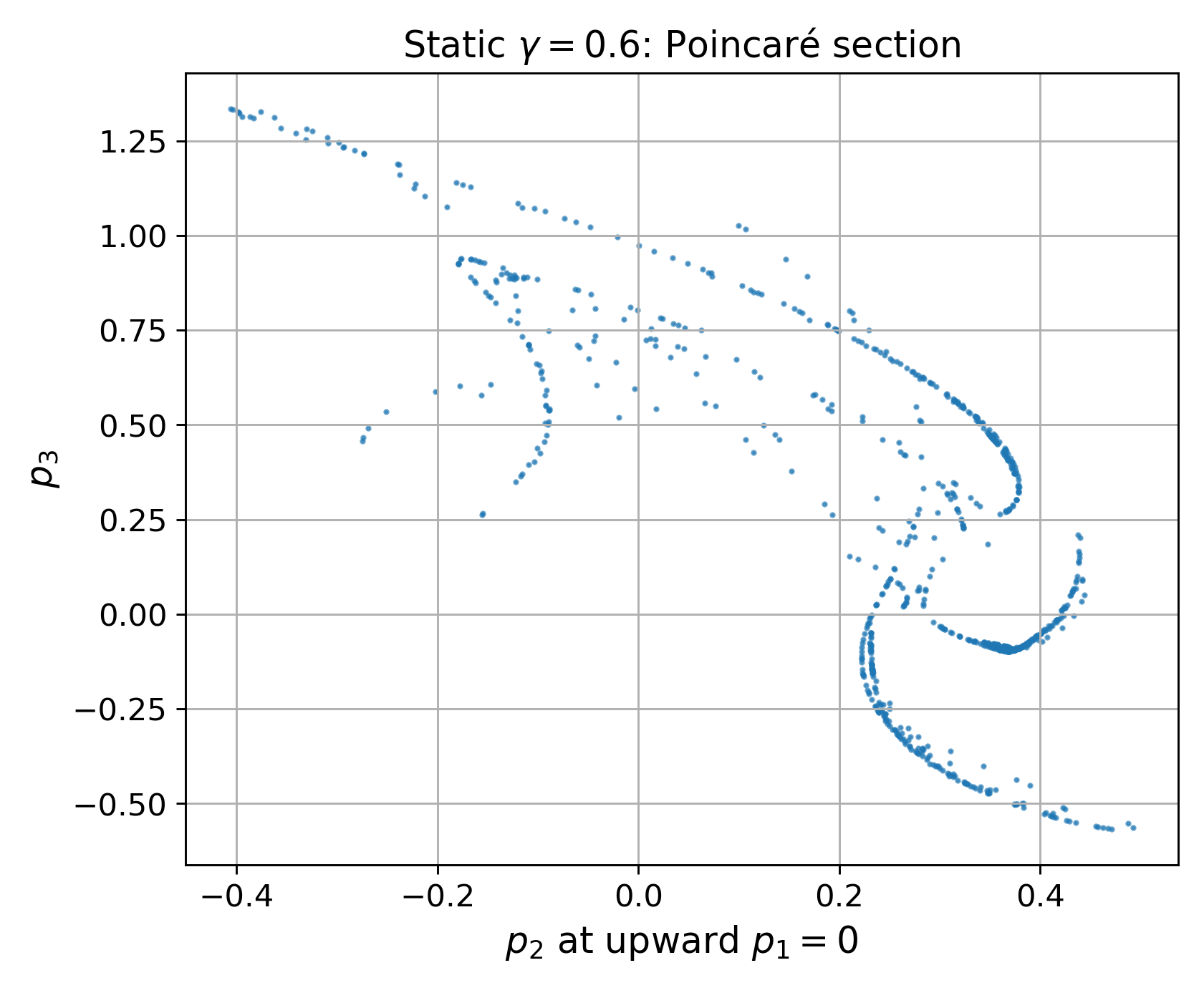}
\caption{The $803$ upward crossings of $p_1=0$.}
\label{fig:riccati-chaotic-poincare}
\end{subfigure}

\medskip
\begin{subfigure}[t]{.94\textwidth}
\centering
\includegraphics[width=\textwidth]{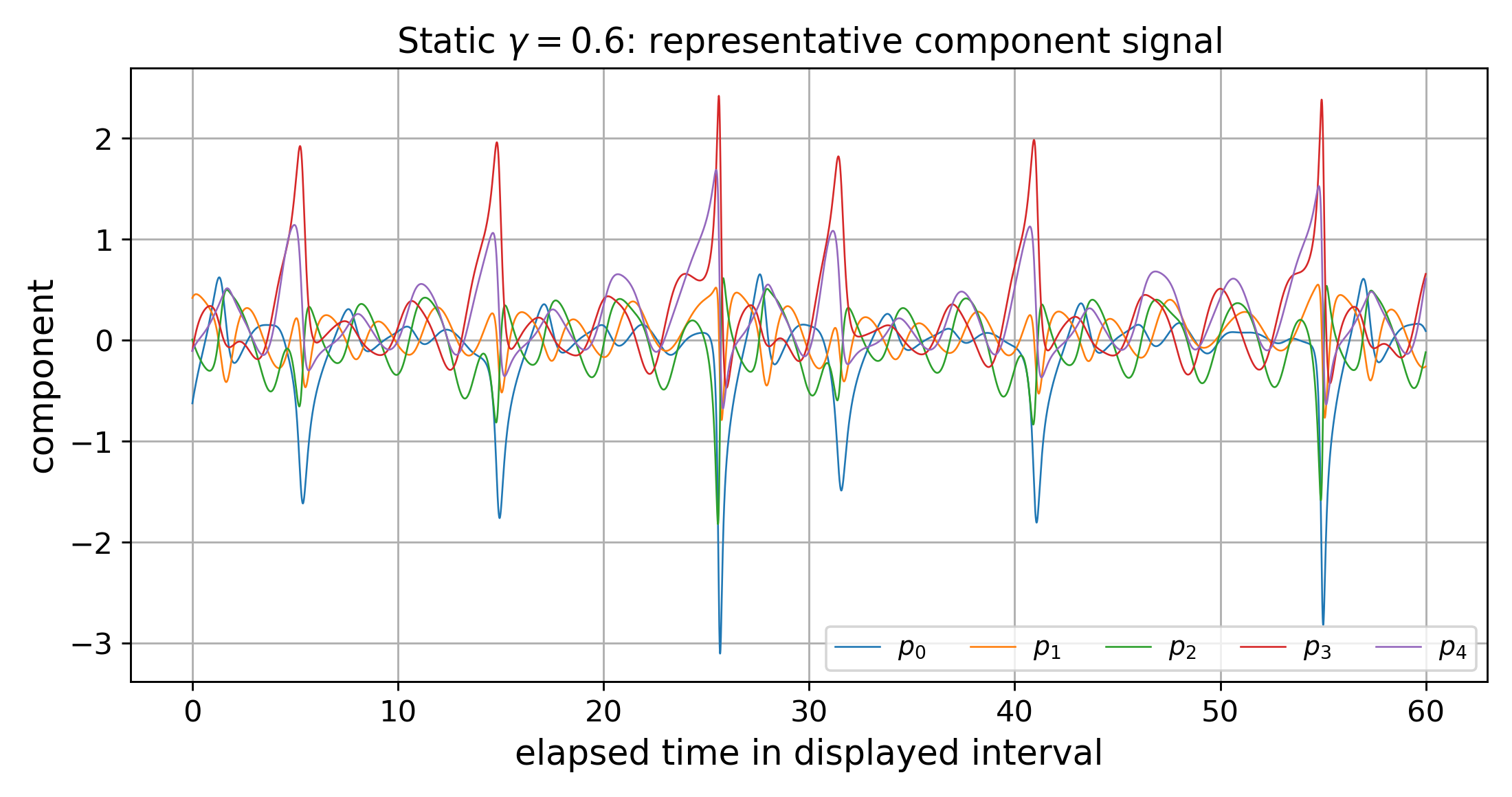}
\caption{A representative late-time component interval.}
\label{fig:riccati-chaotic-timeseries}
\end{subfigure}
\caption{Fixed $\gamma=0.6$: numerical evidence for irregular dynamics with a positive finite-time exponent. The integration intervals for the plotted solutions and the independent spectrum calculations are specified separately in Appendix~\ref{app:riccati-validation}.}
\end{figure}

For a periodic example at $\gamma=0.7$,
the
solution is sampled after a transient at $60{,}001$ times over $600$ time units and has $287$ times $t$ with $p_1(t)=0$ and $\dot p_1(t)>0$.
Independent reintegration of the periodic orbit on the broad-window branch gives period $T\approx6.260607$. One orbit period contains two local maxima of $p_3$  but three times with $p_1=0$ and $\dot p_1>0$,
explaining the three clusters in the Poincar\'e section.  The number of local maxima of $p_3$ and the
number of crossings of the Poincaré section
therefore describe different properties of the same orbit. This branch is distinct from the branch $P_2$.

\begin{figure}[t]
\centering
\begin{subfigure}[t]{.48\textwidth}
\centering
\includegraphics[width=\textwidth]{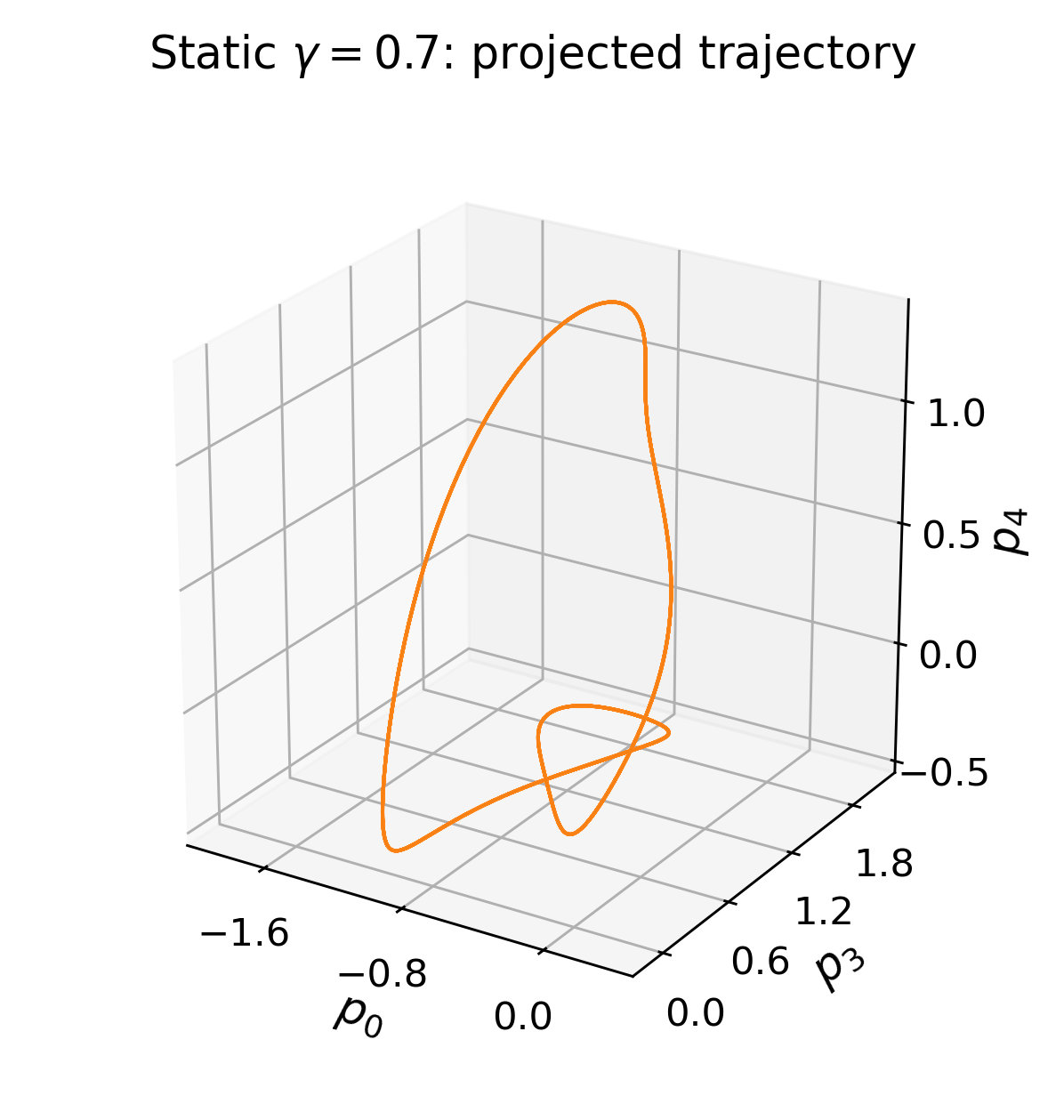}
\caption{Projected broad-window periodic orbit.}
\label{fig:riccati-regular-trajectory}
\end{subfigure}
\hfill
\begin{subfigure}[t]{.48\textwidth}
\centering
\includegraphics[width=\textwidth]{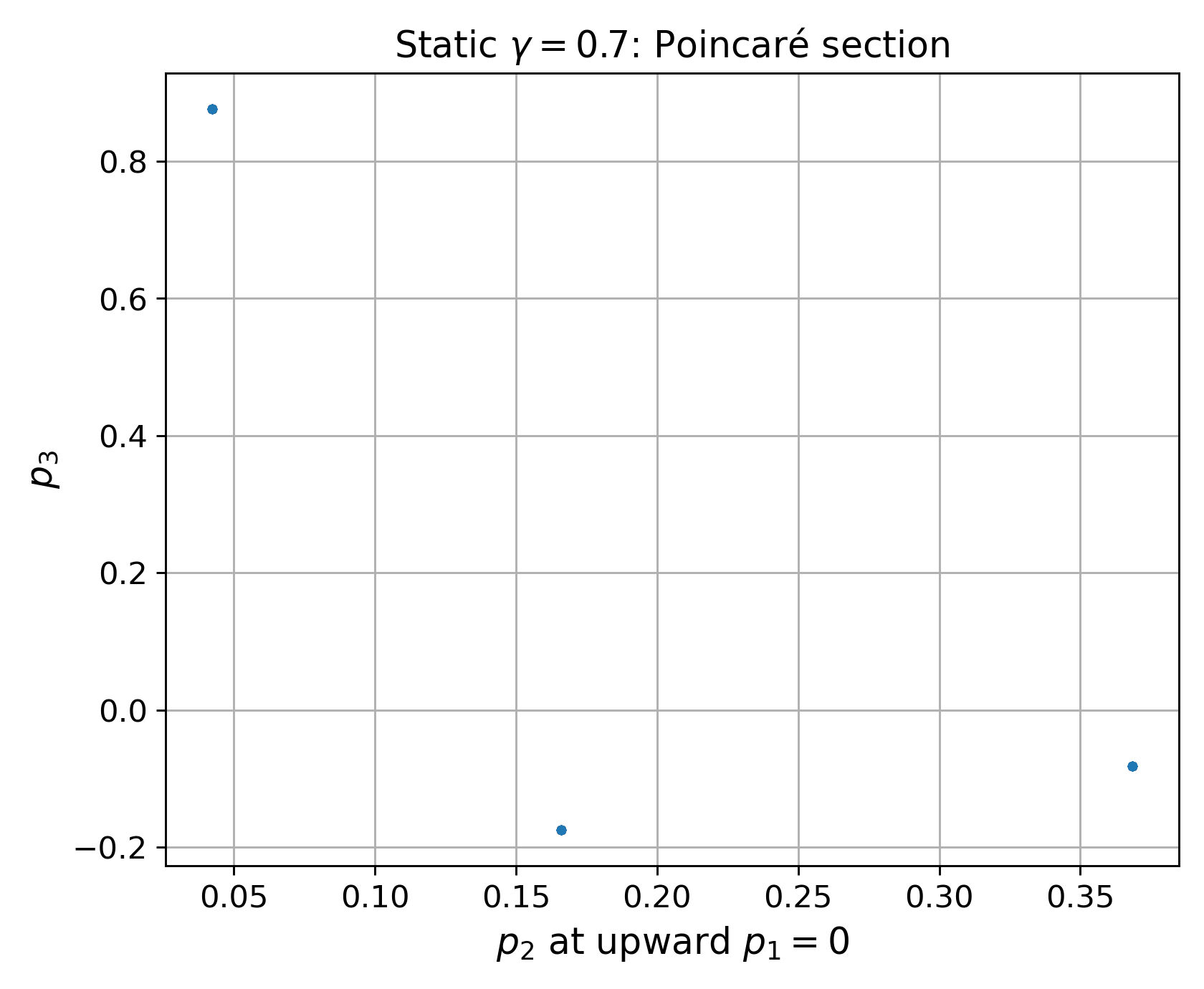}
\caption{Three upward-crossing clusters per period.}
\label{fig:riccati-regular-poincare}
\end{subfigure}

\medskip
\begin{subfigure}[t]{.94\textwidth}
\centering
\includegraphics[width=\textwidth]{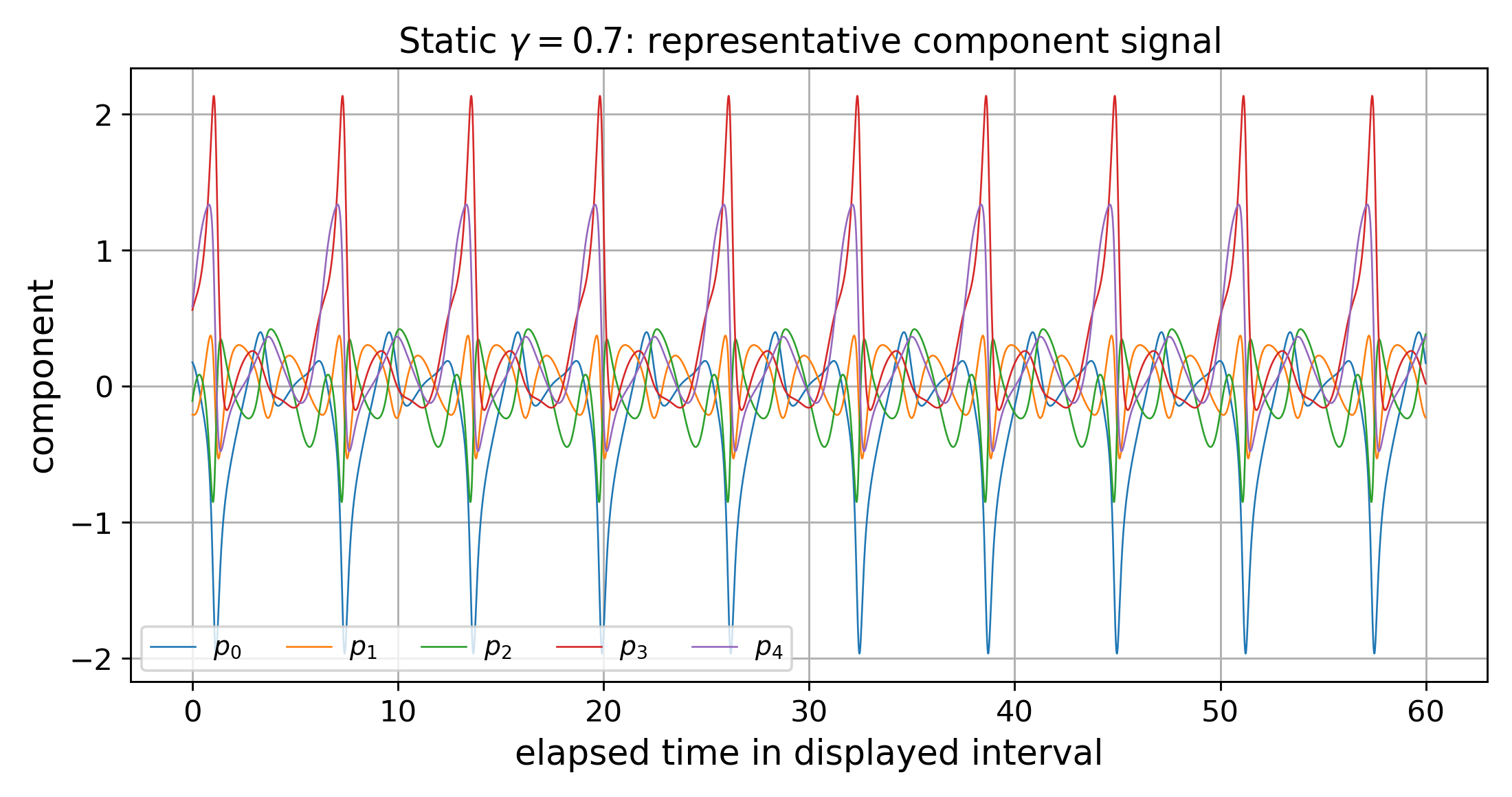}
\caption{A representative late-time component interval.}
\label{fig:riccati-regular-timeseries}
\end{subfigure}
\caption{Fixed $\gamma=0.7$: the three-point crossing pattern is compatible with two local maxima of $p_3$ per orbit period. Counts depend on the observable and section used.}
\end{figure}

\subsection{Periodically varying $\gamma$
}

For prescribed $\gamma(t)$ we consider $\dot p=F_{\gamma(t)}(p)$. Computations for
several frequencies $\omega$
use the time-dependent reference coordinates $p=S_{\gamma(t)}q$, with $S_\gamma=\operatorname{diag}(1,1,1,1,\gamma)$ from \cref{eq:weighted-state-main}. Since both protocols below remain strictly positive, this is an invertible coordinate change, not a different model. Differentiation gives

\begin{equation}
 \dot q=S_\gamma^{-1}F_\gamma(S_\gamma q)
 -S_\gamma^{-1}\dot S_\gamma q,
 \label{eq:hidden-coordinate}
\end{equation}

The fifth reference coordinate therefore contains $-(\dot\gamma/\gamma)q_4$. In this subsection, plotted coordinates are those of $p$,
in particular $p_4=\gamma q_4$. The  initial  value
is $q(0)=\tfrac12(1,-1,-1,1,-1)$; the actual initial state is $S_{\gamma(0)}q(0)$. Since $S_{\gamma(t)}$ and $S_{\gamma(t)}^{-1}$ are bounded, this coordinate change does not
change Lyapunov exponents
but finite-time estimates still depend on the coordinates and on
the length of the time interval.

For $\gamma(t)$ periodic with period $T_f=2\pi/\omega$, we consider the stroboscopic map
$p(t_0)\mapsto p(t_0+T_f)$. Unlike the monodromy matrix of a periodic orbit of an autonomous system,
its derivative has no eigenvalue 1 forced by the direction of the flow.
Consequently, an attracting $T_f$-periodic solution can have a negative largest
Lyapunov exponent.

{\bf Protocol I:  $0.60\le\gamma(t)\le0.76$.}
The first protocol is
\begin{equation}
 \gamma_{\mathrm I}(t)=0.68+0.08\sin(\omega t),
 \qquad 0.60\leq\gamma_{\mathrm I}(t)\leq0.76.
 \label{eq:protocol-I}
\end{equation}

It repeatedly passes through values of $\gamma$ at which the dynamics for fixed $\gamma$ is irregular, such as $\gamma=0.6$, and periodic, such as $\gamma=0.7$.
The estimated largest Lyapunov exponent is positive at the sampled frequencies up to $\omega=0.8$,
attaining approximately $0.1521$ there, and negative at $\omega=1$ and $1.5$. The value at $\omega=2$ is
small and negative. Thus the sampled estimates first increase and then decrease with frequency.
Negative estimates suggest contraction over the time interval used; they do not alone
prove that the solution becomes $T_f$-periodic.

{\bf Protocol II:  $0.05\le\gamma(t)\le1$.}
The second protocol is
\begin{equation}
 \gamma_{\mathrm{II}}(t)=0.525+0.475\sin(\omega t),
 \qquad 0.05\leq\gamma_{\mathrm{II}}(t)\leq1.
 \label{eq:protocol-II}
\end{equation}
It passes through nearly all of $[0,1]$ while keeping $\gamma(t)>0$, so that $S_{\gamma(t)}$
stays invertible.
The computed time series show  time intervals of small amplitude,
interrupted by short bursts of large amplitude, and more regular behaviour
at selected frequencies. A nearly zero finite-time exponent, such as that obtained at $\omega=1.5$, does not by itself determine the type of   attractor.  The fixed-step integration
failed in the fixed-step scan at $\omega=2$; this case is omitted from the exponent plot and examined separately using adaptive integration below.

\begin{table}[t]
\centering
\caption{Largest finite-time Lyapunov exponent estimates for the two protocols, obtained by rescaling nearby trajectories in the reference coordinates. The accumulation interval is $800$ time units after a transient of $400$ time units. No exponent is reported for Protocol II at $\omega=2$ (---), where the fixed-step integration failed.}
\label{tab:protocol-lle}
\small
\begin{tabular}{rcc}
\toprule
$\omega$ & Protocol I & Protocol II\\
\midrule
$0.01$ & $0.05557$ & $-0.06437$\\
$0.10$ & $0.06833$ & $0.02262$\\
$0.20$ & $0.07486$ & $0.01802$\\
$0.40$ & $0.11225$ & $0.05239$\\
$0.60$ & $0.13385$ & $0.07206$\\
$0.80$ & $0.15211$ & $0.00574$\\
$1.00$ & $-0.01554$ & $0.06493$\\
$1.50$ & $-0.01106$ & $-2.3\times10^{-6}$\\
$2.00$ & $-0.00195$ & ---\\
\bottomrule
\end{tabular}
\end{table}

At $\omega=0.01$, the forcing period
is approximately $628.32$, so the   averaging time
used in \cref{tab:protocol-lle} spans only $1.27$ forcing periods and the discarded transient only $0.64$ periods. These lowest-frequency estimates do not establish converged long-time behaviour: longer
averaging times that are multiples of the forcing period
are required. All entries in the table are computed by the same method (two nearby trajectories,
Appendix~\ref{app:numerics}) over the same time interval.

\begin{figure}[t]
\centering
\begin{subfigure}[t]{.94\textwidth}
\centering
\includegraphics[width=\textwidth]{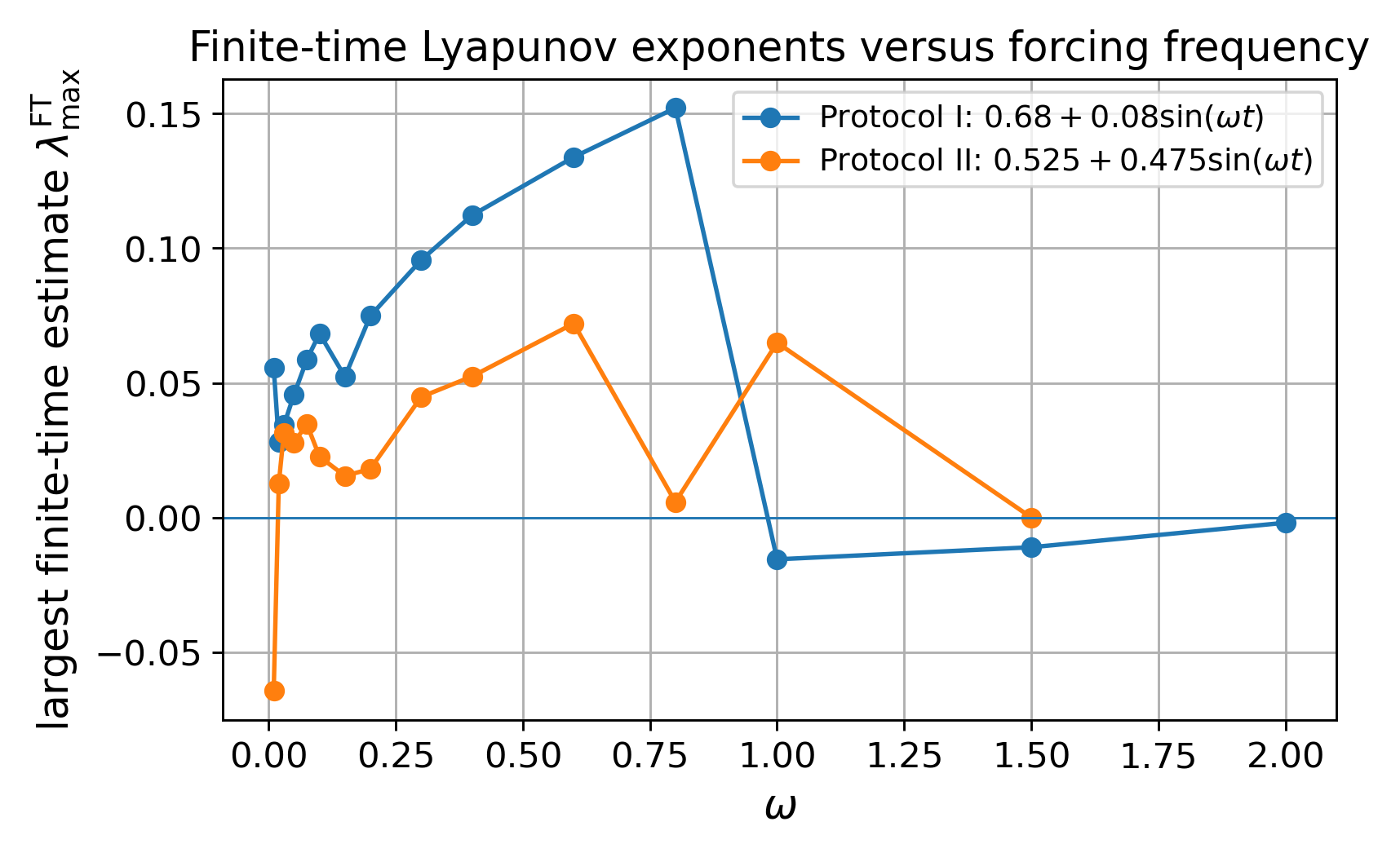}
\caption{Largest finite-time Lyapunov exponent estimates for both protocols.}
\label{fig:both-protocol-scans}
\end{subfigure}
\caption{Frequency dependence of the finite-time Lyapunov exponent estimates over a fixed time interval. Protocol II at $\omega=2$ is omitted because its fixed-step integration failed; the adaptive calculation for this case is presented separately.}
\end{figure}

For Protocol I, the time series at $\omega=0.2$ and $0.8$ illustrate the different
behaviour at these two frequencies.
A longer simulation at $\omega=0.8$ gives $230{,}001$ sampled states over $1{,}150$ time units and $147$ stroboscopic points at $\omega t=0\pmod{2\pi}$. Each stroboscopic point is computed by a short integration from the preceding sampled state, as described in Appendix~\ref{app:riccati-validation}. These points describe the same solution rather than an independent long-time convergence test.

\begin{figure}[t]
\centering
\begin{subfigure}[t]{.94\textwidth}
\centering
\includegraphics[width=\textwidth]{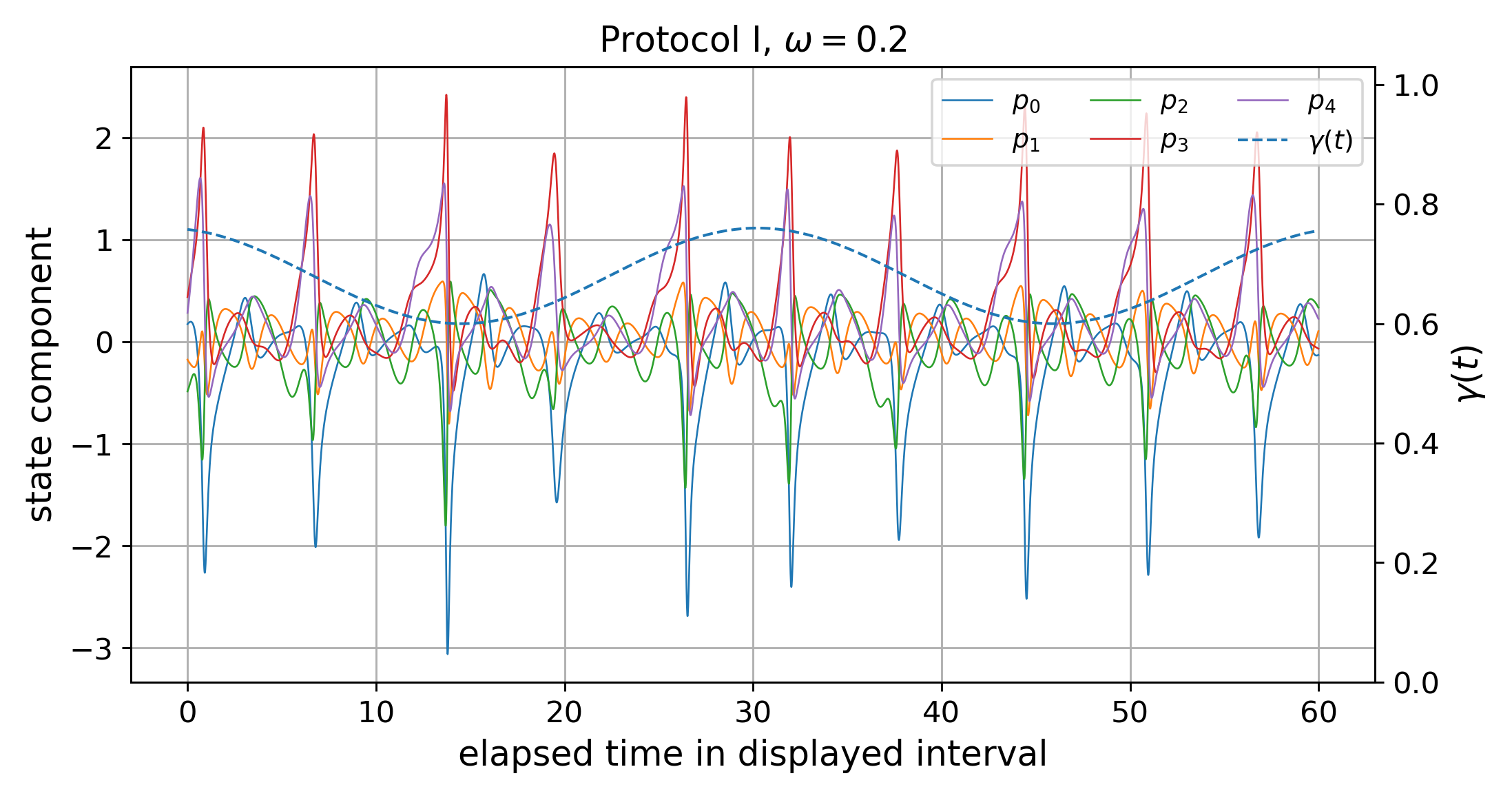}
\caption{Protocol I at $\omega=0.2$.}
\label{fig:protocol-I-02}
\end{subfigure}

\medskip
\begin{subfigure}[t]{.94\textwidth}
\centering
\includegraphics[width=\textwidth]{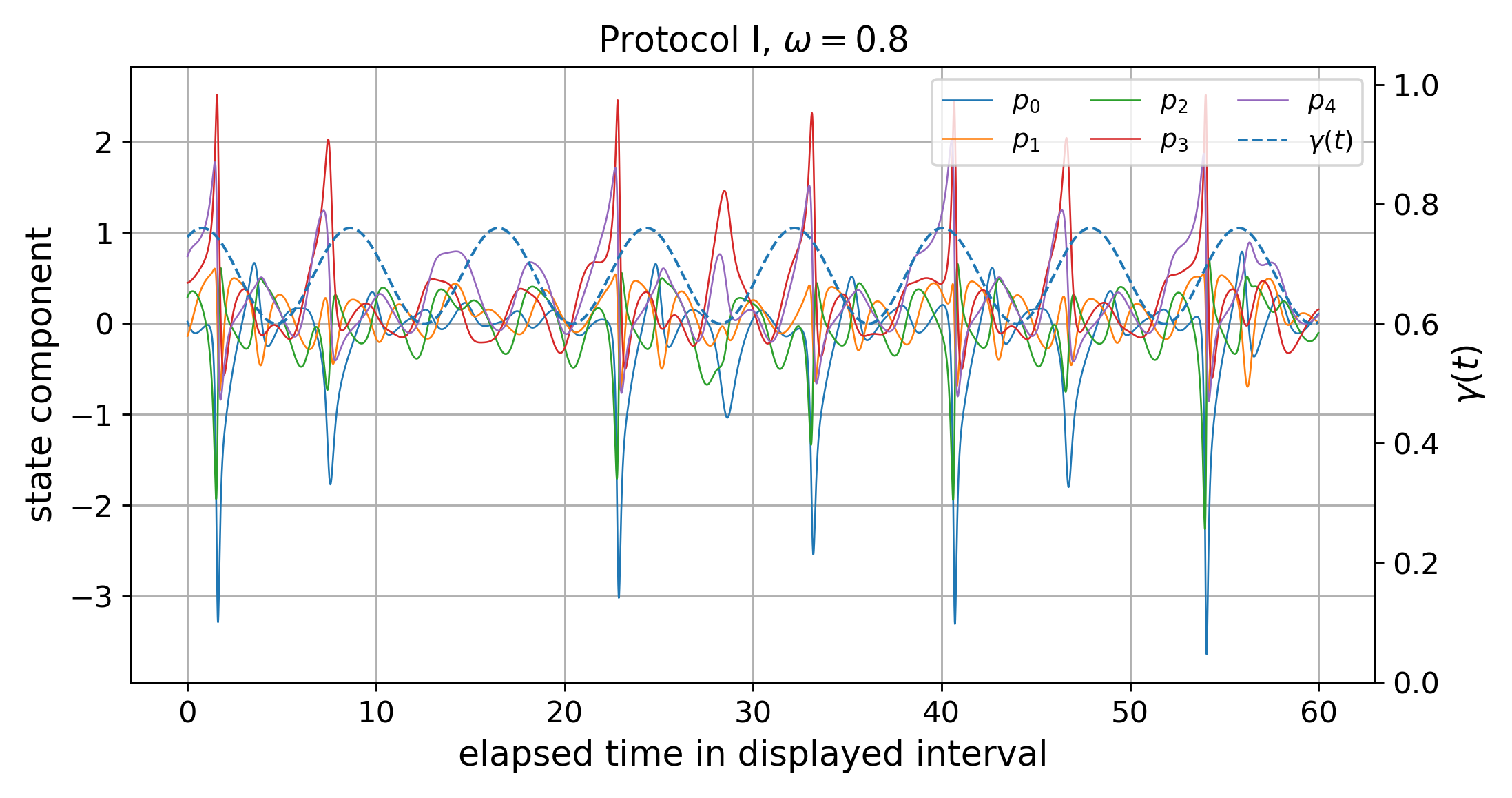}
\caption{Protocol I at $\omega=0.8$.}
\label{fig:protocol-I-08}
\end{subfigure}
\caption{Protocol I time series. The dashed curve and right-hand axis show $\gamma(t)$, allowing direct comparison of the state response with the forcing phase.}
\end{figure}

\begin{figure}[t]
\centering
\begin{subfigure}[t]{.48\textwidth}
\centering
\includegraphics[width=\textwidth]{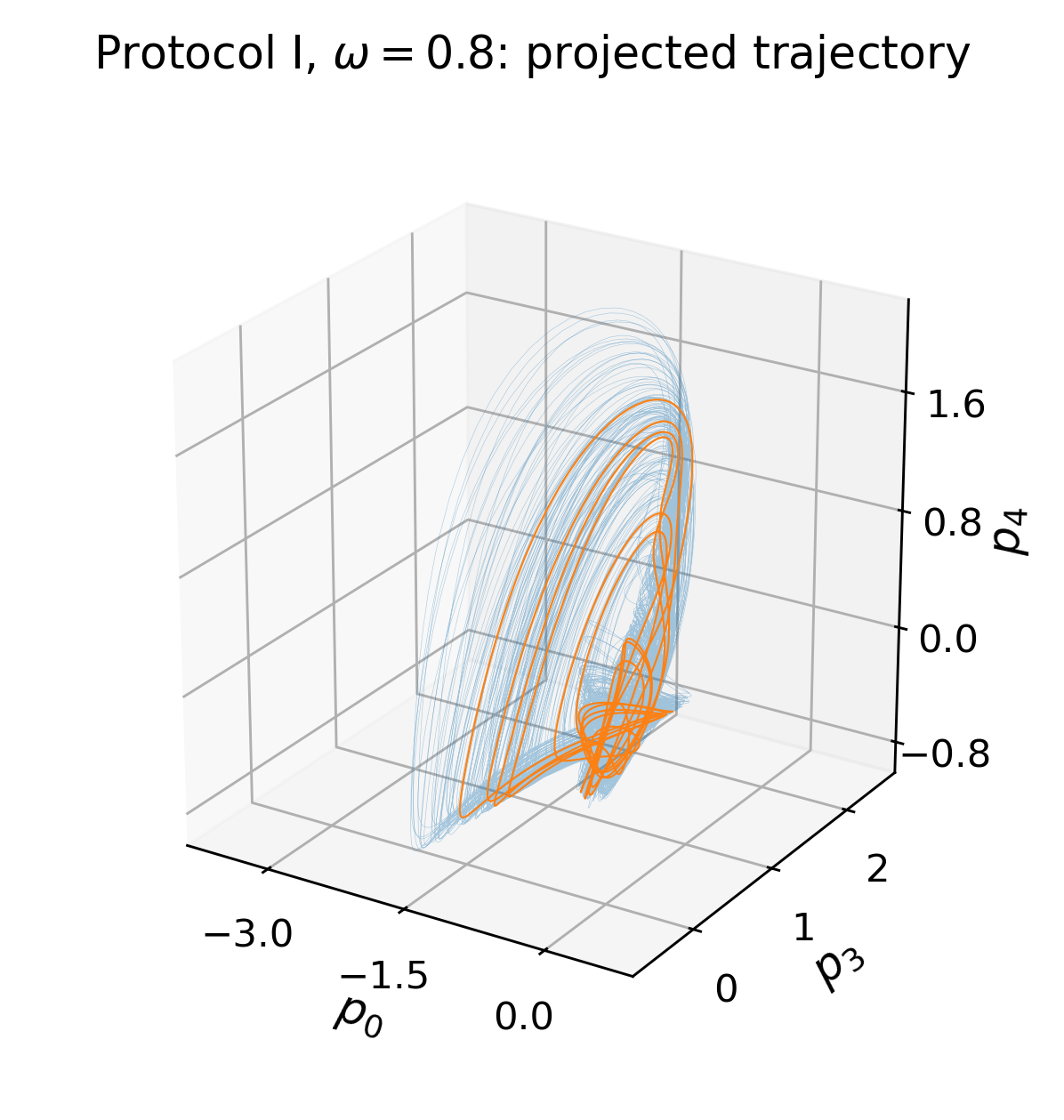}
\caption{Projection of the post-transient trajectory.}
\label{fig:protocol-I-dense-trajectory}
\end{subfigure}
\hfill
\begin{subfigure}[t]{.48\textwidth}
\centering
\includegraphics[width=\textwidth]{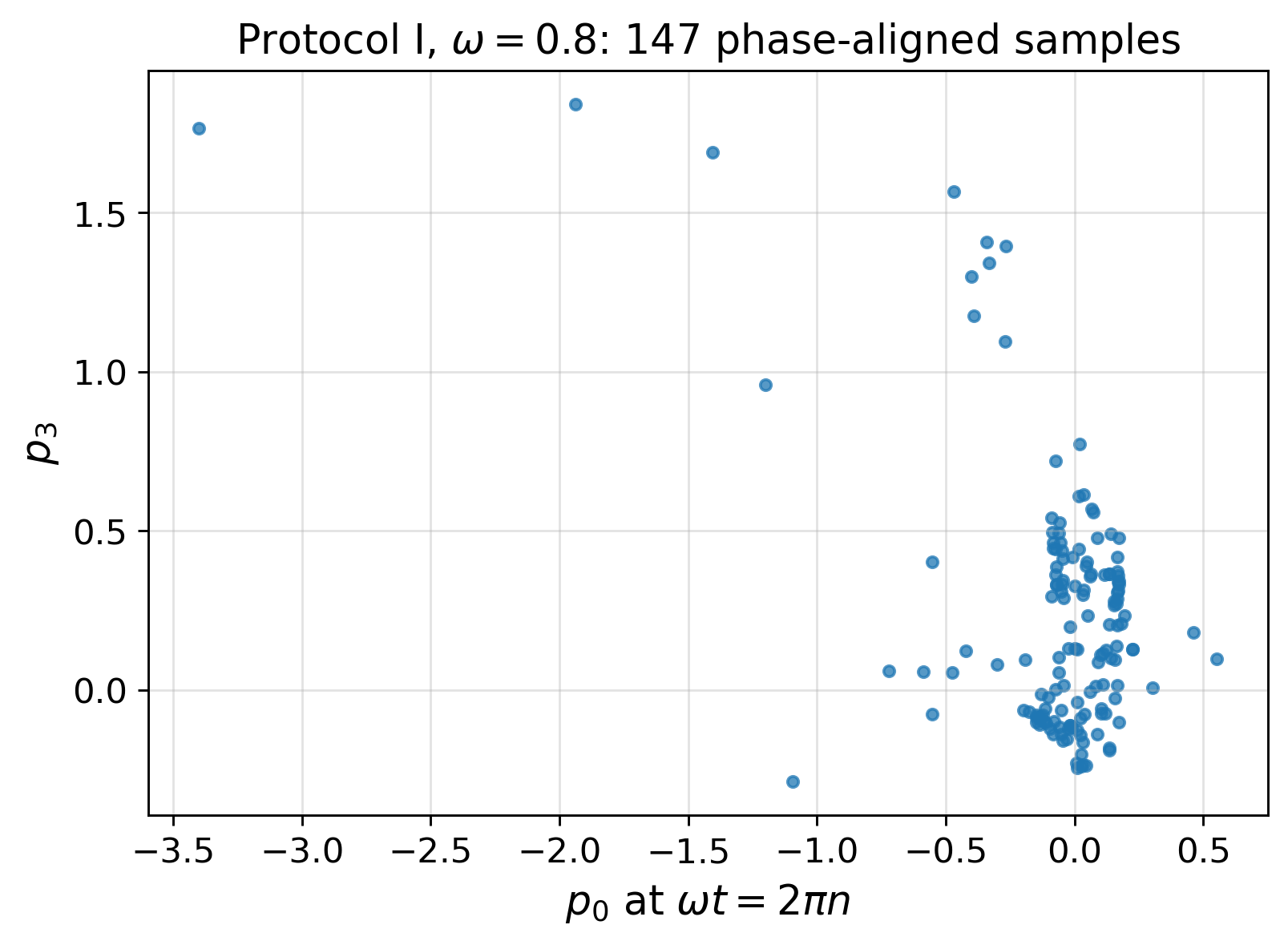}
\caption{Stroboscopic section with 147 points.}
\label{fig:protocol-I-stroboscopic}
\end{subfigure}
\caption{Protocol I at $\omega=0.8$: projected trajectory and stroboscopic section over the same finite time interval. These numerical results do not constitute a certified invariant set.}
\end{figure}

For Protocol II, the time series at $\omega=0.2$ contains short bursts of large amplitude separated by time intervals of small amplitude.
The solution
at $\omega=0.8$ is more regular
and has a much smaller positive exponent estimate in the frequency scan. A longer simulation at $\omega=0.2$ gives $230{,}001$ sampled states over $1{,}150$ time units and $36$ stroboscopic points. The small number of forcing periods limits the conclusions that can be drawn from this section.

\begin{figure}[t]
\centering
\begin{subfigure}[t]{.94\textwidth}
\centering
\includegraphics[width=\textwidth]{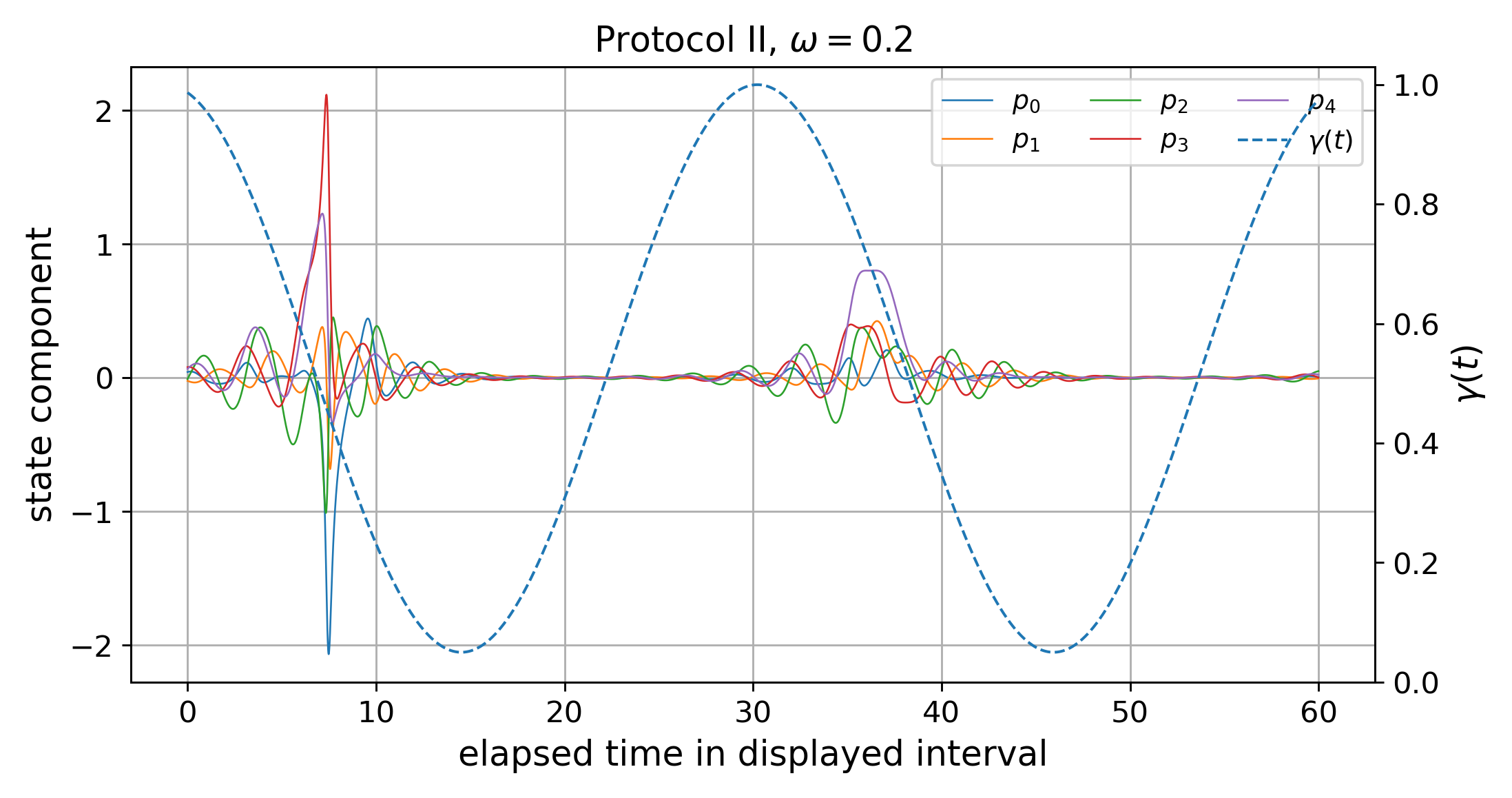}
\caption{Protocol II at $\omega=0.2$: intermittent bursts.}
\label{fig:protocol-II-02}
\end{subfigure}

\medskip
\begin{subfigure}[t]{.94\textwidth}
\centering
\includegraphics[width=\textwidth]{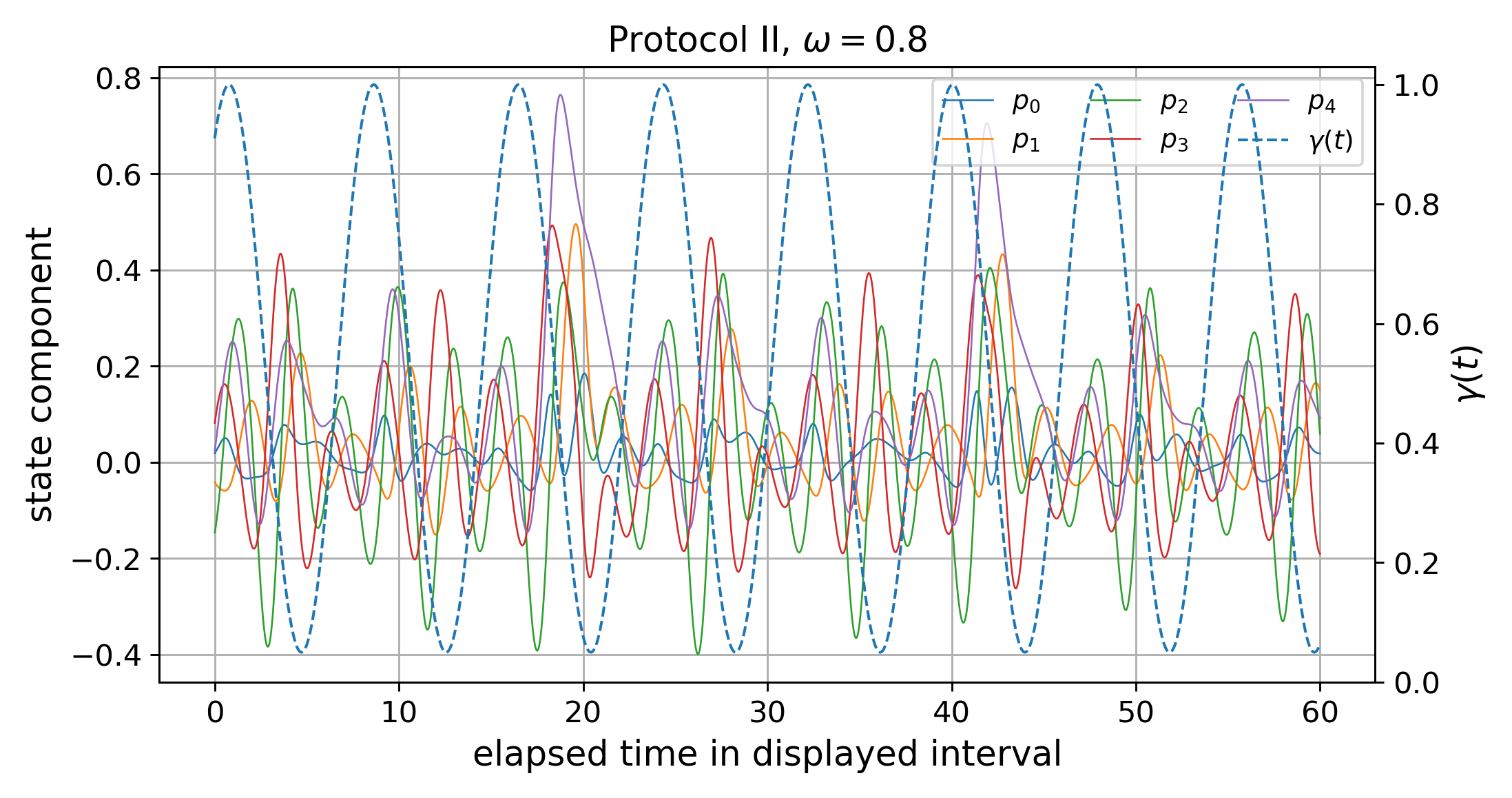}
\caption{Protocol II at $\omega=0.8$.}
\label{fig:protocol-II-08}
\end{subfigure}
\caption{Protocol II time series with the forcing $\gamma(t)$ shown on the right-hand axis. The curves show the actual state coefficients; the fifth component is $p_4$, not the reference coordinate $q_4$.}
\end{figure}

\begin{figure}[t]
\centering
\begin{subfigure}[t]{.48\textwidth}
\centering
\includegraphics[width=\textwidth]{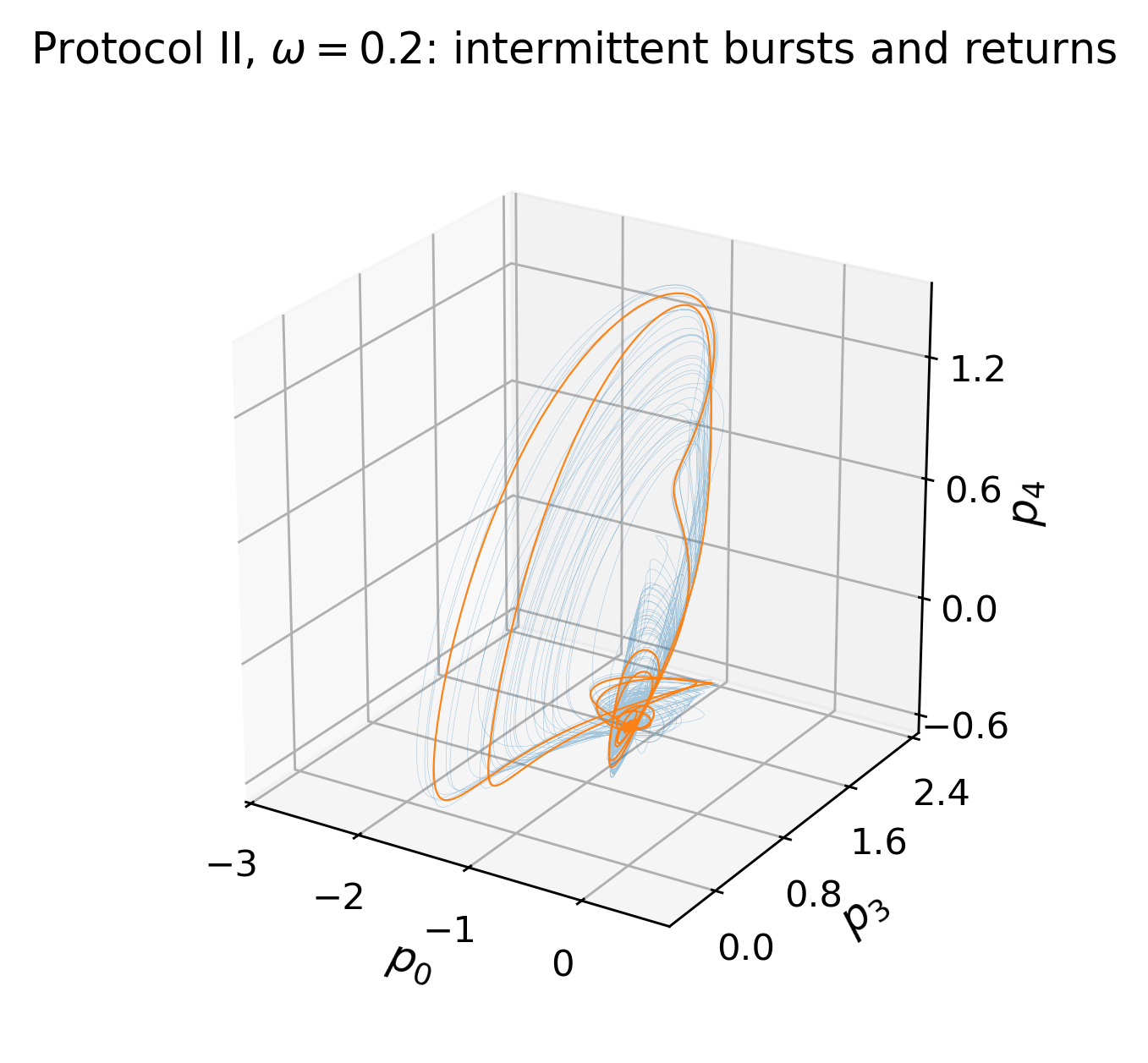}
\caption{Projected trajectory showing intermittent bursts and returns.}
\label{fig:protocol-II-dense-trajectory}
\end{subfigure}
\hfill
\begin{subfigure}[t]{.48\textwidth}
\centering
\includegraphics[width=\textwidth]{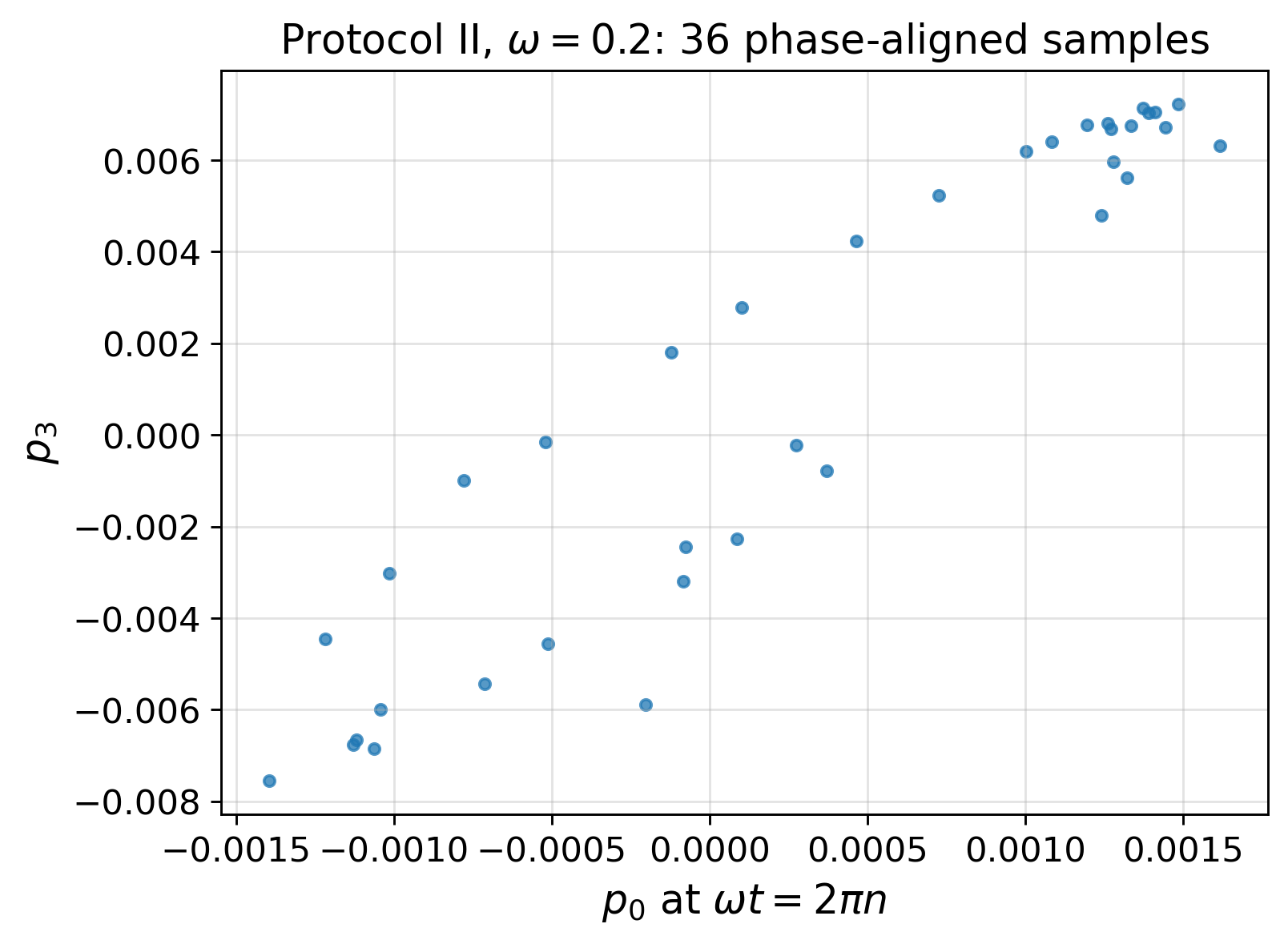}
\caption{Stroboscopic section with 36 points.}
\label{fig:protocol-II-stroboscopic}
\end{subfigure}
\caption{Protocol II at $\omega=0.2$: projected trajectory and stroboscopic section over the same time interval.}
\end{figure}

{\bf A periodic solution for Protocol II at $\omega=2$.}
The fixed-step integration
failed during a large-amplitude transient. Independent adaptive integrations to $t=1200$, using both actual and reference coordinates, did not reproduce this failure. All three tested combinations of coordinate representation and tolerance reached the final time and approached the same solution for large $t$.
Their transient maxima differed, so these calculations do not establish the maximal amplitude
or global boundedness for all initial conditions.

At $\omega=2$,
the forcing period is $T_f=\pi$. Integrating the state and variational equation over one further forcing period gives
\begin{equation}
 \|p(1200+T_f)-p(1200)\|_2\approx3.5\times10^{-13},
 \qquad \rho(\Phi(T_f))\approx0.454821<1
 \label{eq:protocol-II-periodic-check}
\end{equation}
where $\rho$ denotes the spectral radius. These values
provide numerical evidence of an attracting $T_f$-periodic solution
for the stated initial-value problem. The Floquet-derived value $\log(\rho)/T_f\approx-0.250781$ is reported separately from the
estimates in \cref{tab:protocol-lle}, because the two calculations use different methods.
The fixed-step integration failure therefore does not
indicate that solutions become unbounded.

\begin{figure}[t]
\centering
\begin{subfigure}[t]{.48\textwidth}
\centering
\includegraphics[width=\textwidth]{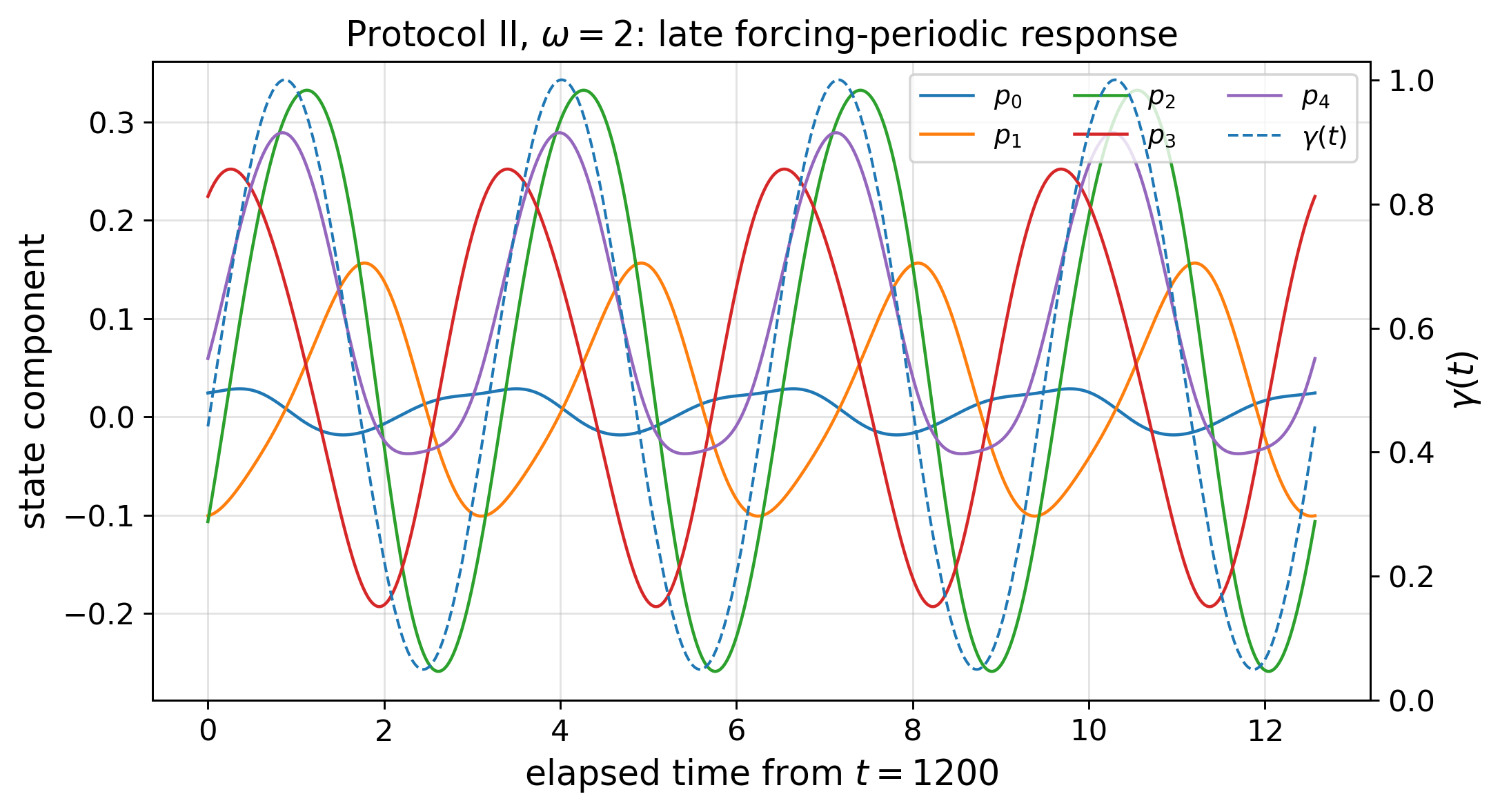}
\caption{The solution over four periods of $\gamma(t)$ at large $t$, with $\gamma(t)$ on the right-hand axis.
}
\label{fig:protocol-II-periodic-response}
\end{subfigure}
\hfill
\begin{subfigure}[t]{.48\textwidth}
\centering
\includegraphics[width=\textwidth]{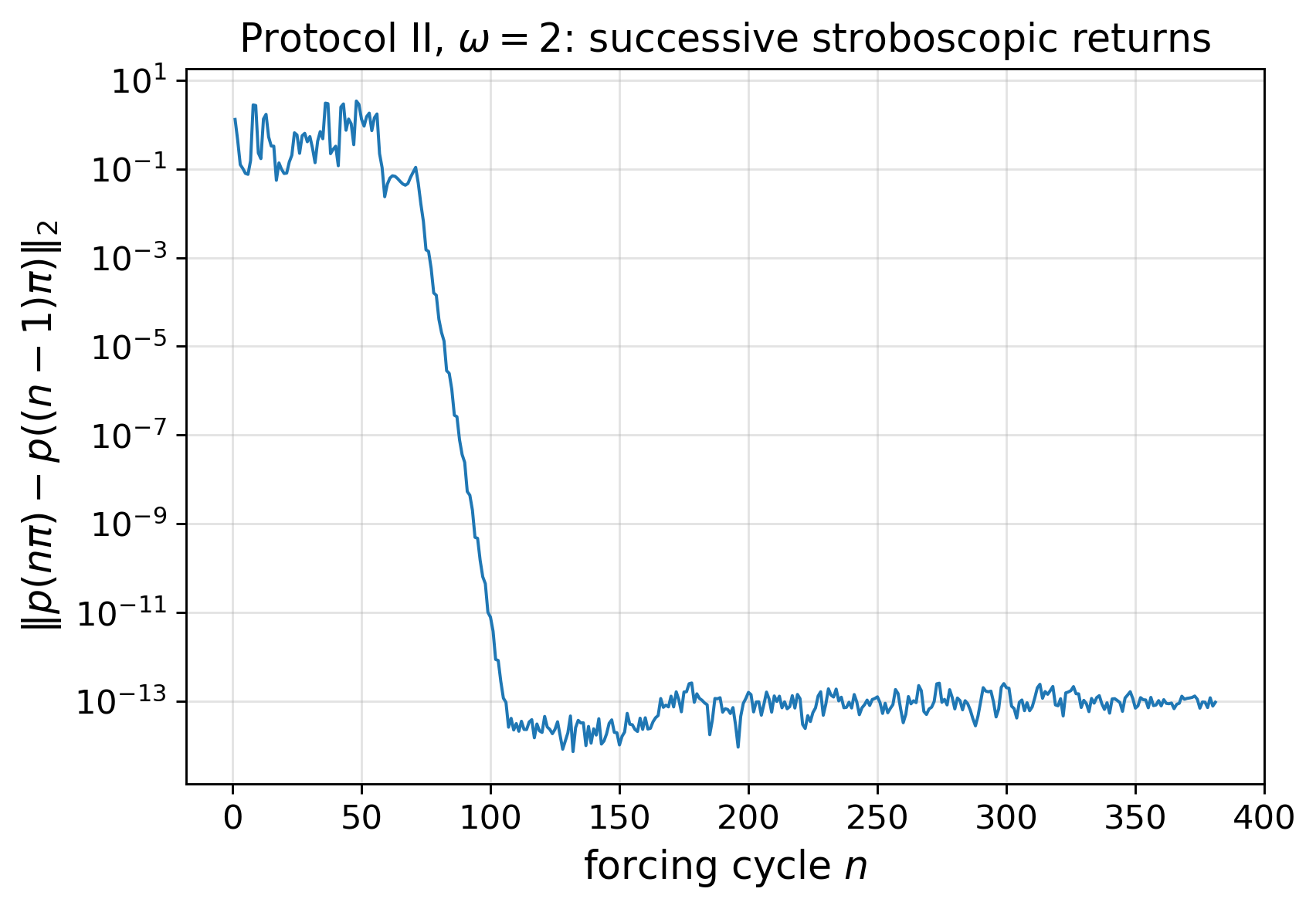}
\caption{$\|p(t+T_f)-p(t)\|$ at $\omega t\equiv0 \pmod{2\pi}$
}
\label{fig:protocol-II-return-residual}
\end{subfigure}
\caption{Protocol~II at $\omega=2$: adaptive integration indicates convergence to a $T_f$-periodic solution. The
residual in (B) eventually reaches round-off level; neither the transient nor this test is rigorously verified.
}
\end{figure}

The two protocols use different ranges of $\gamma$.
Protocol I repeatedly traverses the interval $[0.60,0.76]$,
whereas Protocol II spans nearly  all of $[0,1]$.
Their finite-time exponent estimates and component time series show different responses to the prescribed variation of $\gamma$.
Adaptive integration and the stroboscopic map
support an attracting $T_f$-periodic solution  r
for Protocol II at $\omega=2$. The other near-zero and low-frequency cases require longer convergence studies before their asymptotic behaviour can be classified.

\begin{remark}
The characteristic polynomial, uniqueness of $\gamma_H$, and transversality of the eigenvalue crossing are established analytically; the sign of the first Lyapunov coefficient is computed numerically
(Appendix~\ref{app:riccati-validation}). The first Lyapunov coefficient, periodic branches, estimated period-doubling parameters, coexisting cycles, finite-time spectra, and  solutions for
periodically varying $\gamma$
are determined numerically. Independent reintegration supports the selected periodic orbits and the  solution
at $\omega=2$, but does not certify all branches or
all frequencies.
The numerical methods, parameter settings and scope of these checks are given in Appendix~\ref{app:riccati-validation} and the accompanying code.
\end{remark}

\null
\FloatBarrier
\appendix
\section{Coordinate product, multiplication matrices and the Riccati components}
\label{app:coordinate-matrices}

Throughout this appendix,
$\SwitchAlg_\gamma=\mathcal A_\gamma(1,1,1,1)$,
$\Pent=\SwitchAlg_1$,
$\Vpure=V_\gamma=\Span_\R\{i,j,k,\ell\}$, $\kg=\frac{1}{\sqrt{1+\gamma^2}},$ and $\sg=\frac{\gamma}{\sqrt{1+\gamma^2}}$.
 When $\gamma$ is clear from the context,
we  use juxtaposition for $\star_\gamma$. Let
\[
 p=p_0+p_1i+p_2j+p_3k+p_4\ell,
 \qquad
 q=q_0+q_1i+q_2j+q_3k+q_4\ell.
\]
Then
\[
 p\star_\gamma q=C_0+C_1i+C_2j+C_3k+C_4\ell,
\]
where
\begin{align}
C_0={}&p_0q_0-p_1q_1-p_2q_2-p_3q_3-p_4q_4,\notag\\
C_1={}&p_0q_1+q_0p_1+\kg(p_2q_3-p_3q_2)
 +\frac{-p_2q_4-p_3q_4+p_4q_2+p_4q_3}{\sqrt2},\notag\\
C_2={}&p_0q_2+q_0p_2+\kg(-p_1q_3+p_3q_1)
 +\frac{-p_1q_4-p_3q_4+p_4q_1+p_4q_3}{\sqrt2},\notag\\
C_3={}&p_0q_3+q_0p_3+\kg(p_1q_2-p_2q_1)
 +\frac{-p_1q_4-p_2q_4+p_4q_1+p_4q_2}{\sqrt2},\notag\\
C_4={}&p_0q_4+q_0p_4
 +\sg(p_1q_2-p_1q_3-p_2q_1+p_2q_3+p_3q_1-p_3q_2).
\label{eq:complete-coordinate-product}
\end{align}
The coordinates $p_4$ and $q_4$ in
\cref{eq:complete-coordinate-product} are the actual coefficients
of $\ell$ in the canonical basis. If a state is specified as
$p=S_\gamma r$ using \cref{eq:weighted-state-main}, then
$p_4=\gamma r_4$, where $r_4$ is the coefficient of $\ell$ in $r$.
The map $S_\gamma$ is not part of the multiplication
$\star_\gamma$.

With respect to the canonical basis $1,i,j,k,\ell$, left multiplication is given by
\begin{equation}
L_p^{\gamma}=
\begin{pmatrix}
p_0&-p_1&-p_2&-p_3&-p_4\\
p_1&p_0&-\kg p_3+p_4/\sqrt2&\kg p_2+p_4/\sqrt2&-(p_2+p_3)/\sqrt2\\
p_2&\kg p_3+p_4/\sqrt2&p_0&-\kg p_1+p_4/\sqrt2&-(p_1+p_3)/\sqrt2\\
p_3&-\kg p_2+p_4/\sqrt2&\kg p_1+p_4/\sqrt2&p_0&-(p_1+p_2)/\sqrt2\\
p_4&\sg(-p_2+p_3)&\sg(p_1-p_3)&\sg(-p_1+p_2)&p_0
\end{pmatrix}.
\label{eq:left-multiplication-matrix}
\end{equation}
If $C=\diag(1,-1,-1,-1,-1)$, the anti-automorphic conjugation gives
\begin{equation}
R_p^{\gamma}=C L_{\bar p}^{\gamma}C.
\label{eq:right-multiplication-matrix}
\end{equation}

\subsection{Complete real Riccati system}

Write $p=w+xi+yj+zk+u\ell$, where $u$ is the actual
coefficient of $\ell$, and set
$\kappa=(1+\gamma^2)^{-1/2}$. Substituting the coefficients
$A_\gamma$ and $B_\gamma$ from
\cref{eq:A-gamma-riccati,eq:B-gamma-riccati}
into \cref{eq:riccati-main}, and using
\cref{eq:complete-coordinate-product}, gives
\begingroup\footnotesize
\begin{align}
\dot w={}&w^2-x^2-y^2-z^2-u^2
+w\left(\gamma u-4y-\frac12\right)
+2x-\frac12y+z+\frac\gamma2u,\label{eq:riccati-component-w}\\
\dot x={}&2wx-2w-\frac12x+4\kappa wz+\kappa y+\frac\kappa2z
-2\sqrt2uw+\frac{\sqrt2}{4}u
-\frac{\sqrt2\gamma}{2}w(y+z)
-\frac{\sqrt2\gamma}{4}(y+z),\label{eq:riccati-component-x}\\
\dot y={}&2w^2+2wy+\frac12w-2x^2-2y^2-\frac12y-2z^2-2u^2
+\frac{3\sqrt2}{2}u
-\frac{\sqrt2\gamma}{2}w(x+z)
-\frac{\sqrt2\gamma}{4}(x+z)
-\kappa x+2\kappa z,\label{eq:riccati-component-y}\\
\dot z={}&2wz-w-\frac12z-4\kappa wx-\frac\kappa2x-2\kappa y
-2\sqrt2uw+\frac{3\sqrt2}{4}u
-\frac{\sqrt2\gamma}{2}w(x+y)
-\frac{\sqrt2\gamma}{4}(x+y),\label{eq:riccati-component-z}\\
\dot u={}&2uw-\frac12u
+\frac{\gamma}{2}\left(u^2-w^2+x^2+y^2+z^2-w\right)\notag\\
&+\gamma\kappa\left(-4wx+4wz-\frac32x-y+\frac52z\right).
\label{eq:riccati-component-u}
\end{align}
\endgroup
The initial value is
\[
(w,x,y,z,u)(0)=\frac12(1,-1,-1,1,-\gamma).
\]
At $\gamma=0$, the final equation reduces to $\dot u=(2w-1/2)u$, so $u(0)=0$ makes the quaternion subalgebra $\mathbb H$ invariant.

\section{Numerical methods and display protocols}
\label{app:numerics}

Fixed-step calculations use fourth-order Runge--Kutta integration;
independent adaptive checks use  DOP853.
Lyapunov exponents are computed by
either nearby-trajectory rescaling or
integration of the variational equations with orthogonalisation,
as specified below. Time series of the coordinates are shown on late time intervals.
For the available code and supporting
numerical results, see Appendix~\ref{app:reproducibility}.

{\bf Riccati equation with fixed $\gamma$.
} The largest Lyapunov exponent
at $201$ values of $\gamma$
is computed using two nearby trajectories
with initial separation $10^{-8}$, rescaling every $0.1$ time unit,
integration step $\Delta t=0.01$, a transient of $300$ time units,
and an averaging time
of $500$. The
local maxima of $p_3$
are computed for
$551$ values of $\gamma$ on $[0.35,0.90]$, with
integration step $0.01$, a transient of $1000$ time units, and
$500$ retained time units; at most the last $150$ maxima are
plotted at each value of $\gamma$.

All five Lyapunov exponents at selected values of $\gamma$
are computed separately using the
analytic variational equations with step $0.005$,
orthogonalisation every $0.1$ time unit, a transient of $1000$,
and an averaging time
of $2000$.
The plotted trajectories use integration and output spacing
$0.01$, retaining $t\in[2500,4500]$ at $\gamma=0.6$ and
$t\in[1900,2500]$ at $\gamma=0.7$.

Periodic orbits are followed by shooting with a phase condition
and simultaneous variational integration. The Floquet multiplier $1$
is selected by alignment of its eigenvector with the
flow direction; stability uses the remaining four multipliers.
Continuation uses steps in $\gamma$.
 If the Newton corrector fails or the continuation stops, this
does not establish a fold or an isola.
Period-doubling parameters are estimated by secant
interpolation between adjacent sampled multipliers across $-1$;
the sample separation is reported, but no rigorous interpolation
error is assigned.

{\bf Riccati equation with periodic $\gamma(t)$.}
The  computations
use two nearby trajectories in the reference
coordinates, initial separation $10^{-8}$, rescaling every $0.1$
time unit, integration step $0.005$, a transient of $400$, and an averaging time
of $800$ time units.
Forcing phase is prescribed and is not perturbed.
These fixed intervals contain different numbers of forcing
periods. At $\omega=0.01$,  the averaging
time
spans only $1.27$ periods
and does not establish converged long-time behaviour.
The fixed-step integration for Protocol~II at $\omega=2$ failed; it
supplies no exponent estimate and is omitted from the exponent plot; its adaptive
reintegration is described below.

The  time series
at $\omega=0.2,0.8$ use integration step
$0.005$ and output spacing $0.01$, retaining $500$ time units
after a transient of $700$.
The longer Protocol I ($\omega=0.8$) and Protocol II
($\omega=0.2$) records retain $1150$ time units after a transient
of $1200$, with integration and output spacing $0.005$.  The plotted fifth coordinate is
$p_4=\gamma q_4$.

{\bf Further checks for the Riccati equation.}
\label{app:riccati-validation}
The characteristic-polynomial identity is verified by exact
symbolic arithmetic. The first Lyapunov coefficient is computed
separately in floating-point arithmetic using the following
convention.
Write $L=D_pF_{\gamma_H}(0)$ and let $\mathcal B$ be the second
derivative, so $F_{\gamma_H}(x)=Lx+\tfrac12\mathcal B(x,x)$.
With $Lq=i\omega_Hq$, $L^T\psi=-i\omega_H\psi$,
$\bar q^Tq=1$, and $\bar\psi^Tq=1$, the convention is
\[
\begin{aligned}
l_1=\frac{1}{2\omega_H}\operatorname{Re}\bigl\{\bar\psi^T[
&\mathcal B(\bar q,(2i\omega_H I-L)^{-1}\mathcal B(q,q))\\
&-2\mathcal B(q,L^{-1}\mathcal B(q,\bar q))]\bigr\}.
\end{aligned}
\]
There is no cubic derivative term because the field is quadratic.
The computed value is $l_1\approx-0.1832582814$; eigenvector and
adjoint residuals are below $3\times10^{-15}$.

The sign of $l_1$ has not been certified by interval arithmetic.

Seven computed periodic-orbit initial points, including the four
in \cref{tab:riccati-bistability}, are reintegrated together with
their fundamental matrices using DOP853, relative tolerance
$10^{-11}$, absolute tolerance $10^{-13}$, and maximum step
$0.03$.

 For $\gamma=0.6$, the Lyapunov exponents are recomputed by
a separate fourth-order
Runge--Kutta implementation that  integrates the state and full
variational matrix, with twice-applied modified Gram--Schmidt
orthogonalisation every $0.1$ time unit. All runs discard $1000$
time units.
The largest estimates are $0.186409$ for step $0.005$ and averaging time
$2000$, $0.180262$ for step $0.0025$ and averaging time
$2000$, and $0.176563$ for step $0.0025$ and
averaging time
$4000$. Their positive sign persists across these
checks, but the values have not converged to six-digit precision.
The sum of the five finite-time exponents agrees with the time average of the divergence $\operatorname{div}F_\gamma(p)=2\gamma p_4+10p_0-8p_2-\tfrac52$
to within $1.1\times10^{-8}$.

The stroboscopic samples in
\cref{fig:protocol-I-stroboscopic,fig:protocol-II-stroboscopic}
are evaluated at $t_k=2\pi k/\omega$ by DOP853 integration from
the immediately preceding stored  point of the
solution,
over intervals
shorter than $0.005$, with relative and absolute tolerances
$10^{-11}$ and $10^{-13}$.
These samples are derived from the same time series,
not independent long-time integrations.

For Protocol II at $\omega=2$, DOP853 reintegration to $t=1200$
uses maximum step $0.05$. The actual-coordinate runs use relative
and absolute tolerance pairs $(10^{-10},10^{-12})$ and
$(10^{-12},10^{-14})$; the reference-coordinate run uses
$(10^{-11},10^{-13})$.
All three runs reach the final time without triggering the  stopping condition $\|p\|=10^6$,
and their final values of $p$
agree within $4\times10^{-13}$.
The sampled transient maxima differ, approximately $9.89$,
$21.13$ and $9.88$, so these calculations do not establish the maximal amplitude.
A one-period variational check from $t=1200$, with tolerances
$(10^{-12},10^{-14})$ and maximum step $0.02$, gives the  values
and stroboscopic spectral radius reported in
\cref{eq:protocol-II-periodic-check}.

Positive finite-time exponents, bounded time series, Poincaré
sections
and numerical Floquet checks are interpreted together
as numerical evidence. None of them alone is a proof of the existence of an invariant set or of global existence of solutions.

\section{Code and data availability}
\label{app:reproducibility}\label{app:D}

The code and supporting data for Part~II are available in
\path{quaternions-to-pentarnions/part-II/} of the repository
\url{https://github.com/poshtkohi/pentarnion-algebra}.
The repository includes symbolic verification routines, numerical
simulation and continuation programs, figure-generation scripts
and supporting data for reproducing the results.

\section{Declaration of generative AI and AI-assisted technologies in the manuscript preparation process}

The presented results, protocols and diagrams are our own. Claude 5.5 was
used in the last step to assist with proofreading the manuscript and to identify possible gaps in the proofs. We
independently checked all such suggestions, and we take full responsibility for the correctness and
content of the result. We support the Leiden Declaration on Artificial Intelligence and Mathematics.

\providecommand{\bysame}{\leavevmode\hbox to3em{\hrulefill}\thinspace}

\end{document}